\documentclass[11pt]{article}

\usepackage[english]{babel}
\usepackage{amsmath,amsthm,amssymb,amsfonts}
\usepackage{bm,bbm}
\usepackage{paralist}
\usepackage{graphicx}
\usepackage{subcaption}
\usepackage{hyperref}
\usepackage[usenames,dvipsnames,svgnames,table]{xcolor}
\usepackage[symbol]{footmisc}

\numberwithin{equation}{section}
\theoremstyle{plain}
\newtheorem{Theorem}{Theorem}[section]
\newtheorem{Lemma}[Theorem]{Lemma}
\newtheorem{Proposition}[Theorem]{Proposition}

\theoremstyle{definition}
\newtheorem{Definition}[Theorem]{Definition}

\DeclareRobustCommand{\stirling}{\genfrac\{\}{0pt}{}}

\newcommand{\wh}{\widehat}
\newcommand{\wt}{\widetilde}

\providecommand{\eps}{\varepsilon}
\newcommand{\pa}{\operatorname{pa}}

\newcommand{\Poi}{\operatorname{Poisson}}

\newcommand{\mB}{\mathcal{B}}

\newcommand{\mD}{\mathcal{D}}

\newcommand{\mF}{\mathcal{F}}
\newcommand{\mG}{\mathcal{G}}
\newcommand{\mH}{\mathcal{H}}

\newcommand{\mJ}{\mathcal{J}}

\newcommand{\mM}{\mathcal{M}}
\newcommand{\mN}{\mathcal{N}}

\newcommand{\mR}{\mathcal{R}}

\newcommand{\mU}{\mathcal{U}}

\newcommand{\KDE}{\widehat{f}_{\operatorname{KDE}}}

\newcommand{\bd}{\mathbf{d}}

\renewcommand{\bf}{\mathbf{f}}

\newcommand{\bp}{\mathbf{p}}

\newcommand{\bt}{\mathbf{t}}
\newcommand{\bv}{\mathbf{v}}

\newcommand{\bu}{\mathbf{u}}

\newcommand{\bx}{\mathbf{x}}
\newcommand{\bX}{\mathbf{X}}

\newcommand{\by}{\mathbf{y}}

\newcommand{\balpha}{\bm{\alpha}}

\newcommand{\bgamma}{\bm{\gamma}}

\newcommand{\btheta}{\bm{\theta}}

\newcommand{\cmmnt}[1]{\ignorespaces}
\newcommand{\1}{\mathbbm{1}}

\newcommand{\ceil}[1]{\lceil {#1}\rceil}
\newcommand{\floor}[1]{\lfloor {#1}\rfloor}

\DeclareMathOperator*{\argmin}{arg\,min}
\DeclareMathOperator*{\argmax}{arg\,max}
\DeclareMathOperator\supp{supp}

\begin{document}

\title{A supervised deep learning method for nonparametric density estimation}
\author{%
    Thijs Bos\footnotemark[1] \footnotemark[3] %
    \ \ and Johannes Schmidt-Hieber\footnotemark[2] \footnotemark[3]%
}
  \footnotetext[1]{Leiden University} 
  \footnotetext[2]{University of Twente}
  \footnotetext[3]{We want to thank Claire Donnat for pointing us to Lindsey's method. The research has been supported by the NWO/STAR grant 613.009.034b and the NWO Vidi grant VI.Vidi.192.021.}
\date{}
\maketitle

\begin{abstract}
Nonparametric density estimation is an unsupervised learning problem. In this work we propose a two-step procedure that casts the density estimation problem in the first step into a supervised regression problem. The advantage is that we can afterwards apply supervised learning methods. Compared to the standard nonparametric regression setting, the proposed procedure creates, however, dependence among the training samples. To derive statistical risk bounds, one can therefore not rely on the well-developed theory for i.i.d.\ data. To overcome this, we prove an oracle inequality for this specific form of data dependence. As an application, it is shown that under a compositional structure assumption on the underlying density, the proposed two-step method achieves convergence rates that are faster than the standard nonparametric rates. A simulation study illustrates the finite sample performance.
\end{abstract}



\section{Introduction}

Machine learning distinguishes between supervised and unsupervised learning tasks \cite{bishop2006PRML, murphy2012machine}. In the supervised framework, the dataset consists of input-output pairs. No outputs are observed in the unsupervised setting. For supervised learning, classical examples are regression and classification; for unsupervised learning, commonly encountered problems are density estimation and clustering. The apparent difference between supervised and unsupervised tasks resulted in machine learning methods that either apply to the supervised or to the unsupervised framework. While neural nets can be applied in both scenarios, the underlying methodology is mostly unrelated: In the supervised context, deep learning is applied to reconstruct the function mapping the inputs to the outputs; in the unsupervised framework, neural networks are employed for instance in ODE-based models for density estimation \cite{grathwohl2018scalable, Onken_Wu_Fung_Li_Ruthotto_2021, marzouk2023distribution} or for feature extraction, e.g.\ by making use of variational autoencoders \cite{IntroVAEKingmaWelling}. Moreover, generative AI methods such as generative adversarial networks (GANs) or diffusion models invoke neural networks and can be viewed as density estimators \cite{dinh2017density, 2020arXiv200203938C, pmlr-v132-schreuder21a, 2022arXiv220202890C, stephanovitch2023wasserstein, pmlr-v202-oko23a}.

In this article, we show how unsupervised multivariate density estimation can be cast into a supervised regression problem. For that, we generate suitable response variables from the data in a first step. Rewriting the problem as supervised learning task allows us to borrow strength from supervised learning methods. We demonstrate this by fitting deep ReLU networks. In the theoretical deep learning literature, it has been shown that supervised deep networks can outperform other methods if the target function exhibits some compositional structure. Making the link to supervised learning allows us to exploit this property also for density estimation. This is highly desirable as a compositional structure is frequently imposed in modelling of densities. Examples include copula models \cite{PairCopulaConstructionsMultipleDependence,nagler2016evading} and Bayesian network models \cite{koller2009probabilistic}, see also Section \ref{S: Examples}.

Theorem \ref{T: Oracle Inequality} is our main theoretical contribution and establishes an oracle inequality for supervised regression methods applied to nonparametric density estimation. The key technical difficulty in the proof is to deal with the dependence incurred by generating the response variables in the first step of the proposed method. To control the dependence, we use a Poissonization argument. Applying the derived oracle inequality, we show in Theorem \ref{T: Convergence Rates} that deep ReLU networks can obtain fast convergence rates, given that the underlying density has a compositional structure. For sufficiently smooth densities, the convergence rates are, up to logarithmic factors in the sample size, the same as the recently obtained minimax rates in the nonparametric regression model under compositional structure on the regression function, \cite{NonParametricRegressionReLU}. But there are also smoothness regimes where the convergence rate is slower by a polynomial order in the sample size if compared to the nonparametric regression case. This is due to the first step in the construction of the estimator that transforms the density estimation problem into a supervised regression problem. But still then there are scenarios where the convergence rate is considerably faster than doing off-the-shelf kernel density estimation without taking the underlying compositional structure of the density into account.

The paper is structured as follows. Section \ref{S:Model} describes the construction of suitable response variables from the data. In Section \ref{S: Main Results} we present a suitable oracle inequality for non-i.i.d.\ data. Furthermore, we provide convergence rates in the case that the regression estimator is a deep neural network and the underlying density are compositional functions. In Section \ref{S: Examples} we shortly discuss some density models that exhibit compositional structure. A small (exploratory) simulation is provided in Section \ref{S: Simulations}. Section \ref{S: RelLit} summarizes related literature. Almost all proofs are deferred to the Appendix.

\subsection{Notation}

We denote vectors and vector valued functions by bold letters. For a vector $\bx=(x_1,\ldots,x_k)^\top$ we define $|\bx|_{\infty}=\max_{i=1,\hdots,k}|x_i|$, $|\bx|_1=\sum_{i=1}^k|x_i|.$ and $|\bx|_0=\sum_{i=1}^k\1_{\{x_i\neq 0\}}$. For partial derivatives we use multi-index notation, that is, if $\balpha=(\alpha_1,\ldots,\alpha_d)\in \{0,1,2,\ldots \}^d$ we set $\partial^{\balpha}:=\partial_{x_1}^{\alpha_1}\hdots\partial_{x_d}^{\alpha_d}.$ We denote the supremum norm of a function $f:\mD\rightarrow \mathbb{R}$ by $\|f\|_\infty=\sup_{\bx\in \mD}|f(\bx)|$. As commonly defined in nonparametric statistics, for a real number $x\in\mathbb{R},$ $\floor{x}$ is the largest integer $< x$ and $\ceil{x}$ is the smallest integer $\geq x$. The minimum and maximum of two real numbers $x,y$ are also denoted by the respective expressions $x\wedge y$ and $x\vee y.$ For two sequences $(a_n)_n$ and $(b_n)_n$, we write $a_n\lesssim b_n$ if there exists a constant $C$ such that $a_n\leq Cb_n$ for all $n$. Moreover, $a_n\asymp b_n$ means that $a_n\lesssim b_n$ and $b_n\lesssim a_n$. If no basis is specified, then $\log=\ln.$

\section{Conversion into a supervised learning problem}
\label{S:Model}

We consider nonparametric density estimation on the hypercube $[0,1]^d$, where we observe $2n$ i.i.d.\ random vectors $\bX_i\in[0,1]^d$ which are distributed according to an unknown density $f_0$ from a nonparametric class. The density estimation problem is to recover this density $f_0$ from the data $\bX_1,\ldots, \bX_{2n}.$ Here the sample size $2n$ is chosen for notational convenience, as we will do data splitting. It is moreover convenient to rename the second half of the dataset and denote it by $\bX_1',\ldots,\bX_n'.$ Thus, we are observing the $2n$ i.i.d.\ random vectors $(\bX_1,\ldots, \bX_{n},\bX_1',\ldots, \bX_{n}').$ In the first step of the proposed method, the second half of the sample $\bX_1',\ldots, \bX_{n}'$ is used to compute an undersmoothed kernel density estimator. From that we construct a response variable $Y_i$ for each of the remaining datapoints $\bX_i$. The response variables $Y_i$ can be interpreted as noisy versions of $f_0(\bX_i).$ The augmented data $(\bX_1,Y_1),\ldots,(\bX_n,Y_n)$ are now viewed as a nonparametric regression problem with the unknown density $f_0$ as regression function. Thus, any nonparametric regression technique could be applied to recover the regression function $f_0$ from the supervised data $(\bX_1,Y_1),\ldots,(\bX_n,Y_n).$ Here we propose to fit a deep neural network. This is motivated by previous work which has shown that deep neural networks can adapt to various forms of structural constraints and avoid the curse of dimensionality \cite{kohler2005adaptive,poggio2017and, bauer2019,NonParametricRegressionReLU,kohler2021rate}. As we will argue below, such structural constraints occur in modelling of multivariate densities. Fitting a neural network to the regression data $(\bX_1,Y_1),\ldots,(\bX_n,Y_n)$ is therefore natural.

In nonparametric statistics, a function $K:\mathbb{R}\to \mathbb{R}$ is called a kernel if $\int K(u) \, du =1.$ If for some positive integer $s,$ we moreover have vanishing moments $\int u^\ell K(u) \, du=0$ for all $\ell=1,\ldots,s$, and $\int |u|^{s+1}K(u) \, du <\infty,$ then $K$ is a called a kernel of order $s$. We now outline the two steps of the method.

\textbf{Step 1:} Choose a kernel $K$ with $\|K\|_\infty <\infty$ and support on $[-1,1].$ For a bandwidth $h_n$ satisfying $(\log(n)/n)^{1/d}\leq h_n\leq 2(\log(n)/n)^{1/d}$ and such that $h_n^{-1}$ is a positive integer for all $n>1$ (existence of such a sequence is guaranteed by Lemma \ref{lem.hn_exist}), consider the multivariate kernel density estimator based on the subsample $\bX_1',\ldots,\bX_n'$ with $\bX_{\ell}'=(X'_{\ell,1},\ldots,X'_{\ell,d})^\top$ given by
\begin{equation}
    \begin{aligned}
    \KDE(\bx):=\frac{1}{nh_n^d}\sum_{\ell=1}^{n}\prod_{r=1}^{d}K\left(\frac{X'_{\ell,r}-x_r}{h_n}\right).
 \end{aligned} 
 \label{eq.KDE_def}
 \end{equation}

For $i=1,\ldots,n,$ define 
\begin{equation}
    Y_i:=\KDE(\bX_i). 
    \label{eq.YiKDE}
\end{equation}
Setting $\epsilon_i:=Y_i-f_0(\bX_i),$ we obtain the regression model
\begin{equation}
    \begin{aligned}
   Y_i=f_0(\bX_i)+\epsilon_i, \quad i=1,\ldots,n.
 \end{aligned} 
 \label{eq.regr_model}
 \end{equation}
\textbf{Step 2:} Compute an estimator $\wh f$ based on the data $(\bX_1,Y_1),\ldots,(\bX_n,Y_n).$

\begin{Definition}
\label{defi.2step}
We refer to any such $\wh f$ as two-stage nonparametric density estimator. If the kernel is of order $s,$ we call $\wh f$ the two-stage nonparametric density estimator with kernel of order $s.$ 
\end{Definition}

Both $\wh f$ and the kernel density estimator $\KDE$ are estimators for $f_0.$ However, because of the small bandwidth, the kernel density estimator severely undersmooths. The variance of $\KDE(\bx)$ at a given point $\bx$ is known to scale with $1/(nh_n^d)\asymp 1/\log(n).$ This means that the noise variables $\eps_i$ scale with $1/\sqrt{\log(n)}$ in the sample size. Therefore, the denoising happens in the second step of the proposed two-step procedure. 

Although the notation seems to suggest that \eqref{eq.regr_model} is the standard nonparametric regression framework, all data points depend on the underlying kernel density estimator $\KDE.$ The pairs $(\bX_1,Y_1),\ldots,(\bX_n,Y_n)$ are henceforth dependent and thus not i.i.d. To deal with this dependence is the main technical challenge in the analysis of the proposed method.

The kernel density estimator in Step 1 undersmooths and does not require knowledge of the true smoothness. The conditions on the kernel $K$ are standard. Taking a kernel of order $s$ together with an optimal bandwidth choice is known to lead to optimal convergence rates if the smoothness of the density is at most $s+1.$ The fact that the bandwidth is chosen such that $h_n^{-1}$ is a positive integer allows us to partition $[0,1]$ into $h_n^{-1}$ disjoint intervals of length $h_n.$ 

For fitting a function to the data $(\bX_1,Y_1),\ldots,(\bX_n,Y_n)$ in the second step of the procedure, machine learning methods aim to minimize a loss. For regression, the most common choice is the least squares loss $\tfrac{1}{n}\sum_{i=1}^n \big(Y_i-f(\bX_i)\big)^2.$ The least squares estimator $\widehat{f}_n$ over a function class $\mF$ for the density $f_0$ is defined as any global minimizer of the least squares loss
\begin{equation*}     \begin{aligned}
    \widehat{f}_n\in\argmin_{f\in\mF} \frac{1}{n}\sum_{i=1}^n \big(Y_i-f(\bX_i)\big)^2.
 \end{aligned} \end{equation*}
Due to the nonconvex energy landscape, neural network training usually does not find the global minimum. The difference between training error of the estimator and training error of the global minimum is commonly referred to as optimization error. For any estimator $\widehat{f}$ taking values in a function class $\mF,$ and data generated from the nonparametric regression model with regression function $f_0,$ we consider here the optimization error 
\begin{equation}
\Delta_n(\widehat{f},f_0):=\mathbb{E}_{f_0}\left[\frac{1}{n}\sum_{i=1}^n(Y_i-\widehat{f}(\bX_i))^2-\inf_{f\in\mF}\frac{1}{n}\sum_{i=1}^n(Y_i-f(\bX_i))^2\right],
\label{eq.opt_error_def}
\end{equation}
where the expectation is taken over the full data set, making $\Delta_n(\widehat{f},f_0)$ deterministic.

The risk of an estimator $\widetilde{f}$ is given by
\begin{equation}\begin{aligned}
R(\widetilde{f},f_0)&:=\mathbb{E}_{f_0,\bX}\left[(\widetilde{f}(\bX)-f_0(\bX))^2\right]
=\int \mathbb{E}_{f_0}\left[(\widetilde{f}(\bx)-f_0(\bx))^2\right] \, f_0(\bx) \, d\bx.
 \end{aligned} \end{equation} 
 Here $\bX\overset{d}{=}\bX_1$ is independent of the data and $\mathbb{E}_{f_0,\bX}$ is the expectation with respect to the joint distribution of $\bX$ and the data set. We denote by $\mathbb{E}_{\bX}$ the expectation with respect to $\bX$.

\section{Main results}\label{S: Main Results}

We assume that the density $f_0$ belongs to the class of $\beta$-H\"{o}lder smooth functions on $\mathbb{R}^d$ with support on $[0,1]^d$. For $\beta>0$ and a domain $\mD\subseteq \mathbb{R}^d,$ the ball of $\beta$-H\"{o}lder functions with radius $Q$ is defined as
\begin{equation}\label{Eq: Holder ball}
    \begin{aligned}
    \mH_d^{\beta}(\mD,Q):=\Bigg\{f: \mD\rightarrow \mathbb{R}:&\sum_{\bgamma: \, 0\leq |\bgamma|_1<\beta} \|\partial^{\bgamma}f\|_{\infty}\\
   &\quad +\sum_{\bgamma:|\bgamma|_1=\floor{\beta}}\sup_{\bx,\by\in\mD,\bx\neq\by}\frac{|\partial^{\bgamma}f(\bx)-\partial^{\bgamma}f(\by)|}{|\bx-\by|_{\infty}^{\beta-\floor{\beta}}}\leq Q\Bigg\},
    \end{aligned}
\end{equation}
where $\|\cdot\|_\infty$ denotes the supremum norm on $\mD$, $\partial^{\bgamma}=\partial_{x_1}^{\gamma_1}\ldots\partial_{x_d}^{\gamma_d},$ and $\bgamma=(\gamma_1,\ldots,\gamma_d)\in \{0,1,2,\ldots\}^d.$ To define the partial derivatives if $\mD$ is not an open set, we assume that there exists an open set $\mU\supset \mD$ and an extension of $f$ on $\mU$ for which the partial derivatives $\partial^\gamma$ for all $\bgamma$ with $ |\bgamma|_1< \beta$ are defined. The class of $\beta$-H\"{o}lder smooth densities on $\mathbb{R}^d$ and support on $[0,1]^d$ can subsequently be defined as
\begin{equation*}
    \begin{aligned}
    \mathcal{C}^{\beta}_d(Q):=\left\{f\in \mH^{\beta}_d(\mathbb{R}^d,Q):\supp f\subseteq [0,1]^d,\int_{[0,1]^d}f(\bx)\,d\bx=1,f\geq 0\right\}.
    \end{aligned}
\end{equation*}
We define the class of $\beta$-H\"{o}lder smooth densities on $[0,1]^d$ by restricting $\beta$-H\"{o}lder smooth densities on $\mathbb{R}^d$ to $[0,1]^d$,
\begin{equation*}
    \begin{aligned}
        \mathcal{C}^\beta_d([0,1]^d,Q):=\left\{f:[0,1]^d\rightarrow \mathbb{R}: \text{there exists } h\in \mathcal{C}^\beta_d(Q) \text{ such that } f=h|_{[0,1]^d}\right\}.
    \end{aligned}
\end{equation*}
Below we assume that the true density lies in this space. The condition that densities in this space can be extended to smooth functions on $\mathbb{R}^d$ is imposed to avoid (technical) difficulties of the kernel density estimator near the boundary of $[0,1]^d$. For a reference dealing with the behaviour of kernel estimators near boundaries, see Section 2.11 of \cite{wand1994kernel}. 

We state the oracle inequality for estimators taking values in an abstract function class $\mF(F)\subseteq\{f:\|f\|_{\infty}\leq F\}$. For that, we denote by $\mN_{\mF}(\delta)$ the covering number of a class $\mF(F)$ with respect to the supremum norm. More specifically, $\mN_{\mF}(\delta)$ is the smallest number of supremum norm balls with radius $\delta$ and centers contained in $\mF$ that are necessary to cover $\mF$.

\begin{Theorem}\label{T: Oracle Inequality}
For $n\geq 3,$ consider the density estimation model defined by \eqref{eq.KDE_def}-\eqref{eq.regr_model} with density $f_0$ in the H\"{o}lder class $  \mathcal{C}^\beta_d([0,1]^d,Q).$ Let $\wh f$ be a two-stage nonparametric density estimator with kernel of order $\floor{\beta}$ as defined in Definition \ref{defi.2step}. If $\wh f$ takes values in the function class $\mF=\mF(F)$, with $F\geq \max\{Q,1\},$ then there exist constants $C_1,C_2,C_3$ only depending on $F,K,d,Q,\beta$ such that for any $\delta>0,$
\begin{equation*}
    \begin{aligned}
    R(\widehat{f},f_0)\leq & \, C_1\frac{\log^2(n)\log(n\vee \mN_{\mF}(\delta))}{n}+C_2\delta+C_3\left(\frac{\log(n)}{n}\right)^{\frac{2\beta}{d}}+6\Delta_n(\widehat{f},f_0)\\
    &+7\inf_{f\in\mF}\mathbb{E}_{\bX}\left[(f(\bX)-f_0(\bX))^2\right].
    \end{aligned}
\end{equation*}
\end{Theorem}

 As common for oracle inequalities, the upper bound contains an approximation term, a complexity term involving the metric entropy, and the optimization error $\Delta_n(\widehat{f},f_0).$ For neural networks and other parametrizable function classes, the metric entropy $\log(\mN_\mF(\delta))$ depends only logarithmically on $\delta$ and one can choose $\delta=1/n$, making the $C_2\delta$ term negligibly small.

 Additionally, the bound contains the term $C_3(\log(n)/n)^{2\beta/d}$ that is due to the bandwidth choice $h_n\asymp (\log(n)/n)^{1/d}$ and a term of the order $h_n^{2\beta}$ that can be traced back to Proposition \ref{P: Bound empirical risk}. To decrease the order of the $C_3(\log(n)/n)^{2\beta/d}$ term, it is tempting to aim for a smaller bandwidth $h_n\ll n^{-1/d}.$ However, even if the data points are equally spaced in $[0,1]^d,$ the distance of two neighboring data points is $n^{-1/d}.$ Thus for bandwidth $h_n\ll n^{-1/d}$, it follows from the definition of the kernel density estimator in \eqref{eq.KDE_def} that the estimated density degenerates into separate spikes centered around the data points $\bX_1',\ldots,\bX_n'.$ As a consequence, a generated response variables $Y_i$ will likely either be extremely large or attain a value near zero and the two-step method that we propose will not work anymore.
 
 Compared to oracle inequalities for i.i.d.\ data in the nonparametric regression model, the main difficulty in the proof of the previous theorem is to deal with the various sources of dependence. Even conditionally on the sample $\bX_1',\ldots,\bX_n',$ two noise variables $\eps_i=\KDE(\bX_i)-f_0(\bX_i)$ and $\eps_j=\KDE(\bX_j)-f_0(\bX_j)$ are highly dependent in the regression model \eqref{eq.regr_model} whenever $\bX_i,\bX_j$ are close. If we also take the randomness of the $\bX_1',\ldots,\bX_n'$ into account, then the evaluation of the kernel density estimator at two deterministic points $\KDE(\bx)$ and $\KDE(\bx')$ are (slightly) dependent random variables even if $\bx$ and $\bx'$ are far away. The rationale behind is that if $\KDE(\bx)>f_0(\bx),$ then it is a bit more likely that $\KDE(\bx')<f_0(\bx')$ as $\int \KDE(\bx)-f_0(\bx) \, d\bx=1-1=0.$ To control this dependence, it is common to use Poissonization techniques, cf.\ \cite{WeakConvergenceEmpProcesses} Section 3.5.2., \cite{EmpiricalProcessesGaenssler}, and \cite{dudley1984course} Section 8.3. To explain the idea underlying Poissonization, consider a Poisson point process on $[0,1]^d.$ For two disjoint sets $A,B\subset [0,1]^d,$ the number of points that fall in set $A$ and the number of points that fall in set $B$ are independent random variables. Also many statistics can be shown to produce independent random variables if they are separately applied to the points in $A$ and the points in $B.$ The Poissonization trick is now to pretend that we do not have $n$ data points $\bX_1',\ldots,\bX_n'$ but $\mM$ data points, that is, we observe $\bX_1',\ldots,\bX_\mM'$ for $\mM$ an independently generated Poisson random variable with intensity $n.$ Then, $\bX_1',\ldots,\bX_\mM'$ can be interpreted as the points of a Poisson point process with intensity $\bx\mapsto nf_0(\bx).$ We can now also redefine the kernel density estimator $\KDE$ for $\bX_1',\ldots,\bX_\mM',$ by keeping the same normalization $1/n$ but summing over $\ell=1,\ldots,\mM.$ Because we have chosen the kernel to have support in $[-1,1],$ $\KDE(\bx)$ only depends on the subset $D(\bx):=\{\bX_i':|\bX_i'-\bx|_\infty\leq h_n, i=1,\ldots,\mM\} \subseteq \{\bX_1',\ldots,\bX_{\mM}'\}.$ If $|\bx-\by|_\infty>2h_n,$ then $D(\bx)$ and $D(\by)$ are disjoint sets and one can even show that the statistics $\KDE(\bx)$ and $\KDE(\by)$ are independent. To control the change of probability going from $n$ to $\mM$ observations, we can apply the following result:
\begin{Lemma}
    \label{lem.Poissonization}
    For $\mM$ and $\bX_1',\bX_2',\ldots$ as above, for any function $h,$ and any measurable set $A$,
    \begin{align*}
        \mathbb{P}\bigg(\sum_{i=1}^n h(\bX_i')\in A\bigg)
        \leq \sqrt{2e\pi n}\, \mathbb{P}\bigg(\sum_{i=1}^{\mathcal{M}} h(\bX_i')\in A\bigg).
    \end{align*}
\end{Lemma}

\begin{proof}[Proof of Lemma \ref{lem.Poissonization}]
We have 
\begin{align*}
        \mathbb{P}\bigg(\sum_{i=1}^n h(\bX_i')\in A\bigg)
        &= \mathbb{P}\bigg(\sum_{i=1}^{\mathcal{M}} h(\bX_i')\in A \, \bigg|\, \mathcal{M}=n\bigg)
        \leq \frac{ \mathbb{P}\big(\sum_{i=1}^{\mathcal{M}} h(\bX_i')\in A\big)}{\mathbb{P}(\mathcal{M}=n)}.
    \end{align*}
Since $\mM$ is a $\Poi(n)$ random variable we have that $\mathbb{P}(\mM=n)=n^{n}e^{-n}/n!.$ By Stirling's formula, see for example \cite{RobbinsStirlingFormula}, $n!\leq\sqrt{2\pi n}(n/e)^{n}e^{1/(12n)}\leq \sqrt{2e\pi n}(n/e)^{n}$ and $1/\mathbb{P}(\mM=n)\leq \sqrt{2e\pi n}.$
\end{proof}

While Poissonization removes dependence, the factor $\sqrt{2e\pi n}$ arises in the bounds.

\subsection{Neural networks}

We study the effect of fitting a deep ReLU network in the regression step of the proposed two-step procedure. We rely on the mathematical formulation of deep neural networks introduced in \cite{NonParametricRegressionReLU} and briefly recall the details for completeness of the exposition. The rectified linear unit (ReLU) activation function is $\sigma(x):=\max\{x,0\}$. For any vectors $\bv=(v_1,\ldots,v_r)^{\top},\by=(y_1,\ldots,y_r)^{\top}\in\mathbb{R}^r$, we define the shifted activation function $\sigma_{\bv}\by:=(\sigma(y_1-v_1),\ldots,\sigma(y_r-v_r))^{\top}.$
The number of hidden layers is specified by $L$ and the width of the layers is denoted by the width vector $\bp=(p_0,\ldots,p_{L+1})\in\mathbb{N}^{L+2}.$ A network with network architecture $(L,\bp)$ is any function of the form
\begin{equation}\label{Eq: Network Architecture}
    \bf:\mathbb{R}^{p_0}\rightarrow \mathbb{R}^{p_{L+1}},\ \bx\mapsto \bf(\bx)=W_L\sigma_{\bv_L}W_{L-1}\sigma_{\bv_{L-1}}\hdots W_1\sigma_{\bv_1}W_0\bx,
\end{equation}
where $W_j$ is a $p_{j+1}\times p_{j}$ weight matrix and $\bv_j\in\mathbb{R}^{p_j}$ is a shift vector. We use the convention that $\bv_0:=(0,\ldots,0)^{\top}\in\mathbb{R}^{p_0}.$
Denote the maximum entry norm of a matrix $W$ by $\|W\|_{\infty}$. The class of ReLU networks with architecture $(L,\bp)$ and parameters bounded in absolute value by one is
\begin{equation*}
    \mF(L,\bp):=\left\{\bf \text{ is of the form \eqref{Eq: Network Architecture} }: \max_{j\in\{0,\ldots,L\}}\|W_j\|_{\infty}\vee |\bv_j|_{\infty}\leq 1\right\}.
\end{equation*}
For a matrix $W$ denote the counting norm (number of non-zero entries) by $\|W\|_{0}$. We are interested in sparsely connected neural networks where the number of non-zero or active parameters is small compared to the total number of parameters. For this we define the class of $s$-sparse networks, that are bounded in uniform norm by $F$, as
\begin{align*}
    \mF(L,\bp,s,F):=
    \bigg\{\bf \text{ is of the form \eqref{Eq: Network Architecture} }: &\max_{j\in\{0,\ldots,L\}}\|W_j\|_{\infty}\vee |\bv_j|_{\infty}\leq 1, \\ & \sum_{j=0}^L\|W_j\|_0+|\bv_j|_0\leq s, \||\bf|_{\infty}\|_{\infty}\leq F\bigg\}.
\end{align*}

\begin{Definition}[Two-stage neural network density estimator]
\label{eq.defi_estimator}
If the two-stage nonparametric density estimator $\wh f$ fits in the second step a neural network from the class $\mF(L,\bp,s,F)$ to the augmented sample $(\bX_1,Y_1),\ldots,(\bX_n,Y_n)$, then we refer to $\wh f$ as two-stage neural network density estimator. If the kernel in the first step of the procedure is of order $s,$ we call $\wh f$ a two-stage neural network density estimator with kernel of order $s.$
\end{Definition}

\subsection{Structural constraints: compositions of functions}
Deep neural networks are built by computing individual layers. Previously derived statistical theory has shown that they are well-suited to pick up compositional structure in the regression function, \cite{kohler2005adaptive,poggio2017and, bauer2019,NonParametricRegressionReLU,kohler2021rate}. In this work we follow the composition structure introduced in \cite{NonParametricRegressionReLU} and impose it on the multivariate density $f_0$, that is, we assume that $f_0=g_q\circ g_{q-1}\circ\ldots\circ g_1\circ g_0,$ with $g_i:[a_i,b_i]^{d_i}\rightarrow[a_{i+1},b_{i+1}]^{d_{i+1}}.$  Denote by $g_i=(g_{ij})^{\top}_{j=1,\ldots,d_{i+1}}$ the components of $g_i$ and let $t_i$ be the maximal number of variables on which each of the $g_{ij}$ depends. It always holds that $t_i\leq d_i$ and for certain models, $t_i$ can be much smaller than $d_i$. Section \ref{S: Examples} provides examples of densities where this is the case. As we consider density estimation on $[0,1]^d,$ it follows that $d_0=d$, $a_0=0$, $b_0=1$ and $d_{q+1}=1$. Since $g_{ij}$ depends on $t_i$ variables, we also interpret it as a function $[a_i,b_i]^{t_i}\rightarrow[a_{i+1},b_{i+1}]^{d_{i+1}}$ whenever this is convenient. Denote by $\alpha_i$ the smoothness of each of the functions $g_{ij}$. Then $g_{ij}\in \mH^{\alpha_i}_{t_i}([a_i,b_i]^{t_i},Q_i)$ and the space of compositions of these smooth functions is given by\begin{equation}
\begin{aligned}
     \mG(q,\bd,\bt,\balpha,Q'):=\Big\{f= \ &g_q\circ\ldots\circ g_0: g_i=(g_{ij})_j:[a_i,b_i]^{d_{i}}\rightarrow [a_{i+1},b_{i+1}]^{d_{i+1}},\\ &   g_{ij}\in \mH^{\alpha_i}_{t_i}([a_i,b_i]^{t_i},Q'), \text{ for some }|a_i|,|b_i|\leq Q'\Big\}, 
\end{aligned}
\label{eq.comp_Hspace_def}
\end{equation}
with $\bd:=(d_0,\ldots,d_{q+1}),$ $\bt:=(t_0,\ldots,t_q),$ and $\balpha:=(\alpha_0,\ldots,\alpha_q).$

If two functions $h,g:\mathbb{R}\rightarrow \mathbb{R}$ have respective smoothness $\alpha_h, \alpha_{g}\leq 1$ then it follows from the definition of the H\"older space that the composition $f:=g\circ h$ has smoothness $\alpha_h\alpha_{g}.$ For $\alpha_h>1$ or $\alpha_f>1,$ this is not necessarily true anymore. It turns out that the convergence rates for a compositional function in $\mG(q,\bd,\bt,\balpha,Q')$ are governed by a notion of effective smoothness indices which are defined as
\begin{equation*}
    \alpha_i^{*}:=\alpha_i\prod_{\ell=i+1}^q(\alpha_{\ell}\wedge 1).
\end{equation*}
Indeed, in the nonparametric regression model with i.i.d.\ observations the minimax estimation rate is up to $\log(n)$-terms
\begin{equation}
    \phi_n:=\max_{i=0,\ldots,q} n^{-\frac{2\alpha^*_i}{2\alpha^*_i+t_i}},
    \label{eq.phin_def}
\end{equation}
cf. \cite{NonParametricRegressionReLU}. A function can be represented as a composition in different ways. In the function representation $f=g_q\circ\ldots\circ g_0$, the $\alpha_i,t_i$ and the components $g_0,\hdots,g_q$ are not identifiable. Since we are only interested in estimating the density $f_0$ this does not constitute a problem.
 
The oracle inequality in Theorem \ref{T: Oracle Inequality} together with the approximation and covering entropy bound results for deep ReLU networks from \cite{NonParametricRegressionReLU} yields a convergence rate result for the proposed two-stage neural networks estimator. Recall that $\Delta_n(\widehat{f}_n,f_0)$ is the optimization error defined in \eqref{eq.opt_error_def}.

\begin{Theorem}[Convergence rates]\label{T: Convergence Rates}
For $n\geq 3,$ consider the density estimation model defined by \eqref{eq.KDE_def}-\eqref{eq.regr_model} with density $f_0$ in the H\"{o}lder class $\mathcal{C}_d^\beta([0,1]^d,Q)\cap \mG(q,\bd,\bt,\balpha,Q).$ Let $\wh f$ be a two-stage neural network density estimator with kernel of order $\floor{\beta}$ as defined in Definition \ref{eq.defi_estimator}
for the neural network class $\mF(L,(p_0,\ldots, p_{L+1}),s,F)$ with parameters satisfying
\begin{compactitem}
    \item[(i)] $F\geq \max\{Q,1\},$
    \item[(ii)] $\sum_{i=1}^q \tfrac{\alpha_i+t_i}{2\alpha_i^*+t_i}\log_2(4t_i\vee 4\alpha_i)\log_2(n)\leq L\lesssim n\phi_n,$
    \item[(iii)] $n\phi_n\lesssim \min_{i=1,\hdots,L,}p_i$,
    \item[(iv)] $s\asymp n\phi_n\log(n)$.
\end{compactitem}
Then there exists a constant $C_4,$ only depending on $q,\bd,\balpha,\bt,F,\beta,K$ and the implicit constants in (ii), (iii), and (iv), such that 
\begin{equation*}
    R(\widehat{f}_n,f_0)\leq C_4 L\max\Big(\phi_n \log^4(n),n^{-2\beta/d}\Big)
    +6\Delta_n(\widehat{f}_n,f_0).
\end{equation*}
\end{Theorem}

Any admissible compositional structure $f=g_q\circ \ldots \circ g_0$ leads to an upper bound on the risk. The estimator achieves therefore the fastest convergence rate among all possible representations.

To analyze the estimation risk, we will now ignore the optimization error $\Delta_n(\widehat{f}_n,f_0)$ and focus on the statistical estimation rate $L\max(\phi_n \log^4(n),n^{-2\beta/d}).$ Choosing depth $L\asymp \log(n),$ the convergence rate for the learned network $\widehat f$ is thus $\phi_n+n^{-2\beta/d},$ up to $\log(n)$-factors. The $n^{-2\beta/d}$-term is due to the kernel density estimator in the first step and already occurs in the general oracle inequality, see also the discussion after Theorem \ref{T: Oracle Inequality}. 

If the density exhibits a compositional structure, it is now of interest to understand which of the two terms $\phi_n$ and $n^{-2\beta/d}$ will drive the convergence rate. If the compositional structure is strong enough to make $\phi_n$ small but $\beta$ is small compared to $d$, then $n^{-2\beta/d}$ dominates the convergence rate. This is faster than the standard nonparametric rate $n^{-2\beta/(2\beta+d)}$ for estimation of a $\beta$-smooth function but still suffers from the curse of dimensionality.

If $2\beta\geq d,$ then $n^{-2\beta/d}=O(n^{-1})$. Since $\phi_n\gg n^{-1},$ the rate is in this case always of order $\phi_n$ (up to log-factors). 
The condition $2\beta\geq d$ appears frequently in the literature on nonparametric statistics and empirical risk minimization. For $d=1,$ $2\beta>1$ is known to be a necessary condition for nonparametric density estimation and nonparametric regression to be asymptotically equivalent if all densities are bounded from below \cite{nussbaum1996asymptotic, MR4050245}. The condition $2\beta\geq d$ seems also necessary to ensure that the nonparametric least squares estimator achieves the optimal nonparametric rate $n^{-2\beta/(2\beta+d)}$, see e.g.\ Section 6.1 in \cite{SH2022}. Barron \cite{barron1994approximation} showed that shallow neural networks can circumvent the curse of dimensionality under a Fourier criterion. A sufficient, but not necessary condition for this Fourier criterion to be finite is that the partial derivatives up to the least integer $\beta$ such that $2\beta\geq d+2$ are square-integrable, see Example 15 in Section IX of \cite{barron1993universal}.

Instead of the proposed two-step method, it seems tempting to further iterate the estimation procedure by generating new response variables $Y_i':=\wh f_n(\bX_i),$ $i=1,\ldots,n,$ from the estimator $\wh f_n$ and running another neural network fit on the newly generated supervised sample $(\bX_1,,Y_1'),\ldots, (\bX_n,Y_n').$ We believe that this can, however, not improve the convergence rate. The reason is that the new network fit cannot decrease the bias that was already present in the estimator $\wh f_n.$ The rate in Theorem \ref{T: Convergence Rates} is obtained by balancing different terms. In particular the squared approximation error that is closely related to the squared bias is of the order of the convergence rate. Thus, if the bias cannot be reduced by another neural network fit also the convergence rate cannot be improved.

In the next section, we provide more explicit examples of densities that satisfy the compositional assumption and attain the convergence rate $\phi_n.$

\section{Examples of multivariate densities with compositional structure}\label{S: Examples}

Compositional structures arise naturally in density modelling. One possibility to see this is to rewrite the joint density $f$ as a product
\begin{equation*}
    f(x_1,\ldots,x_d)=f(x_d \rvert x_1,\ldots, x_{d-1})\cdot \ldots \cdot f(x_2\rvert x_1) f(x_1).
\end{equation*}
Each factor $f(x_i\rvert x_1,\ldots, x_{i-1})$ is a function of $i$ variables. But the effective number of variables can be much smaller under conditional independence of the variables. When $\bX=(X_1,\ldots,X_d)^\top$ is generated for instance from a Markov chain, $X_i$ only depends on $X_{i-1}$ and the density is a product of bivariate conditional densities
\begin{equation}
    f(x_1,\ldots,x_d)=f(x_d \rvert x_{d-1})\cdot \ldots \cdot  f(x_2\rvert x_1) f(x_1).
    \label{eq.MC}
\end{equation}
Such a structure could occur if the individual data vectors are recordings from a time series, that is, every observation $\bX_i=(X_{i,1},\ldots,X_{i,d})^\top$ contains measurements of the same quantity taken at $d$ different times instances. We now assume that the density is of the form
\begin{align}\label{Eq: General Product form}
    f(x_1,\ldots,x_d)
    =
    \prod_{I\in \mR} \psi_I(x_I),
\end{align}
with $\mR\subseteq \{ S\subset\{1,\ldots,d\}, |S|\leq r\}$, $r$ a given number, $x_I=(x_i)_{i\in I},$ and $\psi_I$ non-negative functions. Observe that $|\mR|\leq \sum_{s=1}^r \binom{d}{s}.$   
\begin{Lemma}\label{L: General potential product2}
    Consider a density $f$ of the form \eqref{Eq: General Product form}. If all the functions $\psi_I$ in the decomposition satisfy $\psi_I\in \mH_{r_I}^{\gamma}([0,1]^{r_I},Q)$ for some $r_I\leq r,$ then the density $f$ can be rewritten as a composition $g_1\circ g_0$ of the form \eqref{eq.comp_Hspace_def}, with
    $(d_0,d_1)=(d,|\mR|)$, $(t_0,t_1)=(r,|\mR|),$ $(\alpha_0,\alpha_1)=(\gamma,\zeta),$ and $\zeta$ arbitrarily large.
\end{Lemma}

Ignoring here and in the rest of this section the optimization error, under the combined conditions of Lemma \ref{L: General potential product2} and Theorem \ref{T: Convergence Rates}, the proposed two-stage neural network density estimator achieves, up to $\log(n)$-factors, the convergence rate 
\begin{align}
    n^{-\frac{2\gamma}{2\gamma+r}}\vee n^{-\frac{2\beta}d},
\end{align}
with $\beta$ the (global) H\"older smoothness of the joint density $f.$ 
If $\beta=\gamma,$ that is, the effective smoothness $\gamma$ coincides with the global H\"{o}lder smoothness $\beta$ of $f$, then the achieved rate is $n^{-\frac{2\gamma}{2\gamma+r}}$ if $\gamma\geq (d-r)/2$ and $n^{-\frac{2\beta}d}$ if $\gamma\leq (d-r)/2.$ We always have $\beta\geq \gamma.$ If $\beta>\gamma,$ we conjecture that in most cases there exists a different compositional representation of the density $f$ with $\beta$-smooth $\psi_I$.

Next, we discuss three examples of models that are of the form \eqref{Eq: General Product form}.

\textbf{Independent variables:} If $\bX=(X_1,\ldots,X_d)$ is a vector containing independent random variables, the joint density is given by
\begin{equation}\label{Eq: joint density of independent variables}
    f(x_1,\ldots,x_d)=\prod_{i=1}^df_i(x_i),
\end{equation}
where $f_i$ is the marginal density of $X_i.$ We assume that $f_i$ is $\gamma_i$-H\"older smooth. If we are unaware of the independence and simply use multivariate kernel density estimators to estimate $f$, we will suffer from the curse of dimensionality as demonstrated for Gaussian densities and Gaussian kernels in Chapter 7 of \cite{MultivariateDensityEstimation}.

Observe that \eqref{Eq: joint density of independent variables} is of the form \eqref{Eq: General Product form}, with $\mR$ the set of singletons. Thus under the combined conditions of Lemma \ref{L: General potential product2} and Theorem \ref{T: Convergence Rates}, we get, up to $\log(n)$-factors, the convergence rate $n^{-2\alpha/(2\alpha+1)}\vee n^{-2\beta/d},$ with $\beta$ the (global) H\"older smoothness of the joint density $f$ and $\alpha:=\min_{i=1,\ldots,d}\gamma_i.$ The construction in Lemma \ref{L: General potential product2} implies that $\beta\geq \alpha.$ The next result shows that in this case we necessarily have equality $\beta=\alpha.$ In other words the smoothness of the joint density $f$ has to be equal to the (effective) smoothness of the least smooth marginal density.

\begin{Lemma}\label{L: Independent variables smoothnes bound} Consider a density $f$ of the form \eqref{Eq: joint density of independent variables} and let $\alpha>0.$ If $f$ is $\alpha$-H\"{o}lder smooth, then $f_1,\ldots,f_d$ are $\alpha$-H\"older smooth.
\end{Lemma}

\textbf{Graphical models:} Let $(X_1,\ldots,X_d)$ be a $d$-dimensional random vector. An undirected graphical model (or Markov random field) is defined by a graph with $d$ nodes representing the $d$ random variables. In this graph, no edge between node $i$ and $j$ is drawn if and only if $X_i,X_j$ are conditionally independent given all the other variables $\{X_1,\ldots,X_d\}\setminus \{X_i,X_j\}.$ A clique in a graph is any fully connected subgraph. When the joint density $f(x_1,\ldots,x_d)$ is strictly positive with respect to a $\sigma$-finite product measure, the Hammersley-Clifford theorem states that 
\begin{align}
    f(x_1,\ldots,x_d)
    = \prod_{C\in \mathcal{C}} \psi_C(x_C),
    \label{eq.3275r}
\end{align}
where $\mathcal{C}$ is the set of all cliques in the graph and $\psi_C$ are suitable functions called potentials \cite{BesagHammersleyCliffordTheorem, LauritzenGraphicalModels}. As we consider densities supported on $[0,1]^d,$ one can take as dominating product measure the uniform distribution on $(0,1)^d$ and the condition requires that the density is strictly positive on $(0,1)^d.$ There is no clear link between the potentials and marginal densities.

Assuming that the true density $f_0$ satisfies \eqref{eq.3275r} with largest clique size $r$ and all potentials having H\"older smoothness $\gamma,$ Lemma \ref{L: General potential product2} implies that, under the conditions of Theorem \ref{T: Convergence Rates}, the two-stage neural network density estimator is able to exploit the underlying low-dimensional structure and achieves the rate $n^{-2\gamma/(2\gamma +r)}\vee n^{-2\beta/d} ,$ up to $\log(n)$-factors.

\textbf{Bayesian networks:} Bayesian network models are widely used to model for instance medical expert systems \cite{koller2009probabilistic, heckerman1992toward} and causal relationships \cite{pearl2009causality}. As in the previous section, consider a $d$-dimensional random vector $(X_1,\ldots,X_d).$ In a Bayesian network, the dependence relationships of the variables are encoded in a directed acyclic graph with nodes $\{1,\ldots,d\}$ \cite{pearl2009causality,koller2009probabilistic,bishop2006PRML,korb2010bayesian}. A directed acyclic graph (DAG) is a directed graph that contains no cycles, meaning one cannot visit the same node twice by following a path along the direction of the edges. The parents $\pa(i)$ of a node $i$ are all nodes that have an edge pointing to node $i.$

The DAG underlying a Bayesian network is constructed such that each variable $X_i$ is conditionally independent of all other variables given the parents $X_{\pa(i)}:=\{X_j:j\in \pa(i)\}$ in the graph. The joint density can now be written as product of conditional densities
\begin{align}
    f(x_1,\ldots,x_d)
    = f_d\big(x_d\rvert x_{\pa(d)}\big)\cdot \ldots \cdot 
    f_1\big(x_1\rvert x_{\pa(1)}\big).
    \label{eq.DAG_decomp}
\end{align}
In particular, if $X_1, \ldots, X_d$ are generated from a Markov chain, the corresponding DAG is $X_1\to X_2 \to \ldots \to X_d.$ Thus $\pa(j)=\{j-1\}$ for $j>1,$ and we recover \eqref{eq.MC}. 

Assuming that the true density $f_0$ satisfies \eqref{eq.DAG_decomp}, that no node in the DAG has more than $r$ parents, and all conditional densities $f_d\big(x_i\rvert x_{\pa(i)}\big)$ have H\"older smoothness $\gamma,$ Lemma \ref{L: General potential product2} shows that, under the conditions of Theorem \ref{T: Convergence Rates}, the two-stage neural network estimator achieves the rate of convergence $n^{-2\gamma/(2\gamma +r)}\vee n^{-2\beta/d} $, up to $\log(n)$-factors.

\subsection{Copulas}\label{Sec: Copulas}
Copulas are widely employed to model dependencies between variables and to construct multivariate distributions, \cite{nelsen2007introduction, cherubini2004copula,czado2019analyzing}.
Denote by $F$ the multivariate distribution function, with marginals $F_1(x_1),\ldots,F_d(x_d)$ and density $f$. Sklar's theorem states that there exists a (unique) $d$-dimensional copula $C$ (a multivariate distribution function with uniformly distributed marginals on $[0,1]$) such that $F(\bx)=C(F_1(x_1),\ldots,F_d(x_d)).$ The density $f$ can then be rewritten by the chain rule as
\begin{equation}\label{Eq: general d dimensional Copula}
    f(\bx)=c\big(F_1(x_1),\ldots,F_d(x_d)\big)\prod_{i=1}^df_i(x_i),
\end{equation}
where $f_i(x_i)=F'_i(x_i)$ is the marginal density with respect to $x_i$ and $c$ is the density of $C$ (assuming that all these densities exist). For a reference, see Section 2.3 of \cite{nelsen2007introduction}.

\begin{Lemma} \label{L: general d dimensional Copula}
    Consider a density $f$ of the form \eqref{Eq: general d dimensional Copula}. If $c\in \mH_{d}^{\gamma_c}([0,1]^{d}, Q_c)$ and $f_i\in \mH_{1}^{\gamma_i}([0,1], Q_i),$ for $i=1,\ldots,d,$ then,  
    the density $f$ can be rewritten as a composition $g_2\circ g_1\circ g_0$ of the form \eqref{eq.comp_Hspace_def},
    with $(d_0,d_1,d_2)=(d,2d,d+1)$, $(t_0,t_1,t_2)=(1,d,d+1)$, $(\alpha_0,\alpha_1,\alpha_2)=(\min_{i=1,\ldots,d}\gamma_i,\gamma_c,\gamma)$, and $\gamma$ arbitrarily large. 
\end{Lemma}

Assume that the true density is of the form \eqref{Eq: general d dimensional Copula}, that all marginals $f_i$ have the same H\"{o}lder smoothness $\gamma_1=\ldots =\gamma_d,$ that $\beta=\gamma_c\wedge \gamma_1,$ and that all the conditions on the kernel and the network architecture underlying Theorem \ref{T: Convergence Rates} are satisfied. Applying the decomposition of the density in Lemma \ref{L: general d dimensional Copula}, Theorem \ref{T: Convergence Rates} yields the convergence rate
$n^{-2\gamma_1/(2\gamma_1+1)}\vee n^{-2\gamma_c/(2\gamma_c+d)}\vee n^{-2\beta/d},$ up to $\log(n)$-factors. When $\gamma_c/d\geq \gamma_1\geq (d-1)/2$, the convergence rate becomes $n^{-2\gamma_1/(2\gamma_1+1)}$ (up to $\log(n)$-factors). If the copula $c$ is smoother than the marginals, in the sense that $\gamma_c>\gamma_1=\beta$, then the obtained convergence rate is faster than the standard nonparametric rate $n^{-2\beta/(2\beta+d)}$ for estimation of $\beta$-H\"{o}lder smooth functions.

As example, consider the $d$-variate Farlie-Gumbel-Morgenstern copula family with parameter vector $\btheta$, which has copula density
\begin{equation*}
    c_{\btheta}(u_1,\ldots,u_d)=1+\sum_{r=2}^d\ \sum_{1\leq j_1<\cdots<j_r\leq d}\theta_{j_1\ldots j_r}\prod_{k=1}^r(1-2u_{j_k}),
\end{equation*}
for a parameter vector $\btheta$ satisfying
\begin{equation*}
    1+\sum_{r=2}^d \ \sum_{1\leq j_1<\cdots<j_r\leq d}\theta_{j_1\ldots j_r}\prod_{k=1}^r\xi_{j_k}\geq 0 \quad \text{for all} \ \xi_{j_k}\in\{-1,1\},
\end{equation*} \cite{johnson1977some,CorrelationDependenceCopula,CopulaIntroductionFGM}. The double summation sums over all $2^d-d-1$ subsets of $\{1,\ldots,d\}$ with at least two elements. Since the input of the copula comes from the distribution functions of the marginals, it holds that $(u_1,\ldots,u_d)\in[0,1]^d.$ This implies $v_j:=(1-2u_j)\in[-1,1],$ and by Lemma \ref{L: Holder class of product of different terms}, $\bv\mapsto\prod_{k=1}^rv_{j_k}\in \mH_{r}^{\gamma_c}([-1,1]^r,2^r)$, for all $\gamma_c\geq r+1$. Together with the chain rule, this yields $\bu\mapsto\prod_{k=1}^r(1-2u_{j_k})\in \mH_{r}^{\gamma_c}([-1,1]^r,4^r)$. The derivative of a sum is the sum of the derivatives and therefore the triangle inequality and $|\btheta|_{\infty}\leq 1$ imply for the copula density
$c_{\btheta}\in \mH_r^{\gamma_c}([-1,1]^d,(2^d-d)4^d),$ for all $\gamma_c\geq d+1.$ So for this family of copulas, the effective smoothness of the composition is determined by the smoothness of the marginals and the convergence rate becomes $n^{-2\gamma_1/(2\gamma_1+1)}\vee n^{-2\beta/d}.$

Explicit low-dimensional copula structures can be imposed using the fact that a $d$-dimensional copula density factorizes into a product of $d(d-1)/2$ bivariate (conditional) copula densities \cite{nagler2016evading, bedford2001probability,PairCopulaConstructionsMultipleDependence, VineCopulaBasedModeling}. The key ingredient in this argument is to successively rewrite the conditional densities using the formula $f_{X|Y}(x|y)=c_{X,Y}(F_X(x),F_Y(y))f_X(x),$ where $c_{X,Y}$ denotes the bivariate copula density of $(X,Y).$ The decomposition into bivariate copulas is not unique. Already for three variables $(X,Y,Z),$ there are two possible decompositions, namely
\begin{equation*}
    f_{X\rvert Y,Z}(x\rvert y,z)=c_{X,Y\rvert Z}\big(F_{X\rvert Z}(x\rvert z),F_{Y\rvert Z}(y\rvert z) \, \rvert \, z\big)f_{X|Z}(x\rvert z)
\end{equation*}
and a second decomposition that interchanges the roles of $y$ and $z.$
The so-called simplifying assumption \cite{stoeber2013simplified,nagler2016evading, VineCopulaBasedModeling} states that all the bivariate copulas in the decomposition are independent of the conditioned variables, in other words 
\begin{equation*}
    c_{i,j\rvert k}(F_{i\rvert k}(x_i\rvert x_k),F_{j\rvert k}(x_j\rvert x_k)\rvert x_k)=c_{i,j\rvert k}\big(F_{i\rvert k}(x_i\rvert x_k),F_{j\rvert k}(x_j\rvert x_k)\big).
\end{equation*}
For the remainder of this section, we will assume that the simplifying assumption holds.

A way to define such decompositions is by relying on regular vines, \cite{nagler2016evading, bedford2001probability,PairCopulaConstructionsMultipleDependence, VineCopulaBasedModeling}. A vine on $d$ variables $X_1,\ldots,X_d$ is a set of trees $(T_1,\ldots,T_r)$, such that the nodes of the first tree $T_1$ are $u_1,\ldots,u_d$. The nodes of the tree $T_i$, for $i=2,\ldots,r$, are (a subset of) the edges of the tree $T_{i-1}$. For a regular vine it furthermore holds that $r=d-1$, that two edges in a tree can only be joined by an edge in the next tree if these edges share a common node, and that the set of nodes of $T_i$ has to be equal to the set of edges of $T_{i-1}$.

Any regular vine on $(X_1,\ldots,X_d)$  defines a factorization of a $d$-dimensional copula, by associating a bivariate copula density to each edge in any of the trees. Copulas defined in this way are called vine-copulas.

Figure \ref{fig:VineExampleDim4} shows an example of a regular vine with four variables. Regular vines such as the one in Figure \ref{fig:VineExampleDim4}, where each tree has one node that has an edge to all other nodes in that tree, are known as canonical-vines \cite{PairCopulaConstructionsMultipleDependence} or C-vines \cite{VineCopulaBasedModeling}. 
The density corresponding to a canonical vine on $d$ variables (up to renumbering the variables) is given by
\begin{equation*}
    \prod_{k=1}^df_{k}(x_k)\prod_{j=1}^{d-1}\prod_{i=1}^{d-j}c_{j,j+i|1,\ldots,j-1}\big(F(x_j|x_1,\ldots,x_{j-1}),F(x_{j+i}|x_{1},\ldots,x_{j-1})\big).
\end{equation*}
Another type of regular vine is the D-vine, \cite{PairCopulaConstructionsMultipleDependence,VineCopulaBasedModeling}. In a D-vine no node in any tree is connected to more than two edges. Figure \ref{fig:CopulaStructureSimulation} shows the first tree of a D-vine on $d$ variables. The density corresponding to a D-vine on $d$ variables (up to renumbering the variables) is given by 
\begin{equation*}
    \prod_{k=1}^df_k(x_k)\prod_{j=1}^{d-1}\prod_{i=1}^{d-j}c_{i,i+j|i+1,\ldots,i+j-1}\big(F(x_i|x_{i+1},\ldots,x_{i+j-1}),F(x_{i+j}|x_{i+1},\ldots,x_{i+j-1})\big).
\end{equation*}

\begin{figure}[!ht]
\begin{center}
\begin{subfigure}[t]{0.32\textwidth}
\includegraphics[width=0.75\linewidth]{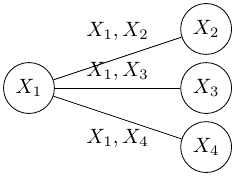}
\caption{First tree}
\label{fig:VineTree1}
\end{subfigure}
\begin{subfigure}[t]{0.32\textwidth}
\includegraphics[width=0.95\linewidth]{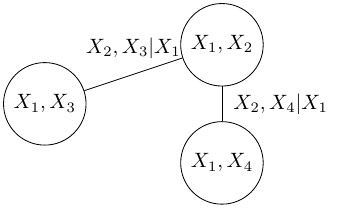}
\caption{Second tree}
\label{fig:VineTree2}
\end{subfigure} 
\begin{subfigure}[t]{0.32\textwidth}
\includegraphics[width=0.95\linewidth]{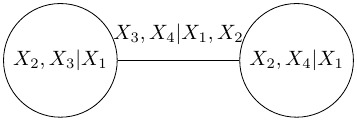}
\caption{Third tree}
\label{fig:VineTree3}
\end{subfigure}
\end{center}
\caption{Example of a regular vine on four variables. Another example is given in Figure \ref{fig:CopulaStructureSimulation}.}
\label{fig:VineExampleDim4}
\end{figure}

 If two random variables $X_1,X_2$ are conditionally independent given $X_3$, then $c_{1,2|3}=1$. If such conditional independence relations hold, one can simplify the vine-structure. For example consider the vine on four variables in Figure \ref{fig:VineExampleDim4}. In the (very simplified) case that $X_2$ and $X_3$ are independent given $X_1$, that $X_2$ and $X_4$ are independent given $X_1$, and that $X_3$ and $X_4$ are independent given $(X_1,X_2)$, only the bivariate copulas on the edges of the first tree (Figure \ref{fig:VineTree1}) appear in the decomposition, cf. Section 3 of \cite{PairCopulaConstructionsMultipleDependence}. More generally, suppose that there exists a canonical vine on $d$ variables such that the bivariate (conditional) copulas associated with all the trees except the first one are equal to one, then under the simplifying assumption, the decomposition becomes 
\begin{equation}\label{Eq: Pair Copula decomposition canonical vine}
    f(\bx)=\prod_{k=1}^df_k(x_k)\prod_{i=2}^dc_{1,i}\big(F(x_1),F(x_i)\big).
\end{equation}
Here we use that $X_1$ is the root of the first tree, which can always be achieved by renumbering the variables. In the case of a D-vine the decomposition (up to renumbering) becomes
\begin{equation}\label{Eq: Pair Copula decomposition d vine}
    f(\bx)=\prod_{k=1}^df_k(x_k)\prod_{i=1}^{d-1}c_{i,i+1}\big(F(x_i),F(x_{i+1})\big).
\end{equation}
Vine copulas where the bivariate copulas associated with all trees except the first one are equal to the independence copula can be interpreted as Markov tree models \cite{TruncatedRegularVines,kirshner2007learning}.

\begin{Lemma}\label{L: Pair Copula composition}
    Consider a density $f$ of the form \eqref{Eq: Pair Copula decomposition canonical vine} or \eqref{Eq: Pair Copula decomposition d vine}. If $f_i\in \mH_{1}^{\gamma}([0,1], Q),$ for all $i=1,\ldots,d,$ and all bivariate copulas are in $\mH_{2}^{\gamma_c}([0,1]^{2}, Q)$, then,  
    the function $f$ can be written as a composition $g_2\circ g_1\circ g_0$, with 
    $(d_0,d_1,d_2)=(d,2d,2d-1)$, $(t_0,t_1,t_2)=(1,2,2d-1)$, $(\alpha_0,\alpha_1,\alpha_2)=(\gamma,\gamma_c,\zeta)$, where $\zeta$ is arbitrarily large. 
\end{Lemma}

If we assume that $\gamma_c=\gamma=\beta$, then under the combined conditions of Theorem \ref{T: Convergence Rates} and Lemma \ref{L: Pair Copula composition}, the proposed two-stage neural network estimator achieves the convergence rate
$n^{-2\gamma/(2\gamma+2)}\vee n^{-2\gamma/d}$ up to $\log(n)$ factors. If $d>2$, this rate is faster than the nonparametric estimation rate $n^{-2\gamma/(2\gamma+d)}$.
Furthermore when $\gamma=\gamma_c\geq d/2-1,$ the rate equals $n^{-2\gamma/(2\gamma+2)}$, up to $\log(n)$-factors. If instead we assume that $\gamma_c\geq 2\gamma=2\beta$, that is, the copulas have at least twice the H\"{o}lder smoothness of the marginals, then the rate becomes $n^{-2\gamma/(2\gamma+1)}\vee n^{-2\gamma/d},$ up to $\log(n)$-factors.

\subsection{Mixture distributions}

If the true density is a mixture and all mixture components can be estimated by a fast convergence rate, it should be possible to also estimate the true density with a fast rate. Below we make this precise, assuming that the true density is a mixture density of the form 
\begin{align}\label{Eq: Mixture model}
    f_0=a_1f_1+\ldots+a_r f_r
\end{align}
with non-negative mixture weights $a_1,\ldots a_r$ summing up to one and densities $f_j$ in the compositional H\"older space $\mG(q_j,\bd_j,\bt_j,\balpha_j,Q')$ defined in \eqref{eq.comp_Hspace_def}. In particular, we allow the parameters $q_j,$ $\bd_j=(d_{0,j},\ldots,d_{q+1,j}),$ $\bt_j=(t_{0,j},\ldots,t_{q,j}),$ and $\balpha_j=(\alpha_{0,j},\ldots,\alpha_{q,j})$ to depend on $j.$ Compositional spaces are not closed under linear combinations and therefore there is no natural embedding of $f$ into the compositional spaces of the $f_j$'s.
As shown next, the convergence rate for estimation of $f$ still coincides with the maximum among all convergence rates for estimation of individual mixture components $f_j.$ Set $\alpha_{i,j}^{*}:=\alpha_{i,j}\prod_{\ell=i+1}^q(\alpha_{\ell,j}\wedge 1)$ and $\phi_n^{\star}:=\max_{j=1,\ldots,r} \phi_{n,j}$, where $$\phi_{n,j} := \max_{i=0,\ldots,q} n^{-\frac{2\alpha^*_{i,j}}{2\alpha^*_{i,j}+t_{i,j}}}$$ is the rate \eqref{eq.phin_def} for estimation of $f_j.$

\begin{Theorem}[Convergence rates for mixture distributions]\label{T: Mixture convergence rates}
Consider the density estimation model defined by \eqref{eq.KDE_def}-\eqref{eq.regr_model} with density $f_0=\sum_{i=1}^ra_if_i$, where $a_1,\ldots a_r$ are non-negative mixture weights summing up to one, and with $f_j\in \mathcal{C}_d^{\beta}([0,1]^d,Q)\cap\mG(q_j,\bd_j,\bt_j,\balpha_j,Q),$ for all $j=1,\ldots,r$. Let $\wh f$ be a two-stage neural network density estimator with kernel of order $\floor{\beta}$ as defined in Definition \ref{eq.defi_estimator}
for the neural network class $\mF(L,(p_0,\ldots, p_{L+1}),s,F)$ with parameters satisfying
\begin{compactitem}
    \item[(i)] $F\geq \max\{Q,1\},$
    \item[(ii)]  $\max_{j=1,\ldots,r}\sum_{i=1}^{q_j}\tfrac{\alpha_{i,j}+t_{i,j}}{2\alpha_{i,j}^*+t_{i,j}}\log_2(4t_{i,j}\vee 4\alpha_{i,j})\log_2(n)\leq L \lesssim n\phi_n^{\star},$
\item[(iii)] $  n\phi_n^{\star}\lesssim \min_{i=1,\hdots,L}p_i,$
    \item[(iv)] $s\asymp  n\phi_n^{\star}\log(n).$
\end{compactitem}
If $n$ is large enough, then there exists a constant $C_6,$ only depending on $r,(q_j,\bd_j,\bt_j,\balpha_j)_{j=1}^r,$ $F,$ $\beta,$ $K$ and the implicit constants in (ii), (iii), and (iv) such that
\begin{equation*}
    R(\widehat{f}_n,f_0)\leq C_6 L\max\Big(\phi_n^{\star}\log^{4}(n), n^{-\frac{2\beta}{d}}\Big)+6\Delta_n(\widehat{f}_n,f_0).
\end{equation*}
\end{Theorem}

\section{Simulations}\label{S: Simulations}

\subsection{Methods}
In a numerical simulation study we compare the proposed two-stage neural network density estimator (named SD for Split Data) as described in Definition \ref{eq.defi_estimator} to two other methods. The FD (full data) method follows the same construction as the two-stage neural network estimator but uses for both steps the full dataset without sample splitting. Thus, we have twice as many data for the individual steps, but also incur additional dependence between the regression variables as each of the constructed response variables $Y_i$ depends on the entire dataset (instead of only on the kernel dataset and the corresponding $X_i$ from the regression set). The neural network based methods are moreover compared to a multivariate kernel density estimator (KDE).

As suggested by the theory, for the first step in the SD and FD method, the bandwidths for the kernel density estimator are chosen of the form $c_1(\log(n)/n)^{1/d}$ and $c_2(\log(2n)/2n)^{1/d}.$ For the KDE method, the bandwidth is $c_3n^{-1/(2\beta+d)}.$ The constants $c_1,c_2,c_3$ are determined based on the average of the optimal bandwidths found by $50$-fold cross-validation, taking as search space the interval $[0.05,1.1]$ with stepsize $0.005$, on five independently generated datasets with sample size $n=200$ from the true density. Taking $n=200$ for the calibration is natural as it is the smallest sample size in the simulation environment.

\subsection{Densities}
\label{sec.densities}

For the different simulation settings, we generate data from five densities. These densities are called Naive Bayes mixing (NBm), Naive Bayes shifting (NBs), Binary Tree mixing (BTm), Binary Tree shifting (BTs) and Copula (C).

\subsubsection{NBm, NBs, BTm, BTs}

The densities (NBm) and (NBs) are so-called Naive Bayes networks \cite{koller2009probabilistic} with DAGs displayed in Figure \ref{fig:NB DAG} and density factorization
\begin{align}\label{Eq: NB factorization}
    f(x_1,\ldots,x_d)
    =f_d(x_d\rvert x_1)\cdot \ldots \cdot f_2(x_2 \rvert x_1) f_1(x_1).
\end{align}
The densities (BTm) and (BTs) are Bayesian networks with DAGs displayed in Figure \ref{fig:BT DAG} and density factorization
\begin{equation}\label{Eq: BT factorization}
    f(x_1,\ldots,x_d)=f_1(x_1)\prod_{j=2}^df_j(x_j\rvert x_{\ceil{(j-1)/2}}).
\end{equation}

\begin{figure}[!ht]
\begin{center}
\begin{subfigure}[t]{0.45\textwidth}
\includegraphics[width=0.95\linewidth]{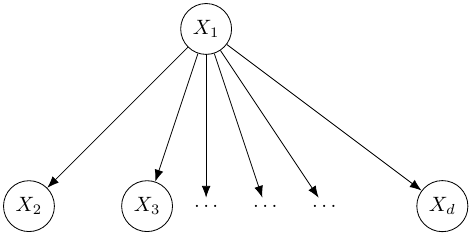}
\caption{Naive Bayes DAG}
\label{fig:NB DAG}
\end{subfigure}
\begin{subfigure}[t]{0.45\textwidth}
\includegraphics[width=0.95\linewidth]{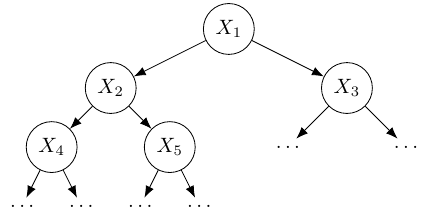}
\caption{Binary tree DAG}
\label{fig:BT DAG}
\end{subfigure} 
\end{center}
\caption{DAG for the Naive Bayes network (a) and the Bayesian network with binary tree structure (b).}
\label{fig:Simulation DAGs}
\end{figure}

For the density $f_{1},$ we use the exponential of a standard Brownian motion on $[0,1]$, normalized such that $f_1$ integrates to one. We use two different types of conditional densities. The mixing conditional density has mixture weights from the conditioned variable, 
\begin{align}\label{Eq: Mixing conditional density}
    f_j(x_j|x_i) = x_i h_j(x_j)+(1-x_i)h_j(1-x_j),
\end{align}
with $h_j$ a density supported on $[0,1]$.
The shifting conditional density incorporates a shift determined by the conditioned variable,
\begin{align}\label{Eq: Shifting conditional density}
    f_j(x_j|x_i) = h_j\big(\max\{x_j-x_i/4,0\}\big),
\end{align}
with $h_j$ a density supported on the interval $[0,3/4],$ so that the support of $f_j(\cdot|x_i)$ is ensured to lie in $[0,1]$.

For the densities (NBm) and (BTm) all conditional densities $f_j(\cdot|\cdot)$ in the factorization are mixing densities \eqref{Eq: Mixing conditional density}. For the densities (NBs) and (BTs) the conditional densities $f_j(\cdot|\cdot)$ in the factorization are shifting densities \eqref{Eq: Shifting conditional density} if $j$ is divisible by $3$ and mixing densities \eqref{Eq: Mixing conditional density} otherwise.

It remains to choose the density $h_j$ in \eqref{Eq: Mixing conditional density} and \eqref{Eq: Shifting conditional density}. We consider scenarios containing both smooth and rough densities. For (NBm), (NBs), (BTm) and (BTs) and all $j$ such that $j-1$ is not divisible by $3$, we set
\begin{align}
    h_j(x)=\Big (1-\frac {2x-1}d\Big)\mathbf{1}\big(0\leq x\leq 1\big).
    \label{eq.36427}
\end{align}
Viewed as functions on $[0,1]$, these densities have arbitrarily large H\"{o}lder smoothness. The densities take values between $1-1/d$ and $1+1/d$ ensuring that in higher dimensions the joint densities, which are products, neither become extremely small or large.

For (NBm) and (BTm) and all $j>1$ such that $j-1$ is divisible by $3$, we take as densities $h_j$ the exponential of the Brownian motion on $[0,1],$ normalized such that $h_j$ integrates to one. Brownian motion has H\"{o}lder smoothness $1/2-\eta$ for any $\eta\in(0,1/2)$, but is almost surely not $1/2$-H\"{o}lder smooth \cite{BrownianMotion}. This means that these densities have low regularity.

 For (NBs) and (BTs) and all $j>1$ such that $j-1$ is divisible by $3$, we take as densities $h_j$ the paths of the exponential of the Brownian motion on $[0,1]$ multiplied with the function $x\mapsto \rho(x)=\max(0,(4x/3)(1-4x/3))$ and normalized such that $h_j$ integrates to one. Multiplication with $\rho$ ensures that the support of these densities is in $[0,3/4],$ as required in the definition \eqref{Eq: Shifting conditional density}.

 The conditional densities $f_j$ defined in \eqref{Eq: Mixing conditional density} and \eqref{Eq: Shifting conditional density} can be interpreted as compositional functions.
\begin{Lemma}\label{L: Structure mixing conditional}
    Consider the mixing conditional density $f_j$ in \eqref{Eq: Mixing conditional density}. If $h_j\in \mH_1^{\gamma_j}([0,1],Q)$, then $f_j$ can be written as the composition $g_1\circ g_0,$ with $(d_0,d_1)=(2,3)$, $(t_0,t_1)= (1,3)$, and $(\alpha_0,\alpha_1)=(\gamma_j,\zeta)$, with $\zeta$ arbitrarily large.
\end{Lemma}

\begin{Lemma}\label{L: Structure shifting conditional}
Consider the shifting conditional density $f_j$ in \eqref{Eq: Shifting conditional density}
If $h_j\in \mH^{\gamma_j}_1([0,3/4],Q)$, then $f_j$ can be written as $g_1\circ g_0,$ with $(d_0,d_1)=(2,1)$, $(t_0,t_1)=(2,1)$, $(\alpha_0,\alpha_1)=(1,\gamma_j).$
\end{Lemma}

The (NBm), (NBs), (BTm) and (BTs) joint densities are thus compositions where the components with low regularity are all univariate functions, making the rate $\phi_n$ dimensionless. The factorization in \eqref{Eq: NB factorization} and the composition of Lemma \ref{L: General potential product2} combined with the composition in Lemma \ref{L: Structure mixing conditional} shows this for the (NBm) model. The factorization in \eqref{Eq: NB factorization} and the composition Lemma \ref{L: General potential product2} combined with the compositions in Lemma \ref{L: Structure mixing conditional} and Lemma \ref{L: Structure shifting conditional} show this for the (NBs) model. The factorization in \eqref{Eq: BT factorization} and the composition of Lemma \ref{L: General potential product2} combined with Lemma \ref{L: Structure mixing conditional} shows this for the (BTm) model and the factorization in \eqref{Eq: BT factorization} and the composition of Lemma \ref{L: General potential product2} combined with the compositions in Lemma \ref{L: Structure mixing conditional} and Lemma \ref{L: Structure shifting conditional} show this for the (BTs) model.

\subsubsection{Simulation setup for copula density model}

For the copula model, the density (C) is associated to a D-vine copula. The first tree of the D-vine is depicted in Figure \ref{fig:CopulaStructureSimulation}. We assume that this first tree captures all the dependencies between the variables. This means that $X_i$ is conditionally independent of $X_{i+j}$ given $X_{i+1},\ldots,X_{i+j-1}$ for all pairs $(j,i)$ with $j\in \{2,\ldots,d-1\}$ and $i\in\{1,\ldots,d-j\}.$

\begin{figure}[!ht]
\begin{center}
   \includegraphics{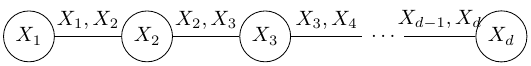}
    \caption{Structure of the first tree of the D-vine copula used in the simulation.}
    \label{fig:CopulaStructureSimulation}
\end{center}
\end{figure}

The bivariate copulas for the density (C) are chosen from the bivariate Farlie-Gumbel-Morgenstern copula family defined via the copula densities $c_{\theta,i,j}(F_i(x_i),F_j(x_j))=1+\theta(1-2F_i(x_i))(1-2F_j(x_j))$, with parameter $|\theta|\leq 1$. As shown in Section \ref{Sec: Copulas}, these copulas have arbitrarily large H\"{o}lder smoothness. If $(i-1)/(d-2)\neq 1/2,$ we use the parameter $\theta=-1+2(i-1)/(d-2)$ for the bivariate copula. Otherwise we use $\theta=1/100.$ The marginal densities are displayed in Figure \ref{fig: MarginalCopula}. \begin{figure}[ht]
\begin{minipage}{0.6\textwidth}
    \begin{equation*}
  f_k(x)=\begin{cases}
    1+\frac{1}{2d}-\frac{1}{d}\sqrt{\frac{1}{4}-x}, & \text{ if } 0\leq x<\frac{1}{4}\\
    1+\frac{1}{2d}-\frac{1}{d}\sqrt{x-\frac{1}{4}}, & \text{ if } \frac{1}{4}\leq x<\frac{1}{2}\\
    1-\frac{1}{2d}+\frac{1}{d}\sqrt{\frac{3}{4}-x}, & \text{ if } \frac{1}{2}\leq x<\frac{3}{4}\\
    1-\frac{1}{2d}+\frac{1}{d}\sqrt{x-\frac{3}{4}}, & \text{ if } \frac{3}{4}\leq x\leq 1.\\
  \end{cases}
    \end{equation*}
\end{minipage}
\begin{minipage}{0.4\textwidth}
\centering
    \includegraphics[width=0.95\linewidth]{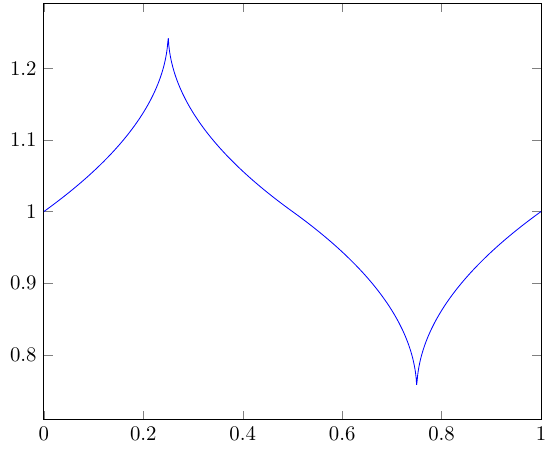}
\end{minipage}
    \caption{Marginal density $f_k(x)$ used in the simulated copula model.}
    \label{fig: MarginalCopula}
\end{figure}
The smoothness of this density is determined by the square root, which has H\"older smoothness $1/2$. The right panel of Figure \ref{fig: MarginalCopula} displays the graph for $d=2.$ This marginal density is appealing as it has a closed-form expression for the density and the c.d.f. The dependency on $d$ of the marginals is to ensure that the marginal densities remain between $1-1/d$ and $1+1/d$ in order to prevent numerical instability.
Since the Farlie-Gumbel-Morgenstern copula is infinitely smooth, we get from Lemma \ref{L: Pair Copula composition} that the effective smoothness of the joint density generated from this vine-copula approach is equal to $1/2$ and thus the rate $\phi_n$ in Theorem \ref{T: Convergence Rates} becomes $n^{-1/2},$ up to $\log(n)$-factors.

\subsection{Neural network training setup}
For both the SD and FD method, we train neural networks with width vector $\bp=(d,\ceil{(2n)^{1/2}},\ceil{(2n)^{1/2}},\ldots,\ceil{(2n)^{1/2}},1)$ and depth $L=\ceil{\log_2(2n)}$. Since the derived convergence rate of the two-stage neural network estimator is $\phi_n=n^{\eta'-1/2}$, for any $\eta'\in(0,1/2)$, in the (NBm), (NBs), (BTm) and (BTs) settings, and $\phi_n=n^{-1/2}$ in the (C) setting, this choice of the network width satisfies the bound in Theorem \ref{T: Convergence Rates}. The chosen depth is of the order $\log(n)$ suggested by the theory, but there might be a mismatch regarding the constants in the lower bound of Condition (ii) in Theorem \ref{T: Convergence Rates}. Since the proof of this result does not optimize the constants, we find it more appealing to work with the generic choice $L=\ceil{\log_2(2n)}$ in the simulations.
Furthermore, Theorem \ref{T: Convergence Rates} imposes a sparsity condition on the networks as well as a condition on the maximum norm of the parameters. In the simulation study we use $\ell_2$-penalization on the weight matrices and the Glorot uniform initialization \cite{glorot2010understanding} to ensure that the parameter values do not become too large. Although these methods do not provide a hard guarantee that the condition on the maximum norm is satisfied, they work reasonably well in practice and the number of learned network parameters exceeding in absolute value one is small compared to the total number of network parameters. We use pruning (using the TensorFlow model optimization package) to enforce sparsity. The fraction of zero network parameters is chosen as $1-2m\log(m)\phi_m/p$, with $p$ the total number of network parameters and $m=2n$ for the FD method and $m=n$ for the SD method.

The source code is available on GitHub \cite{Bos_Simulation-code_A_supervised_2023}.

\subsection{Simulation Results}

For each of the five densities described in Section \ref{sec.densities}, we generate four training samples, with respective sample sizes $200, 1000, 5000, 25000$. For both the SD and FD method, $50$  neural networks are trained with different random initialization on each training sample. Repeating the network fit on the same sample highlights the variation of test performance with respect to the initialization and the achieved training loss. We compare the performance of all the methods on $10^6$ test samples. This sample is only used for computing the test error and none of the methods has access to the test samples during training.
 Figures \ref{fig:NaiveBayes}-\ref{fig:Copula} report the test errors for the five different settings.

 \begin{figure}[!ht]
\begin{center}
\begin{subfigure}[t]{0.49\textwidth}
\includegraphics[width=0.95\linewidth]{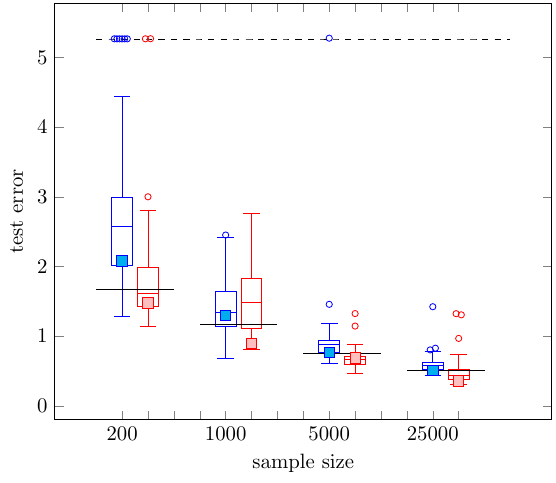}
\caption{NBs, dimension 4}
\label{fig:NBs4}
\end{subfigure}
\begin{subfigure}[t]{0.49\textwidth}
\includegraphics[width=0.95\linewidth]{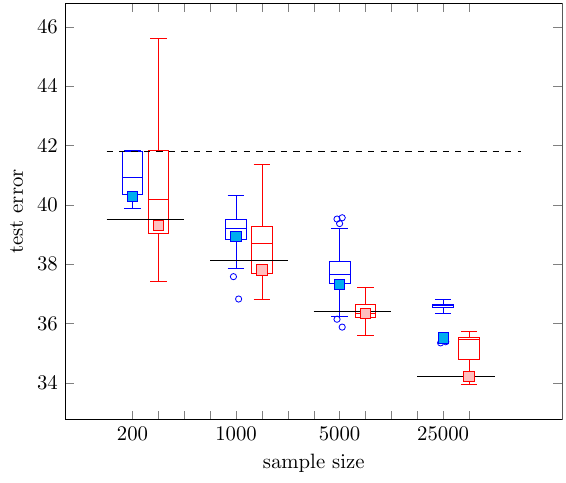}
\caption{NBs, dimension 12}
\label{fig:NBs12}
\end{subfigure} 
\begin{subfigure}[t]{0.49\textwidth}
\includegraphics[width=0.95\linewidth]{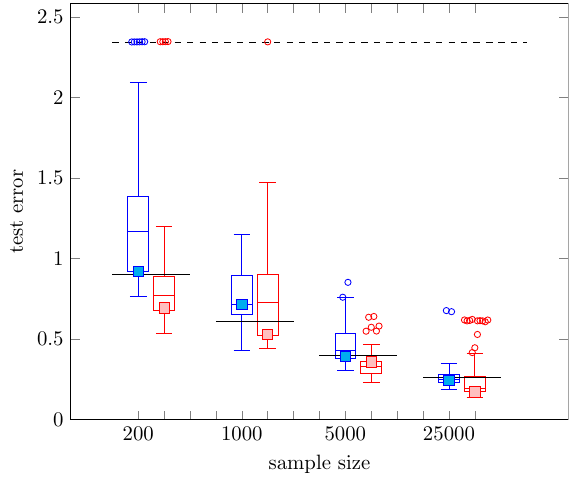}
\caption{NBm, dimension 4}
\label{fig:NBm4}
\end{subfigure}
\begin{subfigure}[t]{0.49\textwidth}
\includegraphics[width=0.95\linewidth]{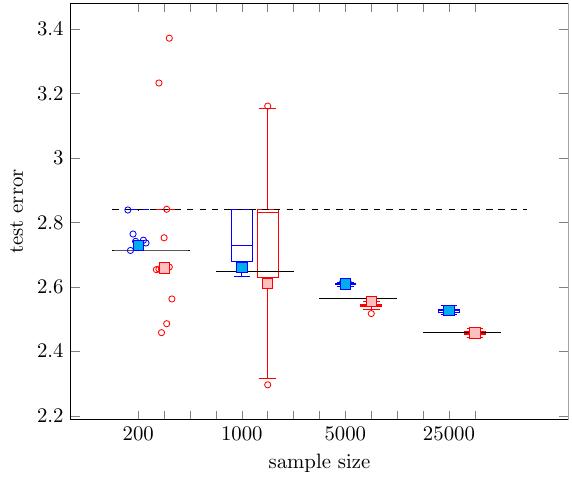}
\caption{NBm, dimension 12}
\label{fig:NBm12}
\end{subfigure} 
\end{center}
\caption{Test errors for the naive Bayes model. SD in blue, FD in red, KDE black bars. The test error of the network with the lowest training error is indicated by the filled square. The black dashed line is the test error of the zero function. Notice that in the individual plots, the $y$-axis has different starting points.
}
\label{fig:NaiveBayes}
\end{figure}

\begin{figure}[!ht]
\begin{center}
\begin{subfigure}[t]{0.49\textwidth}
\includegraphics[width=0.95\linewidth]{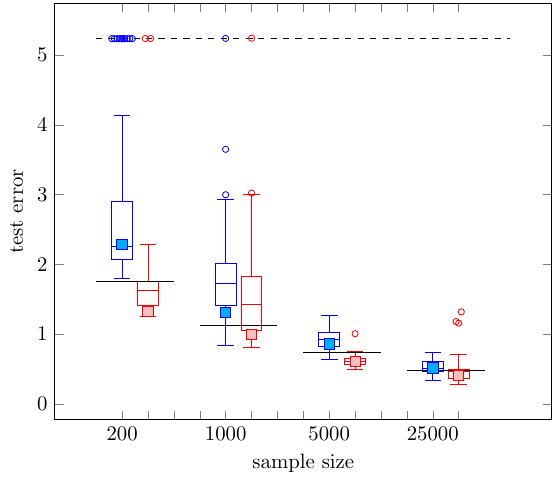}
\caption{BTs, dimension 4}
\label{fig:BTs4}
\end{subfigure}
\begin{subfigure}[t]{0.49\textwidth}
\includegraphics[width=0.95\linewidth]{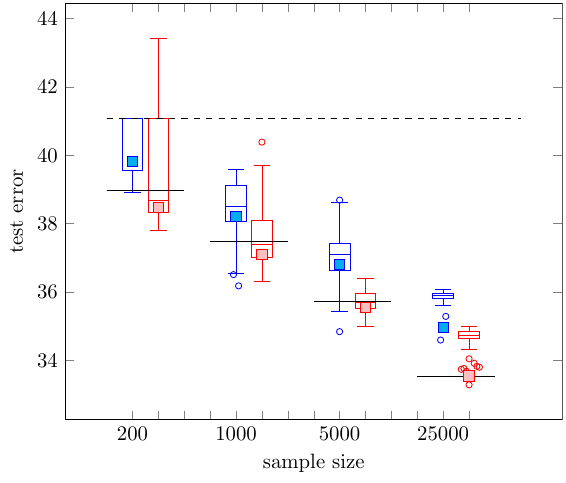}
\caption{BTs, dimension 12}
\label{fig:BTs12}
\end{subfigure} 
\begin{subfigure}[t]{0.49\textwidth}
\includegraphics[width=0.95\linewidth]{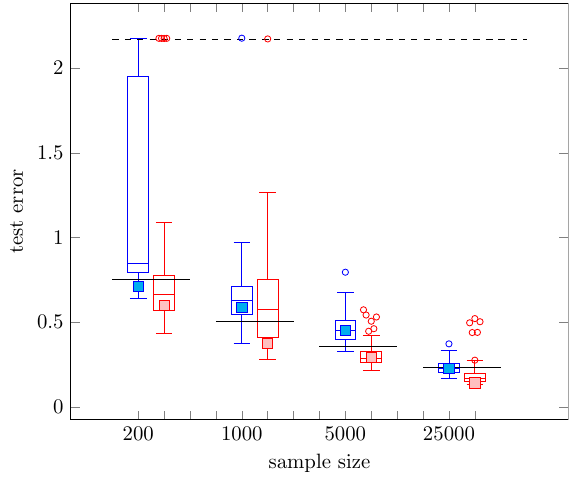}
\caption{BTm, dimension 4}
\label{fig:BTm4}
\end{subfigure}
\begin{subfigure}[t]{0.49\textwidth}
\includegraphics[width=0.95\linewidth]{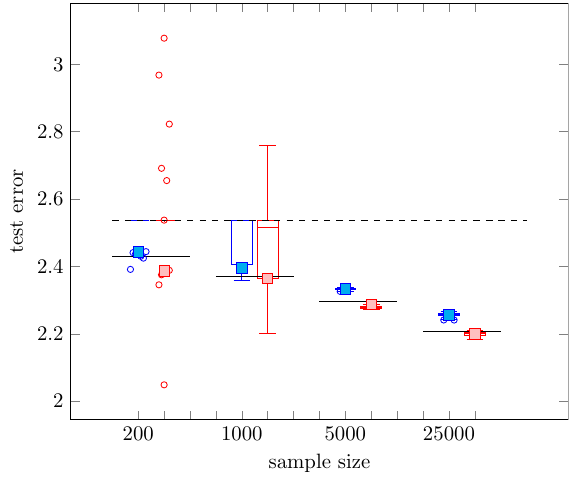}
\caption{BTm, dimension 12}
\label{fig:BTm12}
\end{subfigure} 
\end{center}
\caption{Test errors for the Bayesian network model. SD in blue, FD in red, KDE black bars. The test error of the network with the lowest training error is indicated by the filled square. The black dashed line is the test error of the zero function. Notice that in the individual plots, the $y$-axis has different starting points.}
\label{fig:BayesianNetwork}
\end{figure}

\begin{figure}[!ht]
\begin{center}
\begin{subfigure}[t]{0.49\textwidth}
\includegraphics[width=0.95\linewidth]{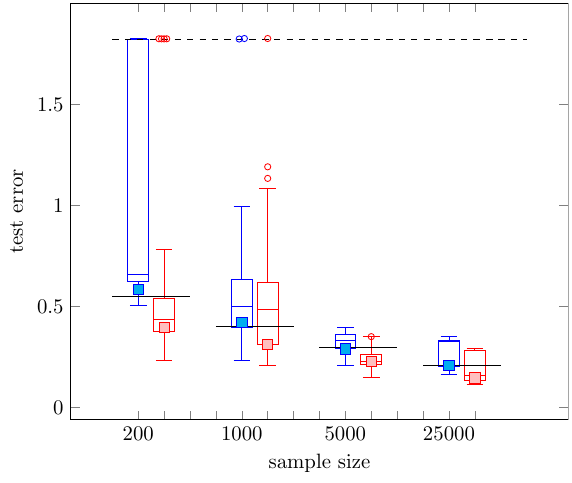}
\caption{Copula, dimension 4}
\label{fig:Copula4}
\end{subfigure}
\begin{subfigure}[t]{0.49\textwidth}
\includegraphics[width=0.95\linewidth]{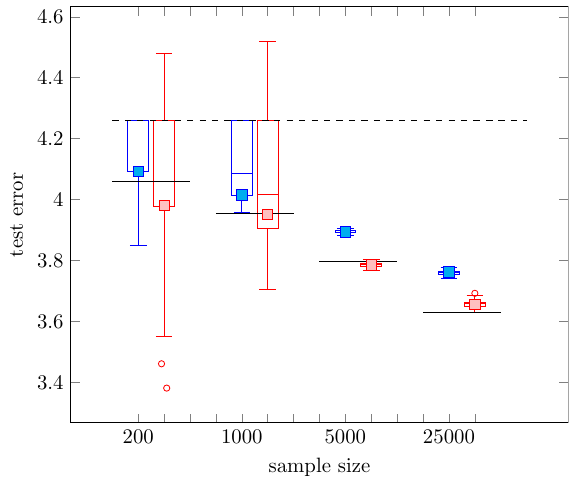}
\caption{Copula, dimension 12}
\label{fig:Copula12}
\end{subfigure} 
\end{center}
\caption{Test errors for the Copula model. SD in blue, FD in red, KDE black bars. The test error of the network with the lowest training error is indicated by the filled square. The black dashed line is the test error of the zero function. Notice that in the individual plots, the $y$-axis has different starting points.}
\label{fig:Copula}
\end{figure}

For the smaller sample sizes, the neural network fit is sometimes the zero function. These reconstructions generate the circles on top of the dashed lines in the plots. The theory claims that among the sparsely connected neural networks that satisfy all the imposed conditions, the one with small training error should perform particularly well. To see whether there is an effect, we mark for every simulation setting the test error of the network with the smallest training error by a filled square. The simulations show that for the FD method, this network fit is often near the first quartile in the box plots and thus indeed performs particularly well among the different random initializations.

To further investigate the relation between training error and test error, we plot for the (NBs) model in dimension four (Figure \ref{fig:ScatterNBs4}) and twelve  (Figure \ref{fig:ScatterNBs12}) the training error versus the test error of all networks, for both the SD and FD method and for each of the four considered sample sizes. The linear line displaying the least squares regression fit has positive slope, except for the SD method with sample size 1000 (in both dimensions four and twelve).

\begin{figure}[!ht]
\begin{center}
\begin{subfigure}[t]{0.24\textwidth}
\includegraphics[width=0.95\linewidth]{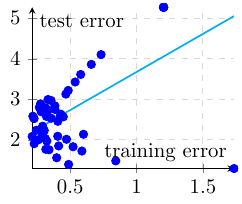}
\caption{SD method, sample size $200$}
\label{fig:ScatterNBs4SD200}
\end{subfigure}
\begin{subfigure}[t]{0.24\textwidth}
\includegraphics[width=0.95\linewidth]{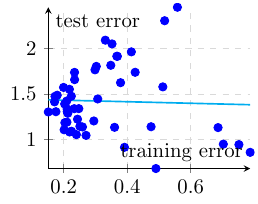}
\caption{SD method, sample size $1000$}
\label{fig:ScatterNBs4SD1000}
\end{subfigure}
\begin{subfigure}[t]{0.24\textwidth}
\includegraphics[width=0.95\linewidth]{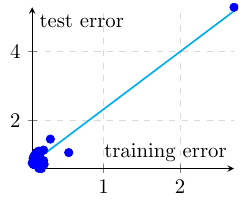}
\caption{SD method, sample size $5000$}
\label{fig:ScatterNBs4SD5000}
\end{subfigure}
\begin{subfigure}[t]{0.24\textwidth}
\includegraphics[width=0.95\linewidth]{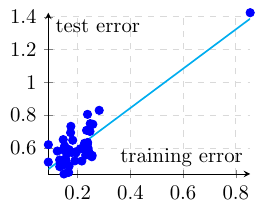}
\caption{SD method, sample size $25000$}
\label{fig:ScatterNBs4SD25000}
\end{subfigure}
\begin{subfigure}[t]{0.24\textwidth}
\includegraphics[width=0.95\linewidth]{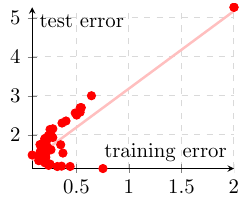}
\caption{FD method, sample size $200$}
\label{fig:ScatterNBs4FD200}
\end{subfigure} 
\begin{subfigure}[t]{0.24\textwidth}
\includegraphics[width=0.95\linewidth]{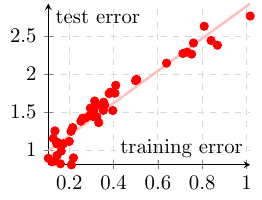}
\caption{FD method, sample size $1000$}
\label{fig:ScatterNBs4FD1000}
\end{subfigure}
\begin{subfigure}[t]{0.24\textwidth}
\includegraphics[width=0.95\linewidth]{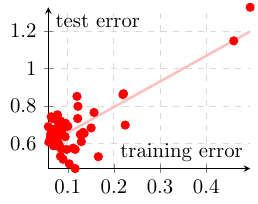}
\caption{FD method, sample size $5000$}
\label{fig:ScatterNBs4FD5000}
\end{subfigure}
\begin{subfigure}[t]{0.24\textwidth}
\includegraphics[width=0.95\linewidth]{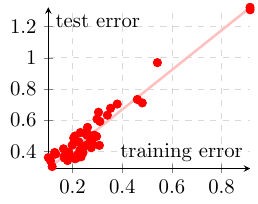}
\caption{FD method, sample size $25000$}
\label{fig:ScatterNBs4FD25000}
\end{subfigure} 
\end{center}
\caption{Scatterplot of the test error versus the training error for the (NBs) model in 4 dimensions. The line shows the linear least squares regression fit.
}
\label{fig:ScatterNBs4}
\end{figure}

\begin{figure}[!ht]
\begin{center}
\begin{subfigure}[t]{0.24\textwidth}
\includegraphics[width=0.95\linewidth]{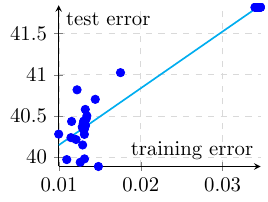}
\caption{SD method, sample size $200$}
\label{fig:ScatterNBs12SD200}
\end{subfigure}
\begin{subfigure}[t]{0.24\textwidth}
\includegraphics[width=0.95\linewidth]{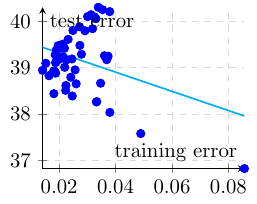}
\caption{SD method, sample size $1000$}
\label{fig:ScatterNBs12SD1000}
\end{subfigure}
\begin{subfigure}[t]{0.24\textwidth}
\includegraphics[width=0.95\linewidth]{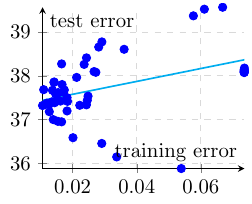}
\caption{SD method, sample size $5000$}
\label{fig:ScatterNBs12SD5000}
\end{subfigure}
\begin{subfigure}[t]{0.24\textwidth}
\includegraphics[width=0.95\linewidth]{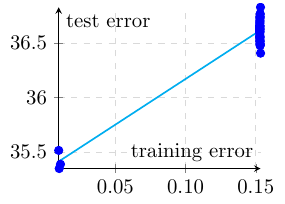}
\caption{SD method, sample size $25000$}
\label{fig:ScatterNBs12SD25000}
\end{subfigure}
\begin{subfigure}[t]{0.24\textwidth}
\includegraphics[width=0.95\linewidth]{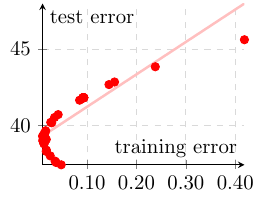}
\caption{FD method, sample size $200$}
\label{fig:ScatterNBs12FD200}
\end{subfigure} 
\begin{subfigure}[t]{0.24\textwidth}
\includegraphics[width=0.95\linewidth]{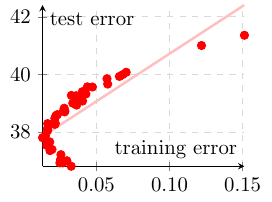}
\caption{FD method, sample size $1000$}
\label{fig:ScatterNBs12FD1000}
\end{subfigure}
\begin{subfigure}[t]{0.24\textwidth}
\includegraphics[width=0.95\linewidth]{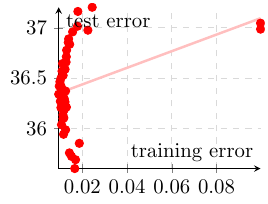}
\caption{FD method, sample size $5000$}
\label{fig:ScatterNBs12FD5000}
\end{subfigure}
\begin{subfigure}[t]{0.24\textwidth}
\includegraphics[width=0.95\linewidth]{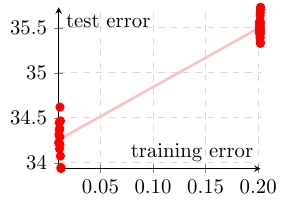}
\caption{FD method, sample size $25000$}
\label{fig:ScatterNBs12FD25000}
\end{subfigure} 
\end{center}
\caption{Scatterplot of the test error versus the training error for the (NBs) model in 12 dimensions. The line shows the linear least squares regression fit.
}
\label{fig:ScatterNBs12}
\end{figure}

Let us now compare the network fits with the smallest training error (indicated by a blue or red square in Figures \ref{fig:NaiveBayes}-\ref{fig:Copula}) to the kernel density estimator. To estimate the joint density depending on four variables, the neural network fits based on the FD method with the lowest training error seem to perform best for all sample sizes. For density estimation on $[0,1]^{12},$ the picture is less clear as there are sample sizes for which the KDE method achieves a comparable or even better test error. The test error of the SD method is consistently higher. In dimension $4,$ it decreases, however, faster than the test errors of the FD and KDE method. Based on the comparison, we do advise to use the two-step method without data splitting and to pick the reconstruction with the smallest training loss based on different random initializations.

While the idea to transform an unsupervised learning problem into a supervised learning problem and using supervised learning methods is appealing, we feel that considerable future effort is required to transform this into stable and efficient algorithms.

\section{Related literature}
\label{S: RelLit}

A more direct method for nonparametric density estimation is to use a class of candidate densities $\mF$ and estimate the density by a maximizer of the log-likelihood $$\argmax_{f\in \mF} \, \sum_{i=1}^n \log f(\bX_i),$$ which is equivalent to minimizing the negative log-likelihood or cross-entropy $$\argmin_{f\in \mF} \, - \sum_{i=1}^n \log f(\bX_i).$$ In principle, $\mF$ could be a class of neural networks that is normalized or constrained to yield (approximately) probability density functions. For general classes $\mF,$ nonparametric maximum likelihood estimators have been analyzed in the literature \cite{MR1739079}. A major drawback of this approach is the instability in low-density regions that is caused by the divergence of the logarithm $\lim_{x\downarrow 0} \log(x).$ For this reason, all convergence rate results that we are aware of require that the densities are bounded away from zero. This is rather restrictive as for machine learning applications, one often expects large low- or even zero-density regions. Note that our approach does not require the true density to be bounded away from zero.

Our method is inspired by previous work on asymptotic equivalence that links the (unsupervised) nonparametric density estimation problem to a (supervised) regression-type model. More precisely, it is shown that if the univariate densities $f$ are defined on $[0,1],$ are more than $1/2$-smooth and are bounded away from zero, then, the statistical model converges in the Le Cam  distance to the statistical problem, where we want to recover $f$ by observing $(Y_t)_{t\in [0,1]}$ with
\begin{align}
    dY_t= 2 \sqrt{f(t)} \, dt + n^{-1/2} dW_t, \quad \text{for all} \ t\in [0,1],
    \label{eq.bcuysgefw}
\end{align}
and $W$ is a Brownian motion. On a high level, convergence in Le Cam distance means that the asymptotic statistical properties are in both models the same. Model \eqref{eq.bcuysgefw} behaves similarly as observing $n$ i.i.d.\ pairs $(U_i,Y_i)$ with $U_i$ uniform on $[0,1]$ and 
$Y_i=2\sqrt{f(U_i)}+\eps_i$ for independent noise variables $\eps_i\sim \mN(0,1).$ This establishes the possibility to transfer nonparametric density estimation into a regression model without losing information regarding asymptotic results. 

While the original proof for the asymptotic equivalence statement was non-constructive \cite{nussbaum1996asymptotic}, follow-up work \cite{MR2102503, MR4050245} has identified a transformation mapping the observations in the nonparametric density model to the process $(Y_t)_{t\in [0,1]}$ satisfying \eqref{eq.bcuysgefw}. The two key steps in the construction are a Poissonization step, mapping the density estimation problem with $n$ observations to density estimation with $\mM\sim \Poi(n)$ observations, followed by a step that constructs the response variables $(Y_t)_{t\in [0,1]}$ via a Haar wavelet decomposition and a quantile coupling argument. While the  asymptotic equivalence literature motivates our two-step density estimation method and its analysis, there are still many differences as asymptotic equivalence focuses on bounding the Le Cam distance, whereas we are proposing a specific method to use supervised deep learning for nonparametric density estimation.

The proposed two-step procedure is moreover related to 
Lindsey’s method which transforms parametric estimation in exponential families into a Poisson regression problem \cite{LindseyModelComparison,LindseyProbabilityComparison, EfronTibshiraniExponential}. The first step of this method discretizes the sample space into disjoint bins. The bin counts follow a multinomial distribution that is then approximated by the Poisson distribution. Assuming Poisson distributed bin counts, maximum likelihood estimation of the parameters results then in a Poisson regression problem. A benefit of Lindsey's transformation is that the normalization constant of the exponential family vanishes. This constant is an integral over the entire domain and hard to compute in high dimensions \cite{MoschopoulosStaniswalisExponentialConditionals,GaoHastieLinCDE}. While Lindsey's method returns one observation per bin and has been formulated for exponential families, the proposed method in this work focuses on nonparametric densities and artificially creates a supervised dataset by computing a response vector for each of the datapoints. Approximation of the bin counts by the Poisson distribution occurs in our approach in the proof. 

While finalizing the article, we became aware of the similar two-step density estimation method \cite{HUYNH2020580} proposed in the pattern recognition literature. For the first step, the authors use the band limited maximum likelihood density estimator proposed in \cite{Agarwal20171294}. However, this article provides no theory.

Beyond the mentioned connections to asymptotic equivalence, Lindsey's method, and \cite{HUYNH2020580}, we are unaware of any other density estimation method that is similar to ours. It is important to emphasize that the success of such a two-step procedure relies on a regression method that achieves faster rates than direct density estimation. While this is the case here, for more traditional function spaces, direct density estimation can be shown to be already rate optimal.

\section{Proofs for Section \ref{S: Main Results}}\label{S: Proofs}

\begin{Lemma}
\label{lem.hn_exist}
If $n>1,$ then there exists a $h_n,$ such that $(\log(n)/n)^{1/d}\leq h_n\leq 2(\log(n)/n)^{1/d}$ and $h_n^{-1}$ is a positive integer.
\end{Lemma}

\begin{proof}
For all $x\geq 0,$ we have $x< 1+x\leq e^x$ and thus $\log n/n <1$ as well as $0<u_n:=2(\log n/n)^{1/d}<2$ for all $n>1.$ For all $y>0,$ one can find an integer $r$ such that $y/2\leq 2^r \leq y.$ If $y<2,$ we must have $r\leq 0.$ Thus, there exists $s\leq 0$ such that $u_n/2\leq 2^s\leq u_n.$ Set $h_n=2^s.$ Since $s\leq 0,$ we must have $h_n^{-1}=2^{-s},$ which is an integer.
\end{proof}

\subsection{Proof of Theorem \ref{T: Oracle Inequality}}\label{S: ProofOracle}
The response variables $Y_i$ in the regression model \eqref{eq.regr_model} are identically distributed, but they are not jointly independent as they all depend through the kernel density estimator on the subsample $(\bX'_{\ell})_{\ell=1}^{n}.$

To deal with the dependence induced by the kernel density estimator, we partition the hypercube $[0,1]^d$ into $h_n^{-d}$ hypercubes with sidelength $h_n.$ By construction $h_n^{-1}$ is an integer and therefore no boundary issues arise. The centers of these $h_n^{-d}$ hypercubes are given by the vectors $h_n(k_1-1/2,k_2-1/2 ,\ldots,k_d-1/2 )^\top \in[0,1]^d$ with $k_1,k_2,\ldots,k_d\in\{1,\ldots,h_n^{-1}\}.$ By numbering these points (the specific numbering of the points is irrelevant), we assign to each center an index in $\mJ:=\{1,\ldots,h_n^{-d}\}.$ The $j$-th bin $\mB_j$ is then the $|\cdot |_{\infty}$-norm ball of radius $h_n/2$ around the $j$-th center $C(\mB_j)$ in this index set. To avoid that boundary points are in two bins, we include a boundary point only if is not already included in a bin with smaller index in the ordering induced by $\mJ$. This construction gives a partition of $[0,1]^d.$ As each bin is a hypercube with sidelength $h_n$, the Lebesgue measure is $h_n^d$ (in $\mathbb{R}^d$). The neighborhood of a bin $\mB_j,$ denoted by $NB(\mB_j),$ are all bins whose centers are at most $|\cdot|_\infty$-distance $h_n$ away from the center of $\mB_j$, in other words,
\begin{align}
    NB(\mB_j)=\bigcup_{\ell : |C(\mB_j)-C(\mB_{\ell})|_{\infty}\leq h_n}\mB_{\ell}
    \label{eq.NB_def}
\end{align}
\noindent
In two dimensions this neighborhood is also known as the Moore neighborhood. 

We further subdivide the bins into equivalence classes. For all sufficiently large $n,$ $h_n\leq 1/3$ and the hypercube $[0,3h_n]^d$ contains exactly $3^d$ bins. Denote by $(j_s)_{s=1}^{3^d}$ the indices of these bins and define the index set $\mJ_s\subset\mJ$ by
\begin{equation*}
    \mJ_s:=\Big\{\ell\in\mJ: \frac 1{3h_n}\big(C(\mB_{\ell})-C(\mB_{j_s})\big)\in\mathbb{Z}^d \Big\}.
\end{equation*}
By construction, the sets $\mJ_s$ are mutually disjoint and $\bigcup_{s}\mJ_s=\mJ.$

Fix a $j\in \mJ.$ Since the kernel $K$ in the kernel density estimator has bandwidth $h_n$ and support contained in $[-1,1]$, the point estimator $\KDE(\bx)$ only depends on the data points from the kernel data set $(\bX'_{\ell})_{\ell=1}^{n}$ that are in $NB(\mB_j).$ 

More generally, for two different indices $j,\widetilde{j}\in\mJ_{s}$, $j\neq \widetilde{j}$ and points $\bx_1\in\mB_j$, $\bx_2\in\mB_{\widetilde{j}},$ the point estimators $\KDE(\bx_1)$ and $\KDE(\bx_2)$ depend on $\{\bX'_{\ell}:\bX'_{\ell} \in NB(\mB_j), 
\ell=1,\ldots,n\}$ and $\{\bX'_{\ell}:\bX'_{\ell} \in NB(\mB_{\wt j}), 
\ell=1,\ldots,n\},$ respectively. The latter two sets are dependent if $n$ is fixed (knowing that a data point is in one of the bins means that there can be at most $n-1$ in any of the other bins). If we instead assume that the sample size of the data set $(\bX_\ell')_{\ell=1}^n$ is not $n$ but $\mM$ with $\mM\sim\Poi(n)$, then $\{\bX'_{\ell}:\bX'_{\ell} \in A, 
\ell=1,\ldots,\mM\}$ and $\{\bX'_{\ell}:\bX'_{\ell} \in B, 
\ell=1,\ldots,\mM\}$ are independent, whenever $A$ and $B$ are disjoint sets. This will formally be shown in the proof of Lemma \ref{L: Bound on the noise-dependend term}. Using Poisson point process theory, we also show in the proof of Lemma \ref{L: Bound on the noise-dependend term} that $\KDE(\bx_1)$ and $\KDE(\bx_2)$ are independent.

Proving oracle inequalities for the risk $R(\widetilde{f},f_0):=\mathbb{E}_{f_0,\bX}[(\widetilde{f}(\bX)-f_0(\bX))^2]$ in the standard i.i.d.\ setting typically first derives an oracle inequality for the empirical risk $\widehat{R}_n(\widehat{f},f_0)$ as
\begin{equation*}
    \widehat{R}_n(\widehat{f},f_0):=\mathbb{E}_{f_0}\left[\frac{1}{n}\sum_{i=1}^n\big(\widehat{f}(\bX_i)-f_0(\bX_i)\big)^2\right].
\end{equation*}
Here empirical refers to the fact that the estimator $\wh f$ is evaluated at the data points $\bX_1,\ldots, \bX_n.$ The derivation of an oracle inequality for the empirical risk can be further subdivided into several steps. The bound below refers to the step where our setting and the i.i.d.\ case differ the most. The proof (presented in Section \ref{S: Technical Proofs}) relies heavily on the construction of the bins above combined with Poissonization. Recall that $\eps_i=Y_i-f(\bX_i).$
\begin{Lemma}\label{L: Bound on the noise-dependend term}
Consider the framework of Theorem \ref{T: Oracle Inequality}. For any fixed $f\in\mF$ and $\log^2(n)\log(n\vee \mN_{\mF}(\delta))\leq n,$ it holds that

 \begin{equation*}
    \begin{aligned}
    &\bigg|\mathbb{E}_{f_0}\bigg[\frac{2}{n}\sum_{i=1}^n\epsilon_i(\widehat{f}(\bX_i)-f(\bX_i))\bigg]\bigg|\\
    &\leq 2^{d+6}14e^2 \|K\|_{\infty}^{2d}F^33^{\frac{7d}{2}}\bigg(\sqrt{\widehat{R}_n(\widehat{f},f_0)}\log(n)
    \sqrt{\frac{\log(n\vee \mN_{\mF}(\delta))}{n}}+
    \log(n)\frac{\log(n\vee \mN_{\mF}(\delta))}{n}+\delta\bigg)\\
    &\quad + \frac{46 F^22^{d}\|K\|_{\infty}^{d}}{n} +8h_n^{2\beta}F^2d^{2\beta} \|K\|_\infty^{2d}  +\frac{\mathbb{E}_{\bX}[(f_0(\bX)-f(\bX))^2]}{4}+\frac{\widehat{R}_n(\widehat{f},f_0)}{4}.
    \end{aligned}
 \end{equation*}    
\end{Lemma}

With this lemma in place, we can prove in Section \ref{S: Technical Proofs} the following bound on the empirical risk. This is similar to step (III) in the oracle inequality of Lemma 4 in \cite{NonParametricRegressionReLU}.

\begin{Proposition}\label{P: Bound empirical risk}
Consider the framework of Theorem \ref{T: Oracle Inequality}. For any fixed $f\in\mF$ and $\log^2(n)\log(n\vee \mN_{\mF}(\delta))\leq n,$
\begin{equation*}
    \begin{aligned}
    \widehat{R}_n(\widehat{f},f_0)&\leq \delta2^{d+6}38e^2 \|K\|_{\infty}^{2d}F^33^{\frac{9d}{2}}\\
        &+\frac{10}{3}\mathbb{E}_{\bX}\left[(f(\bX)-f_0(\bX))^2\right]+\frac{8}{3}\Delta_n(\widehat{f},f_0)\\
        &+2^{d+6}38e^2 \|K\|_{\infty}^{2d}F^33^{\frac{7d}{2}}\log(n)\frac{\log(n\vee \mN_{\mF}(\delta))}{n}\\
        &+\frac{124F^22^{d}\|K\|_{\infty}^{d}}{n}+22h_n^{2\beta}F^2d^{2\beta} \|K\|_\infty^{2d}\\
        &+4^{d+7}19^2e^4 \|K\|_{\infty}^{4d}F^63^{7d}\log^2(n)\frac{\log(n\vee \mN_{\mF}(\delta))}{n}.
    \end{aligned}
\end{equation*}
\end{Proposition}

We now have all ingredients to finish the proof of Theorem \ref{T: Oracle Inequality}.

\begin{proof}[Proof of Theorem \ref{T: Oracle Inequality}]
If $\log^2(n)\log(n\vee \mN_{\mF}(\delta))\geq n$, the statement follows with $C_1=4F^2$ by observing that $R(\widehat{f},f_0)\leq 4F^2.$

It remains to consider the case $\log^2(n)\log(n\vee \mN_{\mF}(\delta))\leq n$. The proof of Lemma 4, Part (I) in \cite{NonParametricRegressionReLU} states that for any $\epsilon\in(0,1],$
\begin{equation}\label{Eq: Part (I) of Lemma 4 of cite NonParametricRegressionReLU}
    \begin{aligned}
    &(1-\epsilon)\widehat{R}_n(\widehat{f},f_0)-\frac{F^2}{n\epsilon}\Big(15\log\big(\mN_{\mF}(\delta)\big)+75\Big)-26\delta F\\
    &\leq R(\widehat{f},f_0)\leq (1+\epsilon)\bigg(\widehat{R}_n(\widehat{f},f_0)+(1+\epsilon)\frac{F^2}{n\epsilon}\Big(12\log(\mN_{\mF}(\delta))+70\Big)+26\delta F\bigg).
    \end{aligned}
\end{equation}
This lemma for the standard nonparametric regression problem relates the risk to its empirical counterpart. The inequality and its proof only depend on the $\bX_i$ and on the function class $\mF$, not on the noise or the response variables $Y_i.$ Since in our regression model \eqref{eq.regr_model} the variables $\bX_i$ are i.i.d. (the dependence is induced by the response variables $Y_i$ and $\epsilon_i$), this inequality is still valid.

Substituting the bound on $\widehat{R}_n(\widehat{f},f_0)$ from Proposition \ref{P: Bound empirical risk} in \eqref{Eq: Part (I) of Lemma 4 of cite NonParametricRegressionReLU}, choosing $\eps=1$ and $f$ as a minimizer over $\mF$ of $\mathbb{E}_{\bX}\left[(f(\bX)-f_0(\bX))^2\right]$, using the fact that $h_n \leq 2(\log(n)/n)^{1/d}$, and replacing the explicit constants by $C_1,C_2,C_3$ yields the result.
\end{proof}

\subsection{Proof of Theorem \ref{T: Convergence Rates}}

The following lemma provides a bound on the covering entropy.

\begin{Lemma}[Lemma 5 combined with Remark 1 of \cite{NonParametricRegressionReLU}]\label{L: Covering Entropy}
    For any $\delta>0$
\begin{equation*}
    \log\big(\mN_{\mF(L,\bp,s,\infty)}(\delta)\big)\leq (s+1)\log\left(2^{2L+5}\delta^{-1}(L+1)p_0^2p_{L+1}^2s^{2L}\right).
\end{equation*}
\end{Lemma}

The proof of Theorem 1 in \cite{NonParametricRegressionReLU} (see also \cite{10.1214/24-AOS2351}) derives the following bound for the approximation error for function approximation in the function class $\mG(q,\bd,\bt,\balpha,Q')$ by sparsely connected deep ReLU networks.

\begin{Theorem}\label{T: Approx Result}
For every function $g\in\mG(q,\bd,\bt,\balpha,Q')$ and whenever
    \begin{compactitem}
        \item[(i)] $\sum_{i=1}^q \tfrac{\alpha_i+t_i}{2\alpha_i^*+t_i}\log_2(4t_i\vee 4\alpha_i)\log_2(n)\leq L \lesssim n\phi_n,$
        \item[(ii)] $ n\phi_n\lesssim \min_{i=1,\hdots,L}p_i,$
        \item[(iii)] $s\asymp n\phi_n\log(n),$
        \item[(iv)] $F\geq \max\{Q',1\},$
    \end{compactitem}
    then there exists a neural network $H\in \mF(L,\bp,s,F)$ 
    and a constant $C_8$ only depending on $q,\bd,\bt,\balpha,F$ and the implicit constants in (i), (ii) and (iii), such that
    \begin{equation*}
        \|g-H\|_{\infty}^2\leq C_8\phi_n.
    \end{equation*}
\end{Theorem}

We now have all the necessary ingredients to prove Theorem \ref{T: Convergence Rates}
\begin{proof}[Proof of Theorem \ref{T: Convergence Rates}]
Apply the general oracle inequality in Theorem \ref{T: Oracle Inequality} with the choice $\delta=n^{-1}$ to the neural network class $\mF(L,\bp,s,F)$ with parameter constraints as in the statement of the theorem. For the approximation error in the oracle inequality, we use Theorem \ref{T: Approx Result}. For the covering entropy, the bound from Lemma \ref{L: Covering Entropy} gives $\log\big(n\vee \mN_{\mF(L,\bp,s,\infty)}(\tfrac 1n)\big)\lesssim (s+1)L\log(n) \asymp n L \phi_n\log^2(n)$. Since $L \gtrsim \log(n),$ we have $(\log(n)/n)^{2\beta/d}\lesssim L(n^{-2\beta/d} \vee n^{-1}).$ As $\phi_n \gg n^{-1},$ $L\phi_n \log^4(n)+(\log(n)/n)^{2\beta/d}\leq L\max(\phi_n \log^4(n),n^{-2\beta/d}).$ Thus, Theorem \ref{T: Oracle Inequality} yields
\begin{align*}
    R(\widehat{f}_n,f_0)\leq & \, C_1\frac{\log^2(n)\log(n\vee \mN_{\mF(L,\bp,s,\infty)}(\tfrac 1n))}{n}+C_2\delta+C_3\left(\frac{\log(n)}{n}\right)^{\frac{2\beta}{d}}+6\Delta_n(\widehat{f}_n,f_0)\\
    &+7\inf_{f\in\mF}\mathbb{E}_{\bX}\left[(f(\bX)-f_0(\bX))^2\right] \\
    \leq &\, C_4 L\max\Big(\phi_n\log^{4}(n), n^{-\frac{2\beta}{d}}\Big)+6\Delta_n(\widehat{f}_n,f_0),
\end{align*}
for a sufficiently large constant $C_4,$ only depending on $q,\bd,\balpha,\bt,F,\beta,K$ and the implicit constants in (ii), (iii), and (iv). This completes the proof.
\end{proof}

\section{Proofs for Section \ref{S: Examples} }\label{S: Other proofs}

\begin{Lemma}\label{L: Holder class of product of different terms}
Let $m$ be a positive integer and $Q>0.$ Then $f:[-Q,Q]^m\rightarrow \mathbb{R}$: $f(\bx):=\prod_{i=1}^mx_i$ is in $\mH^{\gamma}_m([-Q,Q]^m,(Q+1)^m),$ for all $\gamma\geq m+1$.
\end{Lemma}
\begin{proof}
    Observe that $|\partial^0f(\bx)|=|f(\bx)|\leq Q^m$, $\partial_{x_j}f(\bx)=\prod_{i=1,i\neq j}^m x_i$ and $\partial_{x_j}\partial_{x_j}f=0,$ for $i=1,\ldots,m.$ This means that for all $\balpha\in\mathbb{Z}^m_{\geq 0}$ it holds that $\partial^{\balpha}f=0$ if $\alpha_j\geq 2$ for some $j\in\{1,\ldots,m\}$. Rephrased, $\partial^{\balpha}f\neq 0$ if and only if $\balpha\in\{0,1\}^m$. Furthermore for $\balpha\in\{0,1\}^m,$ $|\partial^{\balpha}f(\bx)|=|\prod_{i:\alpha_i=0}x_i|\leq Q^{m-|\balpha|_0},$ where $|\cdot|_0$ denotes the counting norm. There are $\binom{m}{m-|\balpha|_0}$ ways to distribute $m-|\balpha|_0$ zeros over a vector of length $m$. Hence for $\gamma\geq m+1,$ we get by the binomial theorem
    \begin{equation*}
        \sum_{\balpha:|\balpha|_1<\gamma}\|\partial^{\balpha}f\|_{\infty}\leq\sum_{k=0}^m\binom{m}{k}Q^{k}=(Q+1)^m.
    \end{equation*}
    If $|\balpha|_1>m,$ then there exists at least one $j$ such that $\alpha_j\geq 2$ implying that $\partial^{\balpha}f=0$ in this case. In the case that $|\balpha|_1=m$, then either there exists a $j$ such that $\alpha_j\geq 2$, so $\partial^{\balpha}f=0$, or $\balpha$ is the vector with only ones, in which case $\partial^{\balpha}f=1$. Hence, $\gamma\geq m+1$ yields
    \begin{equation*}
        \sum_{\balpha:|\balpha|_1=\floor{\gamma}}\ \sup_{\bx,\by\in[-Q,Q]^m, \,  \bx\neq\by} \ \frac{|\partial^{\balpha}f(\bx)-\partial^{\balpha}f(\by)|}{|\bx-\by|_{\infty}^{\gamma-\floor{\gamma}}}=0.
    \end{equation*}
    Together with the definition of the H\"{o}lder ball in \eqref{Eq: Holder ball}, the statement follows.
\end{proof}

\begin{proof}[Proof of Lemma \ref{L: General potential product2}] 
    The function $g_0=(g_{0,1},\ldots,g_{0,|\mR|})$ is given by $g_{0,I}=\psi_I$ for all $I\in \mR$. From $\psi_I\in \mH_{r_I}^{\gamma}([0,1]^{r_I},Q)$ and $|I|\leq r$ it follows that $t_0=r$ and $\alpha_0=\gamma.$
    The function $g_1(u_1,\ldots,u_{|\mR|})=\prod_{I\in \mR}u_I$ is the product of $|\mR|$ different factors in $[-Q,Q].$ Applying Lemma \ref{L: Holder class of product of different terms} yields $g_1\in \mH^{\zeta}_{|\mR|}([-Q,Q]^{|\mR|},(Q+1)^{|\mR|})$ for all $\zeta\geq |\mR|+1.$ So $t_1=|\mR|$ and $\alpha_1$ is arbitrarily large.
    \end{proof}

\begin{proof}[Proof of Lemma \ref{L: Independent variables smoothnes bound}]
    Since $f$ is $\alpha$-H\"older smooth, there exists a constant $Q$ such that $f\in \mH^{\alpha}_d([0,1]^d,Q).$ Thus for any $k=0,1,\ldots,\floor{\alpha},$
    \begin{equation}\label{Eq: Chain rule for independent variables}
        \frac{\partial^{k}}{\partial x_j^k}f(\bx)=\bigg(\prod_{i\neq j}f_i(x_i)\bigg)f_j^{(k)}(x_j).
    \end{equation}
    Since $\prod_{i\neq j}f_i$ is a density on $[0,1]^{d-1},$ it is nonnegative and $C:=\prod_{i\neq j}\|f_i\|_\infty>0,$ with $\|\cdot\|_\infty$ the supremum norm on $[0,1]^d.$ Since $f_j$ only depends on $x_j$ and $f$ is $\alpha$-H\"{o}lder smooth, for any $k=0,1,\ldots, \floor{\alpha},$
    \begin{equation}\label{Eq: Infinity norm bound for independent variables}
        \frac{Q}{C}\geq \frac{1}{C}\sup_{(x_1,\ldots,x_d) \in [0,1]^d}\bigg |\bigg(\prod_{i\neq j}f_i(x_i)\bigg)f_j^{(k)}(x_j)\bigg |\geq \frac{\prod_{i\neq j}\|f_i\|_\infty}{C} \left\|f_j^{(k)}\right\|_{\infty}= \left\|f_j^{(k)}\right\|_{\infty}.
    \end{equation}
    Similarly, by the $\alpha$-H\"{o}lder smoothness of $f$ and \eqref{Eq: Chain rule for independent variables},
    \begin{align}\label{Eq: Holder seminorm bound for independent variables}
    \begin{split}
        \frac{Q}{C}&\geq \frac{1}{C} \ \sup_{\bx,\by\in[0,1]^d,\bx\neq\by} \frac{\left|\frac{\partial^{\floor{\alpha}}}{\partial x_j^{\floor{\alpha}}}f(\bx)-\frac{\partial^{\floor{\alpha}}}{\partial x_j^{\floor{\alpha}}}f(\by)\right|}{|\bx-\by|_{\infty}^{\alpha-\floor{\alpha}}} \\
        &\geq \frac{\prod_{i\neq j}\|f_i\|_\infty}{C}  \sup_{x,y\in[0,1],x\neq y}\frac{\left|f_j^{(\floor{\alpha})}(x)-f_j^{\floor{\alpha})}(y)\right|}{|x-y|^{\alpha-\floor{\alpha}}}.
    \end{split}
    \end{align}
    From \eqref{Eq: Infinity norm bound for independent variables} and \eqref{Eq: Holder seminorm bound for independent variables} it follows that $f_j\in \mH^{\alpha}_1([0,1],(\floor{\alpha}+1)Q/C).$
\end{proof}

\begin{proof}[Proof of Lemma \ref{L: general d dimensional Copula}]
The function $g_0=(g_{0,1},\ldots,g_{0,2d})$ is given by $g_{0,i}(x_i)=f_i(x_i)$ for $i=1,\ldots,d$ and $g_{0,i}(x_{i-d})=F_{i-d}(x_{i-d})$ for $i=d+1,\ldots,2d.$ Each of these functions is univariate, so $t_0=1.$ Since $F_{i-d}$ is the c.d.f.\ of $f_{i-d},$ it holds that $F_{i-d}\in \mH_{1}^{\gamma_i+1}([0,1], Q_i+1)$. Thus the function $g_{0,i}$ with the smallest H\"{o}lder smoothness has to correspond to one of the functions $f_i$ and $\alpha_0=\min_{i=1,\ldots,d}\gamma_i$.
The function $g_1=(g_{1,1},\ldots,g_{1,d+1})$ satisfies $g_{1,i}(y_i)=y_i$ (the identity function) for $i=1,\ldots, d$ and $g_{1,d+1}(\bv)=c(v_1,\ldots,v_d)$, so $t_1=d$. For $i=1,\ldots,d$ the domain of $g_{1,i}$ is $[0,\|f_i\|_{\infty}]\subseteq [0,Q_i]$, so $g_{1,i}\in \mH^{\gamma}_1([0,Q_i], Q_i+1),$ for all $\gamma\geq 2.$ This means the H\"{o}lder smoothness of $g_{1,i}$ can be taken arbitrarily large, that consequently $g_{1,d+1}$, corresponding to the copula $c$, has the smallest H\"{o}lder smoothness among the component functions of $g_1,$ and thus $\alpha_1=\gamma_c$. Set $Q:=Q_c\vee(\max_{i=1,\ldots,d}Q_i)$, then
$g_2(u,y_1,\ldots,y_d)=u\prod_{i=1}^dy_i$ is the product of $d+1$ different factors in $[-Q,Q]^{d+1}$. Applying Lemma \ref{L: Holder class of product of different terms} yields $g_2\in \mH^{\gamma}_{d+1}([-Q,Q]^{d+1}, (Q+1)^{d+1})$ for all $\gamma\geq d+2.$ So $t_2=d+1$ and the smoothness index $\alpha_2$ can be taken to be arbitrarily large.\end{proof}

\begin{proof}[Proof of Lemma \ref{L: Pair Copula composition}]
    The function $g_0=(g_{0,1},\ldots,g_{0,2d})$ is given by $g_{0,i}(x_i)=f_i(x_i)$ for $i=1,\ldots,d$ and $g_{0,i}(x_{i-d})=F_{i-d}(x_{i-d})$ for $i=d+1,\ldots,2d$. Recall that $f_i\in \mH_{1}^{\gamma}([0,1], Q),$ for all $i=1,\ldots,d.$
    Since $F_{i-d}$ is the c.d.f.\ of $f_{i-d},$ it holds that $F_{i-d}\in \mH_{1}^{\gamma+1}([0,1], Q+1)$. So $t_0=1$ and $\alpha_0=\gamma$. 
    
    The function $g_1=(g_{1,1},\ldots,g_{1,d+(d-1)})$ satisfies $g_{1,i}(u_i)=u_i$ (the identity function) for $i=1,\ldots,d$. For $f$ of the form \eqref{Eq: Pair Copula decomposition canonical vine} it holds that $g_{1,i}(v_1,v_2)=c_{1,i+1-d}(v_1,v_2)$ for $i=d+1,\ldots,d+(d-1)$ and for $f$ of the form \eqref{Eq: Pair Copula decomposition d vine} we have that $g_{1,i}(v_1,v_2)=c_{i-d,i+1-d}(v_1,v_2)$ for $i=d+1,\ldots,d+(d-1)$. Due to $[0,\max_{i=1,\ldots,d}\|f_i\|_{\infty}]\subseteq [0,Q],$ we can define $g_{1,i}$ on $[0,Q].$ Since $g_{1,i}$ is in this case the identity, treating the cases $0<\beta\leq 1$ and $\beta>1$ separately, we find $g_{1,i}\in \mH^{\beta}_1([0,Q], 2Q+1),$ for all $\beta>0.$ By definition all bivariate copulas are in $\mH_2^{\gamma_c}([0,1]^2,Q).$ This means that  $t_1=2$ and $\alpha_1=\gamma_c.$

    The function $g_2(u_1,\ldots,u_d,y_1,\ldots,y_{d-1})=\prod_{k=1}^du_k\prod_{j=1}^{d-1}y_j$ is the product of $2d-1$ factors and can be defined on $[-Q,Q]^{2d-1}$. Thus by Lemma \ref{L: Holder class of product of different terms} it  holds that $g_2\in \mH_{2d-1}^{\zeta}([-Q,Q]^{2d-1}, (Q+1)^{2d-1})$ for all $\zeta\geq 2d$, so $t_2=2d-1$ and $\alpha_2$ is arbitrarily large.
\end{proof}

\subsection{Proof of Theorem \ref{T: Mixture convergence rates}}

We work in the density estimation model as defined in Section \ref{S:Model} with mixture density $f_0=\sum_{j=1}^ra_jf_j$, where $a_1,\ldots a_r$ are non-negative mixture weights summing up to one, and $f_j\in \mathcal{C}_d^{\beta_j}([0,1]^d,Q)\cap\mG(q_j,\bd_j,\bt_j,\balpha_j,Q),$ for all $j=1,\ldots,r.$  Recall that $\alpha_{i,j}^{*}:=\alpha_{i,j}\prod_{\ell=i+1}^q(\alpha_{\ell,j}\wedge 1),$ and $\phi_n^{\star}:=\max_{j=1,\ldots,r} \phi_{n,j}$, where $$\phi_{n,j} := \max_{i=0,\ldots,q} n^{-\frac{2\alpha^*_{i,j}}{2\alpha^*_{i,j}+t_{i,j}}}.$$

\begin{Lemma}[Approximation of mixtures]\label{L: Approximation of mixtures} Whenever
    \begin{compactitem}
        \item[(i)] $\max_{j=1,\ldots,r}\sum_{i=1}^{q_j} \tfrac{\alpha_{i,j}+t_{i,j}}{2\alpha_{i,j}^*+t_{i,j}} \log_2(4t_{i,j}\vee 4\alpha_{i,j})\log_2(n)\leq L\lesssim n\phi_n^{\star},$
        \item[(ii)] $ n\phi_n^{\star}\lesssim \min_{i=1,\hdots,L}p_i,$
        \item[(iii)] $s\asymp  n\phi_n^{\star}\log(n),$
        \item[(iv)] $F\geq \max\{Q,1\},$
    \end{compactitem} then,
   for $n$ large enough, there exists a network $H\in \mF(L,\bp,s,F)$
and a constant $C_9$ only depending on $(q_j,\bd_j,\bt_j,\balpha_j)_{j=1}^r,$ $ r, F$ and the implicit constants in (i), (ii) and (iii) such that
    \begin{equation*}
        \bigg\|\sum_{j=1}^ra_jf_j-H\bigg\|_{\infty}^2\leq C_9\phi_n^{\star}.
    \end{equation*}
    
\end{Lemma}
\begin{proof}
For positive constants $c_L, c_p,c_{s\ell},c_{su},$ let $L^\star$, $\bp^\star$ and $s^\star$ be such that
        \begin{compactitem}
        \item[(i')] $\max_{j=1,\ldots,r}\sum_{i=1}^{q_j}\tfrac{\alpha_{i,j}+t_{i,j}}{2\alpha_{i,j}^*+t_{i,j}}\log_2(4t_{i,j}\vee 4\alpha_{i,j})\log_2(n)\leq L^\star\leq c_L n\phi_n^{\star}$
        \item[(ii')] $ n\phi_n^{\star}\leq c_p \min_{i=1,\hdots,L}p_i^\star$
        \item[(iii')] $c_{s\ell} n\phi_n^{\star}\log(n)\leq  s^\star\leq c_{su} n\phi_n^{\star}\log(n).$
    \end{compactitem}
    For $n$ large enough, we have
    \begin{compactitem}
        \item[(I)] $c_L n\phi_n^{\star}\leq (c_{s\ell}/(2r)) n\phi_n^{\star}\log(n)$,
        \item[(II)]  $n\phi_n^{\star}>rc_p$,
        \item[(III)]$\floor{c_Ln\phi_{n,j}}\geq \sum_{i=1}^{q_j}\tfrac{\alpha_{i,j}+t_{i,j}}{2\alpha_{i,j}^*+t_{i,j}}\log_2(4t_{i,j}\vee 4\alpha_{i,j})\log_2(n)$, \ for all $j=1,\ldots,r,$
        \item[(IV)]$(c_{sl}/(4r))n\phi_{n,j}\log(n)\geq 1,$  \ for all $j=1,\ldots,r.$
    \end{compactitem}
    
    For $j=1,\ldots,r$ define $L_j:=\min\{L^\star,\floor{c_Ln\phi_{n,j}}\}$, $p_{i,j}=\floor{p_i^\star/r}$ and $s_j=\floor{s^\star\phi_{n,j}/(2r\phi_n^\star)}$.
    Recall that $\phi_n^\star=\max_{j=1,\ldots,r}\phi_{n,j}$. Using the definition of $L_j$ and (III) yields
\begin{equation*}
    \sum_{i=1}^{q_j}\frac{\alpha_{i,j}+t_{i,j}}{2\alpha_{i,j}^*+t_{i,j}}\log_2(4t_{i,j}\vee 4\alpha_{i,j})\log_2(n)\leq L_j \leq c_L n\phi_{n,j}.
\end{equation*}
Using (ii'), (II), and the definitions of $\phi_n^*$ and $\bp_j$ we get that 
\begin{equation*}
    n\phi_{n,j}\leq 2c_pr\min_{i=1,\hdots,L}\floor{p_{i}^*/r}=2c_pr\min_{i=1,\hdots,L}p_{i,j}.
\end{equation*}
From (IV), the definition $s_j=\floor{s^\star\phi_{n,j}/(2r\phi_n^\star)}$, (iii'), and $\floor{u}\geq u-1$ for all $u\in\mathbb{R}$, it follows that 
\begin{equation*}
    \frac{c_{sl}}{4r} n\phi_{n,j}\log(n)\leq \frac{c_{sl}}{2r} n\phi_{n,j}\log(n) -1 \leq s_j \leq \frac{c_{su}}{2r} n\phi_{n,j}\log(n).
\end{equation*}
This means that for $j=1,\ldots,r,$ the class $\mF(L_j,\bp_j,s_j,F)$ and the function $f_j\in\mathcal{C}_d^{\beta}([0,1]^d,Q)\cap\mG(q_j,\bd_j,\bt_j,\balpha_j,Q)$ satisfy the conditions of Theorem \ref{T: Approx Result}. Applying Theorem \ref{T: Approx Result} gives us that for each $j=1\ldots,r$ there exist a network $H_j\in \mF(L_j,\bp_j,s_j,F)$
    such that $\|f_j-H_j\|_{\infty}^2\leq C_{8,j}\phi_{n,j}$.
    Since $a_j$ is in $[0,1]$, multiplying the last weight matrix of $H_j$ with $a_j$ yields a network $a_jH_j$ in the same network class as $H_j$ such that $\|a_jf_j-a_jH_j\|_{\infty}^2\leq C_{8,j}\phi_{n,j}.$
    
    Whenever $L_j<L^\star,$ we can synchronize the depth by adding additional layers with identity weight matrix such that
    \begin{equation*}
        \mF_j(L_j,\bp_j,s_j,F)\subset \widetilde{\mF}_{j}(L^*,(\bp_j,\underbrace{1,\ldots,1}_{(L^\star-L_j)\text{ times}}),s_j+(L^\star-L_j),F).
    \end{equation*}
    For ease of notation define $\widetilde{\bp}_j=(\bp_j,1,\dots,1).$
    Placing all these networks in parallel yields a network
    \begin{equation*}
        H\in\mF\Big(L^\star,\sum_{j=1}^r\widetilde{\bp}_j,\sum_{j=1}^r\big(s_j+(L^\star-L_j)\big),F\Big),
    \end{equation*}
   such that
    \begin{equation*}
        \left\|\sum_{j=1}^ra_jf_j-H\right\|_{\infty}^2\leq \left(\sum_{j=1}^r\|a_jf_j-a_jH_j\|_{\infty}\right)^2\leq \left(\sum_{j=1}^r\sqrt{C_{8,j}\phi_{n,j}}\right)^2 \leq r^2\max_{j=1,\ldots,r}C_{8,j}\phi_{n,j}.
    \end{equation*}
    A network with width $\bp$ and sparsity $s$ can always be embedded in a larger network of the same depth with width $\widetilde{\bp}\geq \bp$ (inequalities between vectors should always be understood as componentwise inequalities) and network sparsity $\widetilde{s}\geq s.$ Thus it remains to show that $\sum_{j=1}^r\widetilde{\bp}_j\leq \bp^\star$ and $\sum_{j=1}^r\big(s_j+(L^\star-L_j)\big)\leq s^\star.$
    First consider the width. Using the definitions of $p_{i,j}$ and $\widetilde{\bp}_j,$ we get for $i=1,\ldots,L^\star$ that $\sum_{j=1}^r\widetilde{p}_{i,j}\leq r\max_{j=1,\ldots,r}\widetilde{p}_{i,j}\leq r\max\{p_i^\star/r,1\}.$ From (II) and (ii'), we get that $p_i^\star/r>1.$ Hence, $\sum_{j=1}^r\widetilde{\bp}_j\leq \bp^\star$.
    Now consider the sparsity. By the definition of $s_j$ it holds that $s_j\leq s^\star/(2r).$ From (i') and (I), we get that $L^\star\leq s^\star/(2r)$. Hence,
    $\sum_{j=1}^r\big(s_j+(L^\star-L_j)\big)\leq \sum_{j=1}^r\big(s_j+L^\star\big)\leq s^\star.$
\end{proof}

\begin{proof}[Proof of Theorem \ref{T: Mixture convergence rates}]
        The derivative of a sum is the sum of the derivatives. Furthermore $(a_1,\ldots,a_r)$ are non-negative mixture weights that sum op to one. Since $f_j\in \mathcal{C}_d^{\beta}([0,1]^d,Q)$ for $j=1,\ldots,r$, this means that also $f_0\in \mathcal{C}_d^{\beta}([0,1]^d,Q).$
        The statement of the theorem now follows from taking $\delta=1/n$ and the network class $\mF(L,\bp,s,F)$ as the function class in Theorem \ref{T: Oracle Inequality}. For the approximation error in the oracle inequality, we use Lemma \ref{L: Approximation of mixtures} and for the covering entropy the bound from Lemma \ref{L: Covering Entropy}. This yields the result.
\end{proof}

\section{Proofs for Section \ref{S: Simulations}}

\begin{proof}[Proof of Lemma \ref{L: Structure mixing conditional}] To represent $f_j(x_j|x_i) = x_i h_j(x_j)+(1-x_i)h_j(1-x_j)$ as a composition $g_1\circ g_0,$ choose $g_0(x_i,x_j)=(x_i,h_j(x_j),h_j(1-x_j)).$ Clearly $t_0=1.$ Since $[0,1] \ni x_i \mapsto x_i$ lies in $\mH^{\gamma}_1([0,1],2),$ for all $\gamma >0,$ we get that $\alpha_0=\gamma_j.$ The function $g_1$ is given by $g_1(x_i,y_1,y_2)=x_iy_1+(1-x_i)y_2,$ so $t_1=3.$ The partial derivatives are $\partial_{x_i}g_1=y_1-y_2$, $\partial_{y_1}g_1=x_i$, $\partial_{y_2}g_1=1-x_i$, $\partial_{x_i}\partial_{y_1}g_1=1$ and $\partial_{x_i}\partial_{y_2}g_1=-1.$ All other partial derivatives of $g_1$ vanish. Thus $g_1\in \mH^{\gamma}_3([0,1]\times[-Q,Q]^2,4(Q+1)),$ for all $\gamma > 3,$ so $\alpha_1$ is arbitrarily large.
\end{proof}

\begin{proof}[Proof of Lemma \ref{L: Structure shifting conditional}]
To represent $f_j(x_j|x_i) = h_j\big(\max\{x_j-x_i/4,0\}\big)$ as a composition $g_1\circ g_0,$ choose $g_0(x_j,x_i)=\max\{0,x_j-x_i/4\}.$ The derivative of this function is discontinuous along the line $x_j-x_i/4=0.$ Observe that $|\max(0,a)-\max(0,a+b)|\leq |b|$, for all real numbers $a,b$. Hence
    \begin{equation*}
        \frac{|g_0(x_j,x_i)-g_0(x_j+u,x_i+v)|}{\max(|u|,|v|)}\leq \frac{|u/4-v|}{\max(|u|,|v|)}\leq \frac{5}{4}.
    \end{equation*}
    Thus $g_0\in \mH^{1}_2([0,1]^2,9/4),$ so $\alpha_0=1.$ The function $g_1$ is given by $g_1(y)=h_j(y),$ thus $t_1=1$ and $\alpha_1=\gamma_j.$
\end{proof}

\section{Proofs for Section \ref{S: Proofs} }\label{S: Technical Proofs}

\begin{proof}[Proof of Lemma \ref{L: Bound on the noise-dependend term}] 

The random variable $\epsilon_i=\KDE(\bX_i)-f_0(\bX_i)$ is not centered. The first step adds and subtracts $\mathbb{E}_{f_0}[\epsilon_i\rvert \bX_i]$ to get the centered random variable $\epsilon_i-\mathbb{E}_{f_0}[\epsilon_i\rvert \bX_i]$ instead. Together with the triangle inequality, this gives
\begin{equation}\label{Eq:Centering the noise}
    \begin{aligned}
    &\left|\mathbb{E}_{f_0}\left[\frac{2}{n}\sum_{i=1}^n\epsilon_i(\widehat{f}(\bX_i)-f(\bX_i))\right]\right|\\
    &\leq \left|\mathbb{E}_{f_0}\left[\frac{2}{n}\sum_{i=1}^n(\epsilon_i-\mathbb{E}_{f_0}[\epsilon_i\rvert \bX_i])(\widehat{f}(\bX_i)-f_0(\bX_i))\right] \right|\\
    &+\left|\mathbb{E}_{f_0}\left[\frac{2}{n}\sum_{i=1}^n(\epsilon_i-\mathbb{E}_{f_0}[\epsilon_i\rvert \bX_i])(f_0(\bX_i)-f(\bX_i))\right]\right|\\
    &+\left|\mathbb{E}_{f_0}\left[\frac{2}{n}\sum_{i=1}^n\mathbb{E}_{f_0}[\epsilon_i\rvert\bX_i](\widehat{f}(\bX_i)-f(\bX_i))\right]\right|\\
    &=:(I)+(II)+(III).
    \end{aligned}
\end{equation}    
By the tower rule, we can in $(II)$ first condition the expectation on $\bX_i$. Now $(II)=0$ follows from
\begin{equation*}
\begin{aligned}
    \mathbb{E}_{f_0}\Big[\left(\epsilon_i-\mathbb{E}_{f_0}\left[\epsilon_i\big\rvert \bX_i\right]\right)\left(f_0(\bX_i)-f(\bX_i)\right)\Big\rvert \bX_i\Big]
    &=\big(\mathbb{E}_{f_0}\left[\epsilon_i\rvert \bX_i\right]-\mathbb{E}_{f_0}\left[\epsilon_i\rvert \bX_i\right]\big)\big(f_0(\bX_i)-f(\bX_i)\big) \\
    &=0.
    \end{aligned}
\end{equation*}
For real numbers $a_i,b_i$, we have $(|a_i|-|b_i|/2)^2\geq 0$ and therefore $|a_ib_i|\leq a_i^2+b_i^2/4$ as well as $\sum_{i}|a_ib_i|\leq \sum_i a_i^2+\frac{1}{4}\sum_i b_i^2.$ Bringing first the absolute value inside the expectation and applying this inequality twice, once to the sequences $(2\mathbb{E}_{f_0}[\epsilon_i\rvert\bX_i]/\sqrt{n})_i$ and $((\widehat{f}(\bX_i)-f_0(\bX_i))/\sqrt{n})_i$ and once to the sequences $(2\mathbb{E}_{f_0}[\epsilon_i\rvert\bX_i]/\sqrt{n})_i$ and $((f_0(\bX_i)-f(\bX_i))/\sqrt{n})_i$ yields
\begin{equation*}
    \begin{aligned}
    &\left|\mathbb{E}_{f_0}\left[\frac{2}{n}\sum_{i=1}^n\mathbb{E}_{f_0}[\epsilon_i\rvert\bX_i](\widehat{f}(\bX_i)-f(\bX_i))\right]\right|\\
    &\overset{(i)}{=}\left|\mathbb{E}_{f_0}\left[\sum_{i=1}^n\frac{2\mathbb{E}_{f_0}[\epsilon_i\rvert\bX_i]}{\sqrt{n}}\frac{(\widehat{f}(\bX_i)-f_0(\bX_i))}{\sqrt{n}}\right]+\mathbb{E}_{f_0}\left[\sum_{i=1}^n\frac{2\mathbb{E}_{f_0}[\epsilon_i\rvert\bX_i]}{\sqrt{n}}\frac{(f_0(\bX_i)-f(\bX_i))}{\sqrt{n}}\right]\right|\\
    &\leq 8\mathbb{E}_{f_0}\left[\frac{1}{n}\sum_{i=1}^n\left(\mathbb{E}_{f_0}[\epsilon_i\rvert\bX_i]\right)^2\right]+\frac{1}{4}\mathbb{E}_{f_0}\left[\frac{1}{n}\sum_{i=1}^n\left(f_0(\bX_i)-f(\bX_i)\right)^2\right]\\
    &+\frac{1}{4}\mathbb{E}_{f_0}\left[\frac{1}{n}\sum_{i=1}^n\left(\widehat{f}(\bX_i)-f_0(\bX_i)\right)^2\right]\\
    &\overset{(ii)}{=} 8\mathbb{E}_{f_0}\left[\left(\mathbb{E}_{f_0}[\epsilon_1\rvert\bX_1]\right)^2\right]+\frac{\mathbb{E}_{\bX}[(f_0(\bX)-f(\bX))^2]}{4}+\frac{\widehat{R}_n(\widehat{f},f_0)}{4},\\
    \end{aligned}
\end{equation*}
where for $(i)$ we added and subtracted the same term and $(ii)$ follows from the definition of $\widehat{R}_n(\widehat{f},f_0)$ and the fact that the $\bX_i$ are i.i.d.,
Proposition \ref{P:Bound on conditional noise Kernel version} gives $\mathbb{E}_{f_0}[(\mathbb{E}_{f_0}[\epsilon_1\rvert\bX_1])^2]\leq h_n^{2\beta}F^2d^{2\beta} \|K\|_\infty^{2d}$ and so
\begin{equation}\label{Eq: Bound on term (III)}
    \begin{aligned}
    (III)\leq 8h_n^{2\beta}F^2d^{2\beta} \|K\|_\infty^{2d} +\frac{\mathbb{E}_{\bX}[(f_0(\bX)-f(\bX))^2]}{4}+\frac{\widehat{R}_n(\widehat{f},f_0)}{4}.
    \end{aligned}
\end{equation}
It remains to bound
\begin{align*}
    (I)= \left|\mathbb{E}_{f_0}\left[\frac{2}{n}\sum_{i=1}^n(\epsilon_i-\mathbb{E}_{f_0}[\epsilon_i\rvert \bX_i])(\widehat{f}(\bX_i)-f_0(\bX_i))\right] \right|
\end{align*}
in \eqref{Eq:Centering the noise}. Let us briefly outline the main ideas. A standard strategy to do this is to use that $\wh f\in \mF$ and bound 
\begin{align*}
    (I)\leq 
    \mathbb{E}_{f_0}\left[\frac{2}{n} \sup_{f\in \mF}\left|\sum_{i=1}^n(\epsilon_i-\mathbb{E}_{f_0}[\epsilon_i\rvert \bX_i])(f(\bX_i)-f_0(\bX_i))\right|\right].
\end{align*}
The remaining step is then to get the supremum $\sup_{f\in \mF}$ out of the expectation. This is the central problem in empirical process theory. Standard empirical process techniques consider a covering of $\mF.$ On each of the balls in the covering the expectation does not change much, such that one can replace the supremum by a maximum over the centers of the covering balls plus some remainder terms. To control the expectation of the maximum over the centers of the balls from the covering, one can now apply the union bound together with concentration bounds. While we will follow these steps, there are various technical challenges that occur because of the dependence in the data. 

The covering number of $\mF$ with supremum norm balls of radius $\delta>0$ has been called $\mN_{\mF}(\delta).$ If $\mN_{\mF}(\delta)<n,$ then one can add some balls with centers in $\mF$ to the covering, to obtain a (not necessarily optimal) covering with \[N=n\vee \mN_{\mF}(\delta)\] balls. By assumption, the $N$ centers $f_1,\ldots,f_N$ lie in $\mF.$ Choose $k^*\in\{1,\ldots,N\}$ such that
\begin{equation*}
    \|\widehat{f}-f_{k^*}\|_\infty=\min_{1\leq\ell\leq N}\|\widehat{f}-f_{\ell}\|_\infty.
\end{equation*} 
In particular, $k^*$ is random. Define $(IV):= |\mathbb{E}_{f_0}[\frac{2}{n}\sum_{i=1}^n(\epsilon_i-\mathbb{E}_{f_0}[\epsilon_i\rvert \bX_i])(f_{k^*}(\bX_i)-f_0(\bX_i))]|.$ This gives us that
\begin{equation}
    \begin{aligned}\label{Eq: Moving to the cover}
    &\left|\mathbb{E}_{f_0}\left[\frac{2}{n}\sum_{i=1}^n(\epsilon_i-\mathbb{E}_{f_0}[\epsilon_i\rvert \bX_i])(\widehat{f}(\bX_i)-f_0(\bX_i))\right] \right|\\
    &\leq\left|\mathbb{E}_{f_0}\left[\frac{2}{n}\sum_{i=1}^n(\epsilon_i-\mathbb{E}_{f_0}[\epsilon_i\rvert \bX_i])(\widehat{f}(\bX_i)-f_{k^{*}}(\bX_i))\right] \right|+ (IV)\\
    &\overset{(i)}{\leq} \mathbb{E}_{f_0}\left[\frac{2\delta}{n}\sum_{i=1}^n\big|\epsilon_i-\mathbb{E}_{f_0}[\epsilon_i\rvert \bX_i]\big|\right]+(IV)\\
    &\overset{(ii)}{\leq} 4\delta \|K\|_{\infty}^d2^dF+(IV)
    \end{aligned}
\end{equation}    
where for $(i)$ we used the property of the $\delta$ cover and the triangle inequality, and for $(ii)$ we used Proposition \ref{P:Bound on expected absolute noise Kernel Version}.

In the next step we split the term $(IV)$ into two parts. One case were the $\bX_i$ used for the regression are distributed `nicely' and a second case where we have an extreme concentration of data points $\bX_i.$ The bad second case can be shown to have small probability. For the derivation, we use the bins $\mB_j$ as defined in Section \ref{S: Proofs}.

Define the set $A_j$ as $A_j:=\{\sum_{i=1}^n\1_{\{\bX_i\in\mB_j\}}\leq 2^{d+3}F\log(n)\}$ and the set $A$ as the intersection 
\begin{align}
    A:=\bigcap_{j\in\mJ}A_j.
    \label{eq.A_def}
\end{align}
The two-stage nonparametric density estimator chooses a bandwidth $h_n \leq 2(\log(n)/n)^{1/d}.$ This is equivalent to $2^d\log(n)\geq nh_n^d.$ Together with the union bound, it follows that
\begin{equation}
    \label{eq.6789}
    \begin{aligned}
    \mathbb{P}_{f_0}(A^c)&\leq\sum_{j\in\mJ}\mathbb{P}_{f_0}(A^c_j)\\
    &=\sum_{j\in\mJ}\mathbb{P}_{f_0}\bigg(\sum_{i=1}^n\1_{\{\bX_i\in\mB_j\}}> (7F+F)nh_n^d\bigg)\\
    &\overset{(iii)}{\leq} \sum_{j\in\mJ} \mathbb{P}_{f_0}\bigg(\sum_{i=1}^n\1_{\{\bX_i\in\mB_j\}}> 7Fnh_n^d+np_j\bigg)\\
    &=\sum_{j\in\mJ} \mathbb{P}_{f_0}\bigg(\sum_{i=1}^n\big(\1_{\{\bX_i\in\mB_j\}}-p_j\big)>7Fnh_n^d\bigg)\\
    &\leq \sum_{j\in\mJ}\mathbb{P}_{f_0}\bigg(\bigg|\sum_{i=1}^n\big(\1_{\{\bX_i\in\mB_j\}}-p_j\big)\bigg|>7Fnh_n^d\bigg),
    \end{aligned}
\end{equation}
where for $(iii)$ we used that $p_j=\int_{\mB_j}f_0(\bx)\,d\bx\leq Fh_n^d$ is the probability that an observation falls into bin $\mB_j.$

We now apply the moment version of Bernstein's inequality stated in Proposition \ref{P:BernsteinAlt} (i). For any $m=1,\ldots$
\begin{equation*}
    \begin{aligned}
    \mathbb{E}_{f_0}\left[|\1_{\{\bX_i\in\mB_j\}}|^m\right]=\mathbb{E}_{f_0}\left[\1_{\{\bX_i\in\mB_j\}}\right]=p_j.
    \end{aligned}
\end{equation*}    
Setting $U=1$ and $v=nF h_n^d \geq np_j,$ we get from  Bernstein's inequality in Proposition \ref{P:BernsteinAlt} (i) that
\begin{equation*}
    \begin{aligned}
    \mathbb{P}_{f_0}\bigg(\bigg|\sum_{i=1}^n\big(\1_{\{\bX_i\in\mB_j\}}-p_j\big)\bigg|>7Fnh_n^d\bigg) &\leq 2\exp\left(-\frac{7^2F^2n^2h_n^{2d}}{2n(Fh_n^d+7Fh_n^d)}\right)\\
    &=2\exp\left(-
   \frac{49}{16} Fnh_n^d\right) \\
    &\leq 2\exp(-3nh_n^d) \\
    &\overset{(v)}{\leq}2n^{-3},
    \end{aligned}
\end{equation*}    
where for $(v)$ we used that by construction of the two-stage nonparametric density estimator, $h_n^d\geq \log(n)/n$. Combined with \eqref{eq.6789}, we find
\begin{equation*}
    \begin{aligned}
    \mathbb{P}_{f_0}(A^c)\leq 2\sum_{j\in\mJ}n^{-3}\leq 2n^{-2},
    \end{aligned}
\end{equation*}    
where the last inequality holds because $n\geq 3>e$ implies $|\mJ|=h_n^{-d}\leq n/\log n\leq n$.

With
\begin{align*}
    \xi_k:= \sum_{i=1}^n\big(\epsilon_i-\mathbb{E}_{f_0}[\epsilon_i\rvert \bX_i]\big)\big(f_k(\bX_i)-f_0(\bX_i)\big)\1_{A}
\end{align*}
one can decompose $(IV)$ as follows
\begin{equation}\label{Eq: Split of (IV)}
    \begin{aligned}
    &\left|\mathbb{E}_{f_0}\left[\frac{2}{n}\sum_{i=1}^n\big(\epsilon_i-\mathbb{E}_{f_0}[\epsilon_i\rvert \bX_i]\big)\big(f_{k^*}(\bX_i)-f_0(\bX_i)\big)\right]\right|\\
    &\leq \left|\mathbb{E}_{f_0}\left[\frac{2}{n}\xi_{k^*}\right]\right|+ \left|\mathbb{E}_{f_0}\left[\frac{2}{n}\sum_{i=1}^n\big(\epsilon_i-\mathbb{E}_{f_0}[\epsilon_i\rvert \bX_i]\big)\big(f_{k^*}(\bX_i)-f_0(\bX_i)\big)\1_{A^c}\right]\right|.
    \end{aligned}
\end{equation}    
Moving the absolute value inside, using that $f_{k^*}$ and $f_0$ are bounded by $F$ and applying the Cauchy-Schwarz inequality yields
\begin{equation}\label{Eq: Bound on the Complement term}
    \begin{aligned}
    &\left|\mathbb{E}_{f_0}\left[\frac{2}{n}\sum_{i=1}^n(\epsilon_i-\mathbb{E}_{f_0}[\epsilon_i\rvert \bX_i])(f_{k^*}(\bX_i)-f_0(\bX_i))\1_{A^c}\right]\right|\\
    &\leq\frac{4F}{n}\sum_{i=1}^n\mathbb{E}_{f_0}\left[\left|\epsilon_i-\mathbb{E}_{f_0}[\epsilon_i\rvert \bX_i]\right|\1_{A^c}\right]\\
    &\leq \frac{4F}{n}\sum_{i=1}^n\sqrt{\mathbb{E}_{f_0}\left[\left|\epsilon_i-\mathbb{E}_{f_0}[\epsilon_i\rvert \bX_i]\right|^2\right]}\sqrt{\mathbb{P}_{f_0}(A^c)}\\
    &\overset{(*)}{\leq} \frac{4F\sqrt{2(65F^22^{2d}\|K\|_{\infty}^{2d}})}{n}\\
    & \leq \frac{46 F^22^{d}\|K\|_{\infty}^{d}}{n}.
    \end{aligned}
\end{equation}    
where for $(*)$ we used Proposition \ref{P:Bound on the centered noise squared Kernel Version} and that $\mathbb{P}_{f_0}(A^c)\leq 2n^{-2},$ and for the last inequality we used that $4\sqrt{130}\leq 46.$

It remains to bound the term $|\mathbb{E}_{f_0}[\tfrac 2n \xi_{k^*}]|.$ Let as before $N=n\vee \mN_{\mF}(\delta),$ set
\begin{align}
    z_k:=\sqrt{\log(N)}\vee \sqrt{n}\|f_k-f_0\|_{n}, 
    \label{eq.zk_def}
\end{align}
and define $z_{k^*}$ as $z_k$ for $k=k^*$. The empirical norm of a function $g$ is
\begin{equation*}
    \|g\|_n:=\Bigg(\frac{1}{n}\sum_{i=1}^n(g(\bX_i))^2\Bigg)^{\frac{1}{2}}.
\end{equation*}
Using that $k^*$ is the index of the center of the ball of the $\delta$-cover closest to $\widehat{f},$ it holds that
\begin{equation*}
    \begin{aligned}
    \frac{z_{k^*}}{\sqrt{n}}=\sqrt{\frac{\log(N)}{n}}\vee\|f_{k^*}-f_0\|_{n}\leq \|\widehat{f}-f_0\|_{n}+\delta+\sqrt{\frac{\log(N)}{n}}.
    \end{aligned}
\end{equation*}    
Together with the Cauchy-Schwarz inequality, we obtain
\begin{equation}\label{Eq:GoodTermCSSplit}
    \begin{aligned}
    \left|\mathbb{E}_{f_0}\left[\frac{2}{n}\xi_{k^*}\right]\right|
    &\leq \frac{2}{\sqrt{n}}\mathbb{E}_{f_0}\left[\left|\frac{\xi_{k^*}}{\sqrt{n}}\right|\right]\\
    &\begin{aligned}\leq \frac{2}{\sqrt{n}}&\mathbb{E}_{f_0}\Bigg[\frac{\|\widehat{f}-f_0\|_{n}+\delta+\sqrt{\frac{\log(N)}{n}}}{\frac{z_{k^*}}{\sqrt{n}}}\Bigg|\frac{\xi_{k^*}}{\sqrt{n}}\Bigg|\Bigg]\end{aligned}\\
    &\begin{aligned}{\leq}& 2\frac{\sqrt{\widehat{R}_n(\widehat{f},f_0)}+\delta+\sqrt{\frac{\log(N)}{n}}}{\sqrt{n}}\sqrt{\mathbb{E}_{f_0}\left[\frac{\xi_{k^*}^2}{z_{k^*}^2}\right]}.\end{aligned}
    \end{aligned}
\end{equation}    
For notational ease, define 
\begin{equation}
    \begin{aligned}
    C_{i,k}:=\frac{f_k(\bX_i)-f_0(\bX_i)}{nh_n^dz_k}\1_A.
    \end{aligned}
    \label{eq.Cik_def}
\end{equation}
Since probabilities are always upper bounded by one, we have for any $a>0$ and any square integrable random variable $T$,
$\mathbb{E}[T^2]=  \int_0^{\infty}\mathbb{P}\left(T^2\geq t\right)\, dt=\int_0^{\infty}\mathbb{P}\left(|T|\geq \sqrt{t}\right)\, dt\leq a + \int_a^{\infty}\mathbb{P}\left(|T|\geq \sqrt{t}\right)\, dt.$ Therefore, for any $a>0,$ 
\begin{align}\label{eq.34}
\begin{split}
    \mathbb{E}_{f_0}[\xi_{k^*}^2/z_{k^*}^2 \, \big| \, \bX_1,\ldots,\bX_n]    &\leq\mathbb{E}_{f_0}\big[\max_k \xi_k^2/z_k^2\, \big| \, \bX_1,\ldots,\bX_n\big]\\
    &\leq a+\int_a^{\infty}\mathbb{P}_{f_0}\big(\max_k |\xi_k/z_k|\geq \sqrt{t}\, \big| \, \bX_1,\ldots,\bX_n\big)\, dt.    \end{split}
\end{align}
The ratio $\xi_k/z_k$ can be rewritten as the sum $\sum_{\ell=1}^n h_k(\bX_\ell'),$ where 
\begin{align*}
    h_k(u)=\sum_{i=1}^n \bigg(\prod_{r=1}^d 
    K\bigg(\frac{u_r-X_{i,r}}{h_n}\bigg)
    -\int_{\mathbb{R}^d}
    \prod_{r=1}^d K\bigg(\frac{v_r-X_{i,r}}{h_n}\bigg) f_0(\bv) \, d\bv\bigg) C_{i,k}
\end{align*}
is $\sigma(\bX_1,\ldots,\bX_n)$-measurable. Now let $\widetilde{\bX}_1,\widetilde{\bX}_2,\ldots$ be i.i.d.\ random variables distributed as $\bX$ and independent of the data. Let $\mM$ be a $\Poi(n)$ random variable independent of the data and of the $\widetilde{\bX}_i.$ By the union bound and Poissonization (Lemma \ref{lem.Poissonization}),
\begin{align}
\begin{split}
    &\mathbb{P}_{f_0}\big(\max_k |\xi_k/z_k|\geq \sqrt{t}\, \big| \, \bX_1,\ldots,\bX_n\big) \\
     &\leq N\max_k\mathbb{P}_{f_0}\big( |\xi_k/z_k|\geq \sqrt{t}\, \big| \, \bX_1,\ldots,\bX_n\big) \\
     &= N\max_k\mathbb{P}_{f_0}\bigg( \bigg|\sum_{\ell=1}^n h_k(\bX_\ell')\bigg|\geq \sqrt{t}\, \bigg| \, \bX_1,\ldots,\bX_n\bigg) \\
     &\leq \sqrt{2e\pi n} \, 
     N\max_k\mathbb{P}_{f_0}\bigg( \bigg|\sum_{\ell=1}^{\mathcal{M}} h_k(\wt{\bX}_\ell)\bigg|\geq \sqrt{t}\, \bigg| \, \bX_1,\ldots,\bX_n\bigg).
     \end{split}
     \label{eq.35}
\end{align}
With $W(\bX_i):=\sum_{\ell=1}^{\mM}\prod_{r=1}^{d}K(\frac{\widetilde{X}_{\ell,r}-X_{i,r}}{h_n}),$ we can write
\begin{align}
    \sum_{\ell=1}^{\mathcal{M}} h_k(\wt{\bX}_\ell)
    = \sum_{i=1}^n \big(W(\bX_i)-\mathbb{E}_{f_0}\big[W(\bX_i) |\bX_i\big]\big) C_{i,k}.
    \label{eq.1234}
\end{align}
Next we rewrite the sum over $n$. For this we use the bins $\mB_j$ and the index sets of bins $\mJ_s$ as defined in Section \ref{S: ProofOracle}. Using that the bins are disjoint and that each bin is in exactly one of the $3^d$ index classes $\mJ_s,$ we have $\sum_{i=1}^n=\sum_{s=1}^{3^d}\sum_{j\in\mJ_s}\sum_{\bX_i\in\mB_j}.$ Here we use $\sum_{\bX_i\in\mB_j}$ as shorthand notation for $\sum_{1\leq i\leq n, \text{s.t.} \bX_i\in\mB_j}.$ For non-negative random variables $U_1,\ldots,U_m,$ $\{U_1+\ldots+U_m\geq \sqrt{t}\}\subseteq \cup_{j=1}^m \{U_j\geq \sqrt{t}/m\}$ and by the union bound $\mathbb{P}(U_1+\ldots+U_m\geq \sqrt{t})\leq m \cdot \max_{j=1,\ldots, m} \mathbb{P}(U_j\geq \sqrt{t}/m).$ Combined with \eqref{eq.1234},
\begin{align*}
    &\mathbb{P}_{f_0}\bigg( \bigg|\sum_{\ell=1}^{\mathcal{M}} h_k(\wt{\bX}_\ell)\bigg|\geq \sqrt{t}\, \bigg| \, \bX_1,\ldots,\bX_n\bigg) \\
    &\leq 3^d \max_{s=1,\ldots,3^d} \mathbb{P}_{f_0}\bigg(3^d\bigg|\sum_{j\in\mJ_s}\sum_{\bX_i\in\mB_j} \big(W(\bX_i)-\mathbb{E}_{f_0}\big[W(\bX_i)|\bX_i\big]\big) C_{i,k}\bigg|\geq \sqrt{t} \, \bigg\rvert  \, \bX_1,\ldots,\bX_n\bigg).
\end{align*}
Thus, \eqref{eq.34}, \eqref{eq.35} and the previous display give for any $a>0,$
\begin{equation}\label{Eq:Pre Bernstein Bound}
    \begin{aligned}
    &\mathbb{E}_{f_0}\Big[\frac{\xi_{k^*}^2}{z_{k^*}^2}\, \Big\rvert \, \bX_1,\ldots,\bX_n\Big]\\
    &\begin{aligned}\leq a+&\int_a^{\infty}N3^d\sqrt{2e\pi n}\max_k\max_{s=1,\ldots,3^d}\\
    &\mathbb{P}_{f_0}\bigg(3^d\bigg|\sum\limits_{j\in\mJ_s}\sum\limits_{\bX_i\in\mB_j}\big(W(\bX_i)-\mathbb{E}_{f_0}[W(\bX_i)|\bX_i]\big)C_{i,k}\bigg|\geq \sqrt{t} \, \bigg\rvert \, \bX_1, \ldots, \bX_n\bigg)\,dt.\end{aligned}
    \end{aligned}
\end{equation}
We will now apply Bernstein's inequality in the form of Proposition \ref{P:BernsteinAlt} (i) to the random variables $Z_j=\sum_{\bX_i\in\mB_j}W(\bX_i)C_{i,k}.$ For that we show first that, conditionally on $\bX_1,\ldots,\bX_n,$ the random variables $Z_j, j\in \mJ_s$ with fixed $s$ are jointly independent. To see this, recall that $W(\bX_i):=\sum_{\ell=1}^{\mM}\prod_{r=1}^{d}K(\frac{\widetilde{X}_{\ell,r}-X_{i,r}}{h_n}).$ The kernel $K$ has support in $[-1,1].$ By the definition of the neighborhood $NB(\mB_{j})$ in \eqref{eq.NB_def}, $Z_j$ only depends on the $\wt \bX_1,\ldots, \wt \bX_n$ that fall into $NB(\mB_{j}),$ that is, $Z_j=\sum_{\bX_i\in\mB_j}\sum_{\ell=1}^{\mM}\prod_{r=1}^{d}K(\frac{\widetilde{X}_{\ell,r}-X_{i,r}}{h_n}) C_{i,k}\1_{\{\widetilde{\bX}_{\ell}\in NB(\mB_{j})\}}.$ The variables $C_{i,k}$ defined in \eqref{eq.Cik_def} depend on $\bX_1,\ldots,\bX_n$ but not on $\wt \bX_1,\ldots,\wt \bX_n.$ Working conditionally on $\bX_1,\ldots,\bX_n$ and interchanging the summations, we can write $Z_j=\sum_{\ell=1}^{\mM}g_{j}(\widetilde{\bX}_{\ell})\1_{\{\widetilde{\bX}_{\ell}\in NB(\mB_{j})\}},$ for suitable real-valued functions $g_1,g_2,\ldots$ 
Since the kernel $K$ has support in $[-1,1],$ it follows from the definition of $\mJ_s$ that if two different indices $j$ and $\widetilde{j}$ are both in $\mJ_s,$ then $\{\bx: g_j(\bx)\neq 0\}\cap \{\bx:g_{\widetilde{j}}(\bx)\neq 0\}=\varnothing.$ Let $\mB(\mathbb{R})$ be the Borel $\sigma$-algebra on $\mathbb{R}$ and define $g(\bx)=\sum_{j\in\mJ_s}g_{j}(\bx)\1_{\{\bx\in NB(\mB_{j})\}}.$ The map $T:[0,1]^d\times \mB(\mathbb{R})\rightarrow [0,1]$, given by
\begin{equation*}
    T(\bx,B)=\begin{cases}
        1, & \text{ if } g(\bx)\in B,\\
        0, & \text{ otherwise},
    \end{cases}
\end{equation*}
defines a transition kernel. Since $\widetilde{\bX}_{\ell}$ are i.i.d.\ and $\mM\sim \Poi(n),$ $\sum_{\ell=1}^{\mM}\delta_{\wt \bX_\ell},$ with $\delta_u$ the point measure at $u,$ is a Poisson point process on $[0,1]^d$. The marking theorem states that a Poisson process on a space $A$ and a transition kernel to the Borel algebra of another space $B$ induces a Poisson point process on the product space $A\times B$, see Proposition 4.10.1(b) of \cite{AdventuresinSP} and Chapter 5 of \cite{PoissonProcesses}. Hence together with the transition kernel $T$, we get from the marking theorem, that $\sum_{\ell=1}^{\mM}\delta_{\widetilde{\bX}_{\ell},g(\widetilde{\bX}_{\ell})}$ is a Poisson point process on the product space $[0,1]^d\times \mathbb{R}.$ Since the neighborhoods $NB(\mB_{j})$ are by construction disjoint sets for different $j\in\mJ_s,$ the processes
\begin{equation*}
    \sum_{\ell=1}^{\mM}\delta_{\widetilde{\bX_{\ell}},g_{j}(\widetilde{\bX_{\ell}})}\1_{\{\widetilde{\bX}_{\ell}\in NB(\mB_{j})\}},
\end{equation*} are independent Poisson point processes for different $j\in\mJ_s.$
Hence, conditionally on the $\bX_i$, the random variables  $Z_j=\sum_{\ell=1}^{\mM}g_{j}(\widetilde{\bX}_{\ell})\1_{\{\widetilde{\bX}_{\ell}\in NB(\mB_{j})\}},$ $j\in\mJ_s$ are jointly independent.

To apply Bernstein's inequality, it remains to check that there exists $U$ and $v$ such that
$\sum_{j \in \mJ_s}\mathbb{E}_{f_0}[|Z_j|^m]\leq \frac{1}{2}m!U^{m-2}v,$ for $m=2,3,\ldots$

We have conditionally on $\bX_i$ that
\begin{equation}\label{Eq:Moments of Z}
\begin{aligned}
   & \mathbb{E}_{f_0}\bigg[\bigg|\sum\limits_{\bX_i\in\mB_j}W(\bX_i)C_{i,k}\bigg|^m \, \bigg | \, \bX_1,\ldots,\bX_n\bigg]\\
   & =\mathbb{E}_{f_0}\bigg[\bigg|\sum\limits_{\bX_i\in\mB_j}\bigg(\sum\limits_{\ell=1}^{\mM}\prod_{r=1}^dK\bigg(\frac{\widetilde{X}_{\ell,r}-X_{i,r}}{h_n}\bigg)\bigg)C_{i,k}\bigg|^m \, \bigg | \, \bX_1,\ldots,\bX_n \bigg]\\
     &\overset{(i)}{\leq} \mathbb{E}_{f_0}\bigg[\bigg(\sum\limits_{\bX_i\in\mB_j}\bigg(\sum\limits_{\ell=1}^{\mM}\bigg|\prod_{r=1}^dK\bigg(\frac{\widetilde{X}_{\ell,r}-X_{i,r}}{h_n}\bigg)\bigg|\bigg)|C_{i,k}|\bigg)^m \, \bigg | \, \bX_1,\ldots,\bX_n \bigg]\\
     &\overset{(ii)}{=}
     \mathbb{E}_{f_0}\bigg[\bigg(\sum\limits_{\bX_i\in\mB_j}\bigg(\sum\limits_{\ell=1}^{\mM}\bigg|\prod_{r=1}^dK\bigg(\frac{\widetilde{X}_{\ell,r}-X_{i,r}}{h_n}\bigg)\bigg|\1_{\{\widetilde{\bX}_{\ell}\in NB(\mB_{j})\}}\bigg)|C_{i,k}|\bigg)^m \, \bigg | \, \bX_1,\ldots,\bX_n \bigg]\\
    &\overset{(iii)}{\leq}\mathbb{E}_{f_0}\bigg[\bigg(\sum\limits_{\bX_i\in\mB_j}\bigg(\sum\limits_{\ell=1}^{\mM}\|K\|_{\infty}^d\1_{\{\widetilde{\bX}_{\ell}\in NB(\mB_{j})\}}\bigg)|C_{i,k}|\bigg)^m\, \bigg | \, \bX_1,\ldots,\bX_n \bigg]\\
    &\overset{(iv)}{=}\mathbb{E}_{f_0}\bigg[\bigg(\sum\limits_{\ell=1}^{\mM}\|K\|_{\infty}^d\1_{\{\widetilde{\bX}_{\ell}\in NB(\mB_{j})\}}\bigg)^m\bigg(\sum\limits_{\bX_i\in\mB_{j}}|C_{i,k}|\bigg)^m \, \bigg | \, \bX_1,\ldots,\bX_n \bigg]\\
    &\overset{(v)}{=}\|K\|_{\infty}^{dm}\bigg(\sum\limits_{\bX_i\in\mB_{j}}|C_{i,k}|\bigg)^m\mathbb{E}_{f_0}\bigg[\bigg(\sum\limits_{\ell=1}^{\mM}\1_{\{\widetilde{\bX}_{\ell}\in NB(\mB_{j})\}}\bigg)^m \bigg].
\end{aligned}
\end{equation}
Where $(i)$ follows from triangle inequality. For $(ii)$ we used that $\bX_i\in\mB_{j}$ and that $K$ has support in $[-1,1],$ so if $\widetilde{\bX}_{\ell}$ is outside $NB(\mB_{j})$ then $\prod_{r=1}^dK\left(\frac{\widetilde{X}_{\ell,r}-X_{i,r}}{h_n}\right)=0.$ For $(iii)$ we use that $\|K\|_{\infty}<\infty$ and that all terms are non-negative. The equality $(iv)$ follows from observing that $\sum_{\ell=1}^{\mM}\|K\|_{\infty}^d\1_{\{\widetilde{\bX}_{\ell}\in NB(\mB_{j})\}}$ does not depend on $i$ and can be taken out of the sum. Finally $(v)$ follows by taking all the constants out of the expectation, recalling that $C_{i,k}$ is $\sigma(\bX_1,\ldots,\bX_n)$-measurable.

Since $\widetilde{\bX}_{\ell}$ are i.i.d.\ and $\mM\sim \Poi(n),$ we have $\sum_{\ell=1}^{\mM}\1_{\{\widetilde{\bX}_{\ell}\in NB(\mB_{j})\}} \sim \Poi(n\widetilde{p}_j),$ where $\widetilde{p}_j$ denotes the probability that $\bX\in NB(\mB_j)$. Expressing the moments of the Poisson distribution as Bell polynomials \cite{MR4347487} gives 
\begin{equation*}
    \begin{aligned}
    \mathbb{E}_{f_0}\bigg[\bigg(\sum\limits_{\ell=1}^{\mM}\1_{\{\widetilde{\bX}_{\ell}\in NB(\mB_{j})\}}\bigg)^m\bigg]=\sum\limits_{t=0}^m(n\widetilde{p}_j)^t\stirling{m}{t}\leq (n\widetilde{p}_j\vee1)^m\sum\limits_{t=0}^m\stirling{m}{t},
    \end{aligned}
\end{equation*}    
where $\stirling{m}{t}$ denote the Stirling numbers of the second kind. The $m$-th Bell number equals the sum $\sum_{t=0}^m\stirling{m}{t}.$ Applying now the bound on Bell numbers derived in Theorem 2.1 of \cite{berend2010improved} gives
\begin{equation*}
    \begin{aligned}
    \sum\limits_{t=0}^m \stirling{m}{t}\leq \left(\frac{m}{\log(m+1)}\right)^m.
    \end{aligned}
\end{equation*}
Due to $m\geq 2,$ $\log(m+1)\geq \log(3)> 1$ and the right hand side of the previous display can be upper bounded by $m^m.$ Using Stirling's formula (\cite{RobbinsStirlingFormula}) again, we get that $\sqrt{2\pi m}m^me^{-m}\leq m!.$ Since $m\geq 2,$ $\sqrt{2\pi m}\geq e$ and $m^m\leq m!e^{m-1}.$ Thus 
\begin{equation*}
    \begin{aligned}
  \mathbb{E}_{f_0}\bigg[\bigg(\sum\limits_{\ell=1}^{\mM}\1_{\{\widetilde{\bX}_{\ell}\in NB(\mB_{j})\}}\bigg)^m\bigg]\leq m!e^{m-1} (n\widetilde{p}_j\vee 1)^m\leq  m!e^{m-1}(F3^dnh_n^d)^m.
    \end{aligned}
\end{equation*}    
The last inequality follows from observing that $\widetilde{p}_j\leq F3^dh_n^d$ (the upper bound on $f_0$ times the Lebesgue measure of $NB(\mB_j)$) and that $3^dFnh_n^d\geq3^dF\log(n)\geq 1$.
Combined with \eqref{Eq:Moments of Z}, this leads to 
\begin{equation*}
    \begin{aligned}
    \mathbb{E}_{f_0}\bigg[\bigg|\sum\limits_{\bX_i\in\mB_{j}}W(\bX_i)C_{i,k}\bigg|^m \bigg|\, \bX_1,\ldots,\bX_n \bigg]
    \leq  m!e^{m-1} (F3^dnh_n^d)^m \|K\|_{\infty}^{dm}\bigg(\sum\limits_{\bX_i\in\mB_{j}}|C_{i,k}|\bigg)^m.
    \end{aligned}
\end{equation*}
The previous inequality suggests to take the parameters $v$ and $U$ in Bernstein's inequality as upper bounds of $\sum_{j\in\mJ_s}(e\|K\|_{\infty}^d)^2(F3^dnh_n^d)^2(\sum_{\bX_i\in\mB_{j}}|C_{i,k}|)^2$ and $e\|K\|_{\infty}^d3^dFnh_n^d \sum_{\bX_i\in\mB_{j}}|C_{i,k}|,$ respectively. To find a convenient expression for $v,$ observe that
\begin{equation*}
    \begin{aligned}
    &\sum\limits_{j\in\mJ_s}(e\|K\|_{\infty}^d)^2(F3^dnh_n^d)^2\bigg(\sum\limits_{\bX_i\in\mB_{j}}|C_{i,k}|\bigg)^2\\
    &=\sum\limits_{j\in\mJ_s}(e3^d\|K\|_{\infty}^dF)^2n^2h_n^{2d}\bigg(\sum\limits_{\bX_i\in\mB_{j}}\left|\frac{(f_k(\bX_i)-f_0(\bX_i))}{nh_n^dz_k}\1_A\right|\bigg)^2\\
    &=\sum\limits_{j\in\mJ_s}\frac{(e3^d\|K\|_{\infty}^dF)^2}{z_k^2}\bigg(\sum\limits_{\bX_i\in\mB_{j}}\left|f_k(\bX_i)-f_0(\bX_i)\right|\1_A\bigg)^2.
    \end{aligned}
\end{equation*}
By the Cauchy-Schwarz inequality,
\begin{equation*}
    \begin{aligned}
    \bigg(\sum\limits_{\bX_i\in\mB_j}|f_k(\bX_i)-f_0(\bX_i)|\1_A\bigg)^2
    &\leq \bigg(\1_A\sum\limits_{\bX_i\in\mB_j}1^2\bigg)\bigg(\sum\limits_{\bX_i\in\mB_j}\big(f_k(\bX_i)-f_0(\bX_i)\big)^2\bigg)\\
    &\leq 2^{d+3}F\log(n)\sum\limits_{\bX_i\in\mB_j}\big(f_k(\bX_i)-f_0(\bX_i)\big)^2,
    \end{aligned}
\end{equation*}    
where for the last inequality we used that the definition of the event $A$ in \eqref{eq.A_def} implies $\sum_{\bX_i\in\mB_j}1\leq 2^{d+3}F\log(n).$ By \eqref{eq.zk_def},  $z_k\geq \sqrt{n}\|f_k-f_0\|_{n}.$ Moreover, $\sum_{i=1}^n =\sum_{j \in \mJ_s}\sum_{\bX_i\in\mB_j}$ and thus
\begin{equation*}
    \begin{aligned}
    &\sum\limits_{j\in\mJ_s}(e\|K\|_{\infty}^d)^2(F3^dnh_n^d)^2\bigg(\sum\limits_{\bX_i\in\mB_{j}}|C_{i,k}|\bigg)^2\\
    &\leq \sum\limits_{j\in\mJ_s}\frac{(e3^d\|K\|_{\infty}^dF)^2}{n\|f_k-f_0\|_{n}^2}2^{d+3}F\log(n)\bigg(\sum\limits_{\bX_i\in\mB_j}(f_k(\bX_i)-f_0(\bX_i))^2\bigg)\\
    &=2^{d+3}\frac{(e3^d\|K\|_{\infty}^d)^2}{n\|f_k-f_0\|_{n}^2}F^3\log(n)n\|f_k-f_0\|_{n}^2\\
    &=2^{d+3}(e3^d\|K\|_{\infty}^d)^2F^3\log(n).
    \end{aligned}
\end{equation*}
Hence we can take $v=2^{d+3}(e3^d\|K\|_{\infty}^d)^2F^3\log(n)$ in Bernstein's inequality.

To obtain a convenient expression for the $U$ in Bernstein's inequality, we now bound $\sum_{\bX_i\in\mB_{j}}|C_{i,k}|.$ Using that by \eqref{eq.zk_def}, $z_k\geq \sqrt{\log(N)}$, that $f_k$ and $f_0$ are bounded by $F,$ and that on the event $A,$  $\sum_{\bX_i\in\mB_j}\leq 2^{d+3}F\log(n)$ gives
\begin{equation*}
    \begin{aligned}
    \sum\limits_{\bX_i\in\mB_{j}}|C_{i,k}|=\sum\limits_{\bX_i\in\mB_{j}}\frac{|f_k(\bX_i)-f_0(\bX_i)|}{nh_n^dz_k}\1_A
    &\leq \frac{2F}{nh_n^d\sqrt{\log(N)}}\sum\limits_{\bX_i\in\mB_{j}}\1_A\leq\frac{2^{d+3}F^2\log(n)}{nh_n^d\sqrt{\log(N)}}.
    \end{aligned}
\end{equation*}
Hence it holds that
\begin{equation*}
    \begin{aligned}
    e\|K\|_{\infty}^d3^dFnh_n^d\bigg(\sum\limits_{\bX_i\in\mB_{j}}|C_{i,k}|\bigg)
    \leq \frac{2^{d+3}e\|K\|_{\infty}^d3^dF^3\log(n)}{\sqrt{\log(N)}}.
    \end{aligned}
\end{equation*}
The support of the kernel is contained in $[-1,1].$ This means that $1\leq 2\|K\|_\infty$ and consequently, $e 3^{d} \|K\|_{\infty}^{d}\geq 1.$ Thus, setting $U=v/\sqrt{\log(N)}$ with $v= 2^{d+3}(e3^d\|K\|_{\infty}^d)^2F^3\log(n)$, as above, we find $U\geq 2^{d+3}e\|K\|_{\infty}^d3^dF^3\log(n)/\sqrt{\log(N)}$ and obtain
\begin{equation*}
    \sum\limits_{j\in\mJ_s} \mathbb{E}_{f_0}\bigg[\bigg|\sum\limits_{\bX_i\in\mB_{j}}W(\bX_i)C_{i,k}\bigg|^m\bigg]\leq  \frac{m!}{2}vU^{m-2},
\end{equation*}
for all $m=2,3,\ldots.$ Consequently we can apply Bernstein's inequality with those choices for $U$ and $v.$

Applying Bernstein's inequality on the sum over the variables $Z_j,$ with $v$ and the bound $U$ as defined above, we get that
\begin{equation*}
    \begin{aligned}
    &\mathbb{P}_{f_0}\bigg(3^d\bigg|\sum\limits_{j\in\mJ_s}\sum\limits_{\bX_i\in\mB_j}\big(W(\bX_i)-\mathbb{E}_{f_0}[W(\bX_i)|\bX_i]\big)C_{i,k}\bigg|\geq \sqrt{t}\, \bigg\rvert \, \bX_1,\ldots,\bX_n\bigg)\\
    &=\mathbb{P}_{f_0}\bigg(\bigg|\sum\limits_{j\in\mJ_s}\left(Z_j-\mathbb{E}_{f_0}[Z_j|\bX_i]\right)\bigg|\geq 3^{-d}\sqrt{t}\, \bigg\rvert \, \bX_1,\ldots,\bX_n\bigg)\\
    &\leq 2\exp\left(-\frac{t3^{-2d}}{2\left(v+U3^{-d}\sqrt{t}\right)}\right)\\
    &= 2\exp\left(-\frac{t3^{-2d}}{2v\left(1+3^{-d}\sqrt{t/ \log(N)}\right)}\right).
    \end{aligned}
\end{equation*}
If $t\geq 3^{2d}\log(N),$ the previous expression can be further bounded by
\begin{align}\label{Eq:TailBoundLarget}
    \leq 2\exp\left(-\frac{\sqrt{t\log(N)} 3^{-d}}{4v}\right).
\end{align}

Observe that this gives us an upper bound that is the same for all collections of bins $\mJ_s$ and all cover centers $k$. Choosing $a=64v^23^{2d}\log(N)$ in \eqref{Eq:Pre Bernstein Bound} gives
\begin{equation*}
    \begin{aligned}
    &\mathbb{E}_{f_0}\Big[\frac{\xi_{k^*}^2}{z_{k^*}^2} \Big\rvert \bX_1,\ldots,\bX_n\Big]\\
    &\overset{(i)}{\leq}64v^23^{2d}\log(N)+2N3^d\sqrt{2e\pi n}\int_{64v^23^{2d}\log(N)}^{\infty}\exp\left(-\sqrt{t}\frac{\sqrt{\log(N)} 3^{-d}}{4v}\right)\,dt\\
    &\overset{(ii)}{=}64v^23^{2d}\log(N)+4N3^d\sqrt{2e\pi n}
    (16v^23^{2d})\frac{\left(2\log(N)+1\right)}{\log(N)}\exp\left(-2\log(N)\right)\\
    &\overset{(iii)}{\leq}64v^23^{2d}\log(N)+1280
    v^23^{3d}\\
    &\overset{(iv)}{\leq}(2^{d+5}10(e\|K\|_{\infty}^d)^2F^3\log(n))^23^{7d}\log(N),
    \end{aligned}
\end{equation*}
where for $(i)$ we used \eqref{Eq:TailBoundLarget} combined with the observation that if $t\geq 64v^23^{2d}\log(N)$ then $t\geq 3^{2d}\log(N)$, since $v\geq 1$. For $(ii)$ we used that $\int_{b^2}^{\infty}e^{-\sqrt{u}c}du=2\int_b^{\infty}se^{-sc}ds=2(bc+1)e^{-bc}/c^2$. For $(iii)$ we used that $\log(N)\geq 1$ so $(2\log(N)+1)/\log(N)\leq 4$ and $N=n\vee \mN_{\mF}(\delta)\geq n$, $\sqrt{2e\pi}\leq 5$, $\log(N)\geq 1.$ For $(iv)$ we substituted $v=2^{d+3}(e3^d\|K\|_{\infty}^d)^2F^3\log(n)$ and used that $\sqrt{1344}=4\sqrt{84}$ and $\sqrt{84}\leq 10$.

Together with \eqref{Eq:GoodTermCSSplit}, this yields
\begin{equation*}
    \begin{aligned}
    &\left|\mathbb{E}_{f_0}\left[\frac{2}{n}\xi_{k^*}\right]\right| \\
    &{\leq} 2\frac{\sqrt{\widehat{R}_n(\widehat{f},f_0)}+\delta+\sqrt{\frac{\log(N)}{n}}}{\sqrt{n}}\sqrt{\mathbb{E}_{f_0}\left[\frac{\xi_{k^*}^2}{z_{k^*}^2}\right]} \\
    &\leq 2\frac{\sqrt{\widehat{R}_n(\widehat{f},f_0)}+\delta+\sqrt{\frac{\log(N)}{n}}}{\sqrt{n}}\sqrt{\left(2^{d+5}10e^2\|K\|_{\infty}^{2d}F^3\log(n)\right)^23^{7d}\log(N)}\\
    &=\left(\sqrt{\widehat{R}_n(\widehat{f},f_0)}+\delta+\sqrt{\frac{\log(N)}{n}}\right) 2^{d+6}10e^2 \|K\|_{\infty}^{2d}F^3\log(n)\sqrt{\frac{3^{7d}\log(N)}{n}}.
    \end{aligned}
\end{equation*}
Inserting this bound in \eqref{Eq: Split of (IV)} together with \eqref{Eq: Bound on the Complement term} gives a bound for (IV). Together with \eqref{Eq: Moving to the cover} and \eqref{Eq: Bound on term (III)} and combining the terms with $\delta,$ using that by assumption $ \log^2(n)\log(N)=\log^2(n)\log(n\vee \mN_{\mF}(\delta))\leq n,$ finishes the proof.\end{proof}

\begin{proof}[Proof of Proposition \ref{P: Bound empirical risk}]
Expanding the square yields
\begin{align*}
    \big(\widehat{f}(\bX_i)-f_0(\bX_i)\big)^2
    &= \big(\widehat{f}(\bX_i)-Y_i+Y_i-f_0(\bX_i)\big)^2 \\
    &=(\widehat{f}(\bX_i)-Y_i)^2+2(\widehat{f}(\bX_i)-Y_i)(Y_i-f_0(\bX_i))+(Y_i-f_0(\bX_i))^2.
\end{align*}
    We use this identity to rewrite the definition $\widehat{R}_n(\widehat{f},f_0)=\mathbb{E}_{f_0}[\frac{1}{n}\sum_{i=1}^n(\widehat{f}(\bX_i)-f_0(\bX_i))^2]$. Applying moreover that for any fixed $f\in\mF,$ we have by definition of $\Delta_n(\widehat{f},f_0)$ that
    \begin{equation*}
       \mathbb{E}_{f_0}\bigg[\frac{1}{n}\sum_{i=1}^n(Y_i-\widehat{f}(\bX_i))^2\bigg]\leq \mathbb{E}_{f_0}\bigg[\frac{1}{n}\sum_{i=1}^n(Y_i-f(\bX_i))^2\bigg]+\Delta_n(\widehat{f},f_0)
    \end{equation*}
yields
\begin{equation*}
\begin{aligned}
 &\widehat{R}_n(\widehat{f},f_0) \\
 &= \mathbb{E}_{f_0}\left[\frac{1}{n}\sum\limits_{i=1}^n\left((\widehat{f}(\bX_i)-Y_i)^2+2(\widehat{f}(\bX_i)-Y_i)(Y_i-f_0(\bX_i))+(Y_i-f_0(\bX_i))^2\right)\right]\\
 &\leq \mathbb{E}_{f_0}\left[\frac{1}{n}\sum\limits_{i=1}^n\left((f(\bX_i)-Y_i)^2+2(\widehat{f}(\bX_i)-Y_i)(Y_i-f_0(\bX_i))+(Y_i-f_0(\bX_i))^2\right)\right]+\Delta_n(\widehat{f},f_0)\\
 &=\mathbb{E}_{f_0}\left[\frac{1}{n}\sum\limits_{i=1}^n\left((f(\bX_i)-Y_i)^2+2(f(\bX_i)-Y_i)(Y_i-f_0(\bX_i))+(Y_i-f_0(\bX_i))^2\right)\right]\\
 &\quad +\mathbb{E}_{f_0}\left[\frac{2}{n}\sum\limits_{i=1}^n(Y_i-f_0(\bX_i))(\widehat{f}(\bX_i)-f(\bX_i))\right]+\Delta_n(\widehat{f},f_0)\\
 &=\mathbb{E}_{f_0}\left[\frac{1}{n}\sum\limits_{i=1}^n(f(\bX_i)-f_0(\bX_i))^2\right]+\mathbb{E}_{f_0}\left[\frac{2}{n}\sum\limits_{i=1}^n(Y_i-f_0(\bX_i))(\widehat{f}(\bX_i)-f(\bX_i))\right]+\Delta_n(\widehat{f},f_0)\\
 &=\mathbb{E}_{\bX}\left[(f(\bX)-f_0(\bX))^2\right]+\mathbb{E}_{f_0}\left[\frac{2}{n}\sum\limits_{i=1}^n\epsilon_i(\widehat{f}(\bX_i)-f(\bX_i))\right]+\Delta_n(\widehat{f},f_0),
\end{aligned}
\end{equation*}
where for the last equality we used that the $\bX_i$ are independent and have the same distribution as $\bX$.

Combined with Lemma \ref{L: Bound on the noise-dependend term}, this yields
\begin{equation*}
    \begin{aligned}
    &\widehat{R}_n(\widehat{f},f_0)\leq \mathbb{E}_{\bX}\left[(f(\bX)-f_0(\bX))^2\right]\\
    &+\sqrt{\widehat{R}_n(\widehat{f},f_0)}2^{d+6}14e^2 \|K\|_{\infty}^{2d}F^3\log(n)\sqrt{\frac{3^{7d}\log(n\vee \mN_{\mF}(\delta))}{n}} \\
    &+2^{d+6}14e^2 \|K\|_{\infty}^{2d}F^33^{\frac{7d}{2}}\log(n)\frac{\log(n\vee \mN_{\mF}(\delta))}{n} \\
    &+\delta2^{d+6}14e^2 \|K\|_{\infty}^{2d}F^33^{\frac{7d}{2}}+\frac{46F^22^{d}\|K\|_{\infty}^{d}}{n}\\
    &+8h_n^{2\beta}F^2d^{2\beta} \|K\|_\infty^{2d} +\frac{\mathbb{E}_{\bX}[(f_0(\bX)-f(\bX))^2]}{4}+\frac{\widehat{R}_n(\widehat{f},f_0)}{4}+\Delta_n(\widehat{f},f_0).
    \end{aligned}
\end{equation*}
Rewriting this and upper bounding constants, yields
\begin{equation*}
    \begin{aligned}
\widehat{R}_n(\widehat{f},f_0)&\leq \frac{5}{3}\mathbb{E}_{\bX}\left[(f(\bX)-f_0(\bX))^2\right] \\
&\quad+\sqrt{\widehat{R}_n(\widehat{f},f_0)}2^{d+6}19e^2 \|K\|_{\infty}^{2d}F^3\log(n)\sqrt{\frac{3^{7d}\log(n\vee \mN_{\mF}(\delta))}{n}}\\
&\quad +2^{d+6}19e^2 \|K\|_{\infty}^{2d}F^33^{\frac{7d}{2}}\log(n)\frac{\log(n\vee \mN_{\mF}(\delta))}{n}+\delta2^{d+6}19e^2 \|K\|_{\infty}^{2d}F^33^{\frac{7d}{2}} \\
    &\quad +\frac{62F^22^{d}\|K\|_{\infty}^d}{n}+11h_n^{2\beta}F^2d^{2\beta} \|K\|_\infty^{2d} +\frac{4}{3}\Delta_n(\widehat{f},f_0).
    \end{aligned}
\end{equation*}
For real numbers $a,c,d,$ satisfying $|a|\leq 2\sqrt{a}c+d,$ we have $|a|\leq 2\sqrt{a}c+d\leq \tfrac 12 |a|+2c^2+d$ and thus $|a|\leq 2d+4c^2.$ Applying this inequality with $a=\widehat{R}_n(\widehat{f},f_0),$
\begin{equation*}
    c=2^{d+6}19e^2 \|K\|_{\infty}^{2d}F^3\log(n)\sqrt{\frac{3^{7d}\log(n \vee \mN_{\mF}(\delta))}{n}},
\end{equation*}
and
\begin{equation*}
    \begin{aligned}
    d=&\delta2^{d+6}19e^2 \|K\|_{\infty}^{2d}F^33^{\frac{7d}{2}}+\frac{62F^22^{d}\|K\|_{\infty}^{d}}{n}\\
&+2^{d+6}19e^2 \|K\|_{\infty}^{2d}F^33^{\frac{7d}{2}}\log(n)\frac{\log(n \vee \mN_{\mF}(\delta))}{n} \\
    &+11h_n^{2\beta}F^2d^{2\beta} \|K\|_\infty^{2d} +\frac{4}{3}\Delta_n(\widehat{f},f_0)+\frac{5}{3}\mathbb{E}_{\bX}\left[(f(\bX)-f_0(\bX))^2\right]
    \end{aligned}
\end{equation*}
yields the result.
\end{proof}

\begin{Proposition}\label{P:Bound on conditional noise Kernel version}
$\left|\mathbb{E}_{f_0}[\epsilon_i\rvert\bX_i]\right|\leq h_n^{\beta} d^\beta \|K\|_\infty^d F.$
\end{Proposition}
\begin{proof}
By the construction of the $\epsilon_i$ in \eqref{eq.YiKDE} and \eqref{eq.regr_model}, 
$\epsilon_i=\KDE(\bX_i)-f_0(\bX_i).$ Using moreover the definition of the multivariate kernel density estimator in \eqref{eq.KDE_def} and writing $|\bv|^{\balpha}$ for $|v_1|^{\alpha_1}\cdot\ldots \cdot |v_d|^{\alpha_d},$ we obtain
\begin{equation*}
    \begin{aligned}
    &\big|\mathbb{E}_{f_0}[\epsilon_i\rvert\bX_i]\big|=\left|\mathbb{E}_{f_0}\left[\frac{1}{nh_n^d}\sum\limits_{\ell=1}^{n}\prod_{r=1}^{d}K\left(\frac{X'_{\ell,r}-X_{i,r}}{h_n}\right)-f_0(\bX_i)\Bigg|\bX_i\right]\right|\\
    &\overset{(i)}{=}\left|\frac{1}{h_n^d}\int_{[0,1]^d}f_0(\bu)\prod_{r=1}^{d} K\left(\frac{u_{r}-X_{i,r}}{h_n}\right)\, d\bu -f_0(\bX_i)\right|\\
    & \overset{(ii)}{=}\left|\int_{\mathbb{R}^d}\left(\prod_{r=1}^{d} K\left(v_r\right)\right)f_0(X_{i,1}+v_1h_n,\ldots,X_{i,d}+v_dh_n)\, d\bv -f_0(\bX_i)\right|\\
    & \overset{(iii)}{=}\left|\int_{\mathbb{R}^d}\left(\prod_{r=1}^{d} K\left(v_r\right)\right)\Big(f_0(X_{i,1}+v_1h_n,\ldots,X_{i,d}+v_dh_n) -f_0(\bX_i)\Big)\, d\bv\right|\\
    &\begin{aligned}\overset{(iv)}{=}&\Bigg| \int_{\mathbb{R}^d}\left(\prod_{r=1}^{d} K\left(v_r\right)\right)\\
    &\begin{aligned}\times\Bigg(\sum\limits_{\balpha:|\balpha|_1\leq \floor{\beta}-1, \balpha\neq 0}&\frac{(h_n\bv)^{\balpha}}{\balpha!}(\partial^{\balpha} f_0)(\bX_i)+\sum\limits_{\balpha:|\balpha|_1=\floor{\beta}}\frac{(h_n\bv)^{\balpha}}{\balpha!}(\partial^{\balpha} f_0)(\bX_i+h_n\tau\bv)\Bigg)\, d\bv\Bigg|\end{aligned}\end{aligned}\\
    &\begin{aligned}\overset{(v)}{=}\Bigg|\int_{\mathbb{R}^d}&\left(\prod_{r=1}^{d} K\left(v_r\right)\right)\left(\sum\limits_{\balpha:|\balpha|_1=\floor{\beta}}\frac{(h_n\bv)^{\balpha}}{\balpha!}\left((\partial^{\balpha} f_0)(\bX_i+h_n\tau\bv)-(\partial^{\balpha} f_0)(\bX_i)\right)\right)\, d\bv\Bigg|\end{aligned}\\
    &\begin{aligned}\overset{(vi)}{\leq} h_n^{\floor{\beta}}\int_{\mathbb{R}^d}&\left|\prod_{r=1}^{d} K\left(v_r\right)\right|\left(\sum\limits_{\balpha:|\balpha|_1=\floor{\beta}}\frac{|\bv|^{\balpha}}{\balpha!}\left|(\partial^{\balpha} f_0)(\bX_i+h_n\tau\bv)-(\partial^{\balpha} f_0)(\bX_i)\right|\right)\, d\bv\end{aligned}\\
    &\overset{(vii)}{\leq}h_n^{\floor{\beta}}\|K\|_\infty^d \int_{[0,1]^d}\left(\sum\limits_{\balpha:|\balpha|_1=\floor{\beta}}\frac{|\bv|^{\balpha}}{\balpha!}\left|h_n\tau\bv\right|_{\infty}^{\beta-\floor{\beta}}F\right)\, d\bv\\
    &\overset{(viii)}{\leq}h_n^{\beta}\|K\|_\infty^d F\int_{[0,1]^d}\sum\limits_{\balpha:|\balpha|_1=\floor{\beta}}\frac{1}{\balpha!}\, d\bv\\
    &\overset{(ix)}{\leq} h_n^{\beta}\|K\|_\infty^d d^\beta F.
    \end{aligned}
\end{equation*}
Here we used for $(i)$ that the $\bX'_{\ell}$ are i.i.d.\ and independent of $\bX_i.$ For $(ii)$ we substituted the transformed variables $v_r=(u_r-X_{i,r})/h_n$ and used that $f_0$ vanishes outside $[0,1]^d,$ since $f_0$ has support in $[0,1]^d$ and is continuous on $\mathbb{R}^d$. For $(iii)$ we used that a kernel integrates to $1$ and that $f_0(\bX_i)$ is a constant with respect to the integration variables. Step $(iv)$ applies $\floor{\beta}$-order Taylor expansion, that is, for a suitable $\tau\in (0,1),$
\begin{equation*}
    \begin{aligned}
    &f_0(\bX_i+h_n\bv)\\
    &=f_0(\bX_i)+\sum\limits_{\balpha:|\balpha|_1\leq \floor{\beta}-1, \balpha\neq 0}\frac{(h_n\bv)^{\balpha}}{\balpha!}(\partial^{\balpha} f_0)(\bX_i)+\sum\limits_{\balpha:|\balpha|_1=\floor{\beta}}\frac{(h_n\bv)^{\balpha}}{\balpha!}(\partial^{\balpha} f_0)(\bX_i+h_n\tau\bv),
    \end{aligned}
\end{equation*}
see Theorem 2.2.5 in \cite{MR3467258}. For $(v)$ we used that $K$ is a kernel of order $\floor{\beta}$ and therefore $\int v^mK(v)\,dv=0$ for all $m=1,\ldots,\floor{\beta}.$ For $(vi)$ we used that $h_n^{\floor{\beta}}$ appears in every term of the sum. Jensen's inequality and triangle inequality are moreover applied to move the absolute value inside the integral and the sum. For $(vii)$ we used that $f_0$ is in the $\beta$-H\"{o}lder ball with radius $F$ and that $K$ has support contained in $[-1,1].$ For $(viii)$ we used that $|\tau|\leq 1.$ To see $(ix),$ observe that for the multinomial distribution with number of trials $\lfloor \beta \rfloor$ and $d$ event probabilities $(1/d,\ldots, 1/d),$ we have 
 \begin{align*}
     1=\sum\limits_{\balpha:|\balpha|_1=\floor{\beta}} \frac{\floor{\beta} !}{\balpha!} \Big(\frac{1}{d}\Big)^{\alpha_1}\cdot \ldots \cdot \Big(\frac{1}{d}\Big)^{\alpha_d}= \floor{\beta}!d^{-\floor{\beta} }\sum\limits_{\balpha:|\balpha|_1=\floor{\beta}} \frac{1}{\balpha!} \geq d^{-\beta }\sum\limits_{\balpha:|\balpha|_1=\floor{\beta}} \frac{1}{\balpha!}.
 \end{align*}
\end{proof}

\begin{Proposition}\label{P:Bound on expected absolute noise Kernel Version}
$\mathbb{E}_{f_0}[|\epsilon_i-\mathbb{E}_{f_0}[\epsilon_i\rvert \bX_i]|]\leq F\|K\|_{\infty}^d2^{d+1}.$
\end{Proposition}
\begin{proof}
By definition, $\epsilon_i=Y_i-f_0(\bX_i)$. Together with conditioning on $\bX_i$, triangle inequality and Jensen's inequality this yields
\begin{equation}\label{Eq: From epsilon difference to bound on Y}
    \begin{aligned}
    \mathbb{E}_{f_0}[|\epsilon_i-\mathbb{E}_{f_0}[\epsilon_i\rvert \bX_i]|]&=\mathbb{E}_{f_0}[|Y_i-\mathbb{E}_{f_0}[Y_i\rvert  \bX_i]|] \\   &\leq 2\mathbb{E}_{f_0}\left[\mathbb{E}_{f_0}[|Y_i|\rvert \bX_i]\right]\\
    &\leq 2\sum_{\ell=1}^n\frac{1}{nh_n^d}\mathbb{E}_{f_0}\left[\mathbb{E}_{f_0}\left[\prod_{r=1}^{d}\left|K\left(\frac{X'_{\ell,r}-X_{i,r}}{h_n}\right)\right|\, \Big\rvert \, \bX_i\right]\right].
    \end{aligned}
\end{equation}    
Using that $\|f_0\|_\infty \leq F$ and the kernel $K$ is supported on $[-1,1],$ we get by substitution 
\begin{equation*}
    \begin{aligned}
    \mathbb{E}_{f_0}\left[\prod_{r=1}^{d}\left|K\left(\frac{X'_{\ell,r}-X_{i,r}}{h_n}\right)\right|\, \Big\rvert \, \bX_i\right]&\leq F\int_{\mathbb{R}^d}\prod_{r=1}^{d} \left| K\left(\frac{u_{r}-X_{i,r}}{h_n}\right)\right| \, d\bu\\
    &=Fh_n^d\int_{\mathbb{R}^d} \prod_{r=1}^{d} \big|K(v_r)\big|\, d\bv\\
    &\leq F\|K\|_\infty^{d} 2^d h_n^d .
    \end{aligned}
\end{equation*}
\end{proof}

\begin{Proposition}\label{P:Bound on the centered noise squared Kernel Version}
$\mathbb{E}_{f_0}\left[\left|\epsilon_i-\mathbb{E}_{f_0}[\epsilon_i\rvert\bX_i]\right|^2\right]\leq 65F^22^{2d}\|K\|_{\infty}^{2d}.$
\end{Proposition}
\begin{proof}
By definition, $\epsilon_i=Y_i-f_0(\bX_i).$ For a non-negative random-variable $T,$ it holds that $\mathbb{E}[T^2]=\int_0^{\infty}\mathbb{P}(T^2\geq t)\,dt=\int_0^{\infty}\mathbb{P}(T\geq \sqrt{t})\,dt.$ Therefore
\begin{align*}
    \mathbb{E}_{f_0}\left[|\epsilon_i-\mathbb{E}_{f_0}[\epsilon_i\rvert\bX_i]|^2\right]
    &=\mathbb{E}_{f_0}\left[\big|Y_i-\mathbb{E}_{f_0}[Y_i\rvert\bX_i]\big|^2\right] \\
    &= \mathbb{E}_{f_0}\bigg[ \mathbb{E}_{f_0}\Big[\big|Y_i-\mathbb{E}_{f_0}[Y_i\rvert\bX_i]\big|^2 \, \Big|\, \bX_i\Big]\bigg] \\
    &= \mathbb{E}_{f_0}\bigg[\int_0^\infty\mathbb{P}_{f_0}\Big(\big|Y_i-\mathbb{E}_{f_0}[Y_i\rvert\bX_i]\big|\geq \sqrt{t} \, \Big|\, \bX_i\Big) \, dt\bigg].
\end{align*}
The probability can also be written as
\begin{equation*}
    \begin{aligned}
    &\begin{aligned}
    \int_0^\infty \mathbb{P}_{f_0}\Big(\big|Y_i-\mathbb{E}_{f_0}[Y_i\rvert\bX_i]\big|\geq \sqrt{t} \, \Big|\, \bX_i\Big) \, dt\end{aligned}\\
    &\begin{aligned}=\int_0^{\infty}\mathbb{P}_{f_0}\Bigg(\Bigg|\sum\limits_{\ell=1}^{n}\Bigg(\prod_{r=1}^{d}&K\left(\frac{X'_{\ell,r}-X_{i,r}}{h_n}\right) \\
    &-\int_{[0,1]^d}f_0(\bu)\prod_{r=1}^{d} K\left(\frac{u_{r}-X_{i,r}}{h_n}\right)\, d\bu\Bigg)\Bigg|\geq nh_n^d\sqrt{t} \, \Bigg|\, \bX_i\Bigg)\,dt.\end{aligned}
    \end{aligned}
\end{equation*}    
This is a sum of i.i.d.\ random variables minus their expectation (conditionally on $\bX_i)$. Using that $\|f_0\|_\infty \leq F$ and the kernel $K$ is supported on $[-1,1],$ we get using substitution 
\begin{equation*}
    \begin{aligned}
    \mathbb{E}_{f_0}\left[\prod_{r=1}^{d}K^2\left(\frac{X'_{\ell,r}-X_{i,r}}{h_n}\right)\, \Big| \, \bX_i\right]&\leq F\int_{\mathbb{R}^d}\prod_{r=1}^{d}K^2\left(\frac{u_{r}-X_{i,r}}{h_n}\right) \, d\bu\\
    &=Fh_n^d\int_{\mathbb{R}^d} \prod_{r=1}^{d}K^2(v_r)\, d\bv\\
    &\leq F\|K\|_\infty^{2d} 2^d h_n^d .
    \end{aligned}
\end{equation*}
Applying the bounded variable version of Bernstein's inequality in Proposition \ref{P:BernsteinAlt} (ii) with $v=F\|K\|_\infty^{2d} 2^d h_n^d$ and $b=3\|K\|_{\infty}^d$ (that is, $b/3=\|K\|_{\infty}$),  we get that
\begin{equation*}
    \begin{aligned}
    &\begin{aligned}\int_0^{\infty}\mathbb{P}_{f_0}\Bigg(\Bigg|\sum\limits_{\ell=1}^{n}\Bigg(\prod_{r=1}^{d}&K\left(\frac{X'_{\ell,r}-X_{i,r}}{h_n}\right)
    -\int_{[0,1]^d}f_0(\bu)\prod_{r=1}^{d} K\left(\frac{u_{r}-X_{i,r}}{h_n}\right)\, d\bu\Bigg)\Bigg|\geq nh_n^d\sqrt{t}\Bigg)\,dt\end{aligned}\\
    &\leq \int_0^{\infty}1\wedge 2\exp\left(-\frac{n^2h_n^{2d}t}{2(n\|K\|_{\infty}^{2d}F2^dh_n^d+\|K\|_{\infty}^dnh_n^d\sqrt{t})}\right)\,dt\\
    &=\int_0^{\infty}1\wedge 2\exp\left(-\frac{nh_n^dt}{2(\|K\|_{\infty}^{2d}F2^d+\|K\|_{\infty}^d\sqrt{t})}\right)\,dt\\
    &\overset{(*)}{\leq} F^22^{2d}\|K\|_{\infty}^{2d}+2\int_{F^22^{2d}\|K\|_{\infty}^{2d}}^{\infty}\exp\left(-\frac{nh_n^d\sqrt{t}}{4\|K\|_{\infty}^d}\right)\,dt\\
    &\overset{(**)}{=}F^22^{2d}\|K\|_{\infty}^{2d}+\frac{64\|K\|_{\infty}^{2d}\left(\frac{F2^{d}\|K\|_{\infty}^{d}nh_n^d}{4\|K\|_{\infty}^d}+1\right)\exp\left(-\frac{F2^{d}\|K\|_{\infty}^{d}nh_n^d}{4\|K\|_{\infty}^d}\right)}{n^2h_n^{2d}}\\
    &\overset{(***)}{\leq} F^22^{2d}\|K\|_{\infty}^{2d}+64\|K\|_{\infty}^{2d}(F2^{d}+1)\\
    &\overset{(****)}{\leq} 65F^22^{2d}\|K\|_{\infty}^{2d},
    \end{aligned}
\end{equation*}    
where we used for $(*)$ that $2(\|K\|_{\infty}^{2d}F2^d+\|K\|_{\infty}^d\sqrt{t})\leq 4\|K\|_{\infty}^d\sqrt{t}$ whenever $t\geq F^22^{2d}\|K\|_{\infty}^{2d}.$ For $(**)$ we used that $\int_{b^2}^\infty e^{-a\sqrt{u}} \,du=2\int_b^{\infty}se^{-sa}ds=2(ba+1)e^{-ba}/a^2,$ with $a=nh_n^d/(4\|K\|_{\infty}^d)$ and $b=F2^{d}\|K\|_{\infty}^{d}$. For $(***),$ we used that $nh_n^d\geq \log(n)\geq 1$ and that $0<\exp(-x)\leq 1$ for $x\geq 0.$ For $(****),$ we used that $F2^{d}+1\leq 2F2^{d}\leq F2^{2d}\leq F^22^{2d}.$ The result follows from observing that $\mathbb{E}[c]=c,$ for any real number $c.$
\end{proof}

\begin{Proposition}\label{P:BernsteinAlt}
Given independent random variables $Z_1,\ldots, Z_n$.
\begin{itemize}
    \item[(i)] (moment version) If for some constants $U$ and $v$ the moment bounds $\sum_{i=1}^n \mathbb{E}[|Z_i|^m]\leq \frac{1}{2}m!U^{m-2}v$ hold for all $m=2,3,\ldots,$ then
\begin{equation*}
    \begin{aligned}
    \mathbb{P}\bigg(\bigg|\sum\limits_{i=1}^n(Z_i-\mathbb{E}[Z_i])\bigg|>t\bigg)\leq 2e^{-\frac{t^2}{2v+2Ut}}.
    \end{aligned}
\end{equation*}
    \item[(ii)] (bounded version) If for some constants $b$ and $v,$ the bounds $|Z_i|\leq b$ and $\sum_{i=1}^n\mathbb{E}[|Z_i|^2]\leq v$ hold for all $i=1,\ldots,n,$ then, 
    \begin{equation*}
    \begin{aligned}
    \mathbb{P}\bigg(\bigg|\sum\limits_{i=1}^n(Z_i-\mathbb{E}[Z_i])\bigg|>t\bigg)\leq 2e^{-\frac{t^2}{2v+2bt/3}}.
    \end{aligned}
\end{equation*}  
\end{itemize}
\end{Proposition}
These formulations of Bernstein's inequality are based on Corollary 2.11 and Equation (2.10) in \cite{MR3185193}.

\bibliographystyle{acm}
\bibliography{LibDensityRegression}

\end{document}